\documentclass{amsart}
\usepackage[all]{xy}
\usepackage{pstricks}

\newtheorem{theorem}{Theorem}
\newtheorem{lemma}{Lemma}
\newtheorem{proposition}{Proposition}
\newtheorem{corollary}{Corollary}

\theoremstyle{remark}

\numberwithin{equation}{section}

\newcommand{\internalcomment}[1]{}

\begin{document}

\title[An approximation formula by interpolation]
{An approximation formula for holomorphic functions by interpolation on the ball}

\author{Amadeo Irigoyen}

\address{Scuola Normale Superiore di Pisa - Piazza dei Cavalieri, 7 - 56126 Pisa, Italy}

\email{
\begin{minipage}[t]{5cm}
amadeo.irigoyen@sns.it
\end{minipage}
}

\begin{abstract}

We deal with a problem of the reconstruction
of any holomorphic function $f$ on the unit ball of
$\mathbb{C}^2$ from its restricions on a union of
complex lines. We give an explicit
formula of Lagrange interpolation's type
that is constructed from the knowledge of
$f$ and its derivatives on these lines. 
We prove that this formula approximates any function
when the number of lines increases.
The motivation of 
this problem comes also from possible applications 
in mathematical economics and medical imaging.

\end{abstract}

\maketitle

\tableofcontents

\section{Introduction}

In this paper, we deal with the following problem of reconstruction: 
$f$ being a holomorphic function on a bounded domain
$\Omega\subset\mathbb{C}^m$, we want to reconstruct $f$ from 
its restriction on an analytic subvariety $Z$ of $\Omega$. We
assume that this analytic subvariety is given by
$Z=\{z\in\Omega,\;g(z)=0\}$, where
$g\in\mathcal{O}\left(\overline{\Omega}\right)$.

A natural way is to construct from the restriction
$f_{\{g=0\}}$ a holomorphic function
$\widetilde{f}\in\mathcal{O}\left(\overline{\Omega}\right)$ 
that interpolates $f$, i.e. 
$\widetilde{f}(z)=f(z),\;\forall\,z\in Z$
(see \cite{berndtsson}, \cite{amar}, \cite{henleit}).
Nevertheless, we know that generally 
$\widetilde{f}\neq f$ then we cannot regain $f$. 
This yields to the following questions: 
what will happen if we make larger the analytic
set $Z$ where $f$ and $\widetilde{f}$ coincide ?
Will $\widetilde{f}$ converge to $f$ ? Else is it only 
true for a certain class of functions $f$ ? Is there also
an explicit formula that gives $\widetilde{f}$ or
another that approximates $f$ ? And what could be the precision of 
this approximation ?
\bigskip

A first consequence is that
$f-\widetilde{f}_N$ will vanish on
the set 
$Z_N,\,N\geq1$,
that is increasingly big but this is not sufficient
to deduce that
$f-\widetilde{f}_N\rightarrow0$. Nevertheless,
it is sufficient to prove that
$\widetilde{f}_N$ is uniformly bounded
on any compact subset of $\Omega$. Indeed, by
the Stieltjes-Vitali-Montel theorem,
one can choose a subsequence 
(precisely, choose a subsequence from 
any subsequence) that converges
to a holomorphic $h$ that vanishes
on 
$\bigcup_{N\geq1}Z_N$, then
$h=0$ (for all subsequence) and
$\widetilde{f}_N\rightarrow f$.
\bigskip

We begin with a special case in the unit disc
$D(0,1)\subset\mathbb{C}$. The analytic set $Z$ is given
by a sequence 
$\{\eta_j,\,j\geq1\}\subset D(0,1)$ and the data
of $f_Z,\;f\in\mathcal{O}\left(\overline{D}(0,1)\right)$,
is the sequence
$\{f(\eta_j),\,j\geq1\}$. Now consider, for all
$\eta\in D(0,1)$ the Blaschke function 
$\varphi_{\eta}\in\mathcal{O}\left(\overline{D}(0,1)\right)$
defined as
\begin{eqnarray*}
\varphi_{\eta}(z) 
& := &
\frac{z-\eta}{1-\overline{\eta}z}
\,.
\end{eqnarray*}
One has, for all $N\geq1$ and all
$z\in D(0,1)$,
\begin{eqnarray*}
\prod_{l=1}^N\varphi_{\eta_l}(z)
\frac{1}{2\pi i}
\int_{|\zeta|=1}
\frac{f(\zeta)\,d\zeta}
{\prod_{l=1}^N\varphi_{\eta_l}(\zeta)(\zeta-z)}
& = &
f(z)
-\sum_{l=1}^N
\prod_{j\neq l}
\frac{\varphi_{\eta_j}(z)}{\varphi_{\eta_j}(\eta_l)}
\,\frac{f(\eta_l)}{1-\overline{\eta_l}z}
\,.
\end{eqnarray*}
Since 
$|\varphi_{\eta}(\zeta)|=1$ for all 
$|\zeta|=1$ and 
$|\varphi_{\eta}(z)|<1$ on $D(0,1)$, 
one has,
for all compact subset
$K\subset D(0,1)$,
\begin{eqnarray*}
\left|
\prod_{l=1}^N\varphi_{\eta_l}(z)
\frac{1}{2\pi i}
\int_{|\zeta|=1}
\frac{f(\zeta)\,d\zeta}
{\prod_{l=1}^N\varphi_{\eta_l}(\zeta)(\zeta-z)}
\right|
\;\leq\;
\frac{\sup_{\zeta\in K}|f(\zeta)|}{\varepsilon_K}
(1-\varepsilon_K)^N
\xrightarrow[N\rightarrow\infty]{}0
\,.
\\
\end{eqnarray*}
It follows that $f$ can be approximated
by the above explicit interpolation formula.
Notice that it is of Lagrange interpolation's type (see below).
In addition, we know the precision of this
approximation.
\bigskip

This example gives an idea
for the principal result of this paper. Nevertheless,
in order to generalize this method to
$\Omega\subset\mathbb{C}^m$, we need to find 
functions 
$\varphi_Z$ of Blaschke type such that
$Z=\{\varphi_Z=0\}$, i.e. that satisfie as well
$|\varphi_Z(\zeta)|=1,\,\forall\,\zeta\in\partial\Omega$.
This will be not possible in our case since we will consider
subvarieties that cross $\partial\Omega$. 

Therefore
we begin with a preliminar result 
(section \ref{prelimsection}, proposition \ref{prelim}) 
that gives the essential idea.
Let be 
$f\in\mathcal{O}\left(\overline{\Omega}\right)$
and assume that $Z=\{g(z)=0\}$ where
$g\in\mathcal{O}\left(\overline{\Omega}\right)$.
Then for all $z\in\Omega$,
\begin{eqnarray}\label{prelimdec}
f(z) & = &
Res(f,g)(z)+PV(f,g)(z)
\,,
\end{eqnarray}
where $Res(f,g)$ 
is a holomorphic function that interpolates $f$
on $\{g=0\}$
and is constructed with the residual current of
$1/g$ (resp. $PV(f,g)$ is constructed with the 
principal value current of $1/g$, 
see \cite{colherr}, \cite{herrlieb}).
\bigskip

In all the following, we will deal with
the unit ball of 
$\mathbb{C}^2$,
\begin{eqnarray*}
\mathbb{B}_2 & = & 
\left\{(z_1,z_2)\in\mathbb{C}^2,\,|z_1|^2+|z_2|^2<1\right\}
\,.
\\
\end{eqnarray*}
We also consider for
$Z=\{g(z)=0\}$
a union of lines that cross the origin. Without loss of generality,
one can choose
\begin{eqnarray}\label{defgn}
g_n(z) & = &
z_1^{m_1}\,
\prod_{j=2}^{n-1}(z_1-\eta_jz_2)^{m_j}
\,z_2^{m_n}
\,,
\end{eqnarray}
where 
$n\geq3,\;m_1,\ldots,m_n\in\mathbb{N}$ and
$0<|\eta_2|\leq\cdots\leq|\eta_{n-1}|$
(one has $\eta_1=0$ and by convention
$\eta_n=\infty$).

On the other hand, we specify that, 
for all $f\in\mathcal{O}(\mathbb{B}_2)$ and
all $m_p\geq2$,
the data 
$f_{\{g_n(z)=0\}}$ is defined as
\begin{eqnarray}\label{data1}
f_{\{z_1^{m_1}=0\}},\,
\left(f_{\left\{(z_1-\eta_pz_2)^{m_p}=0\right\}}
\right)_{2\leq p\leq n-1},\,
f_{\{z_2^{m_n}=0\}}
\end{eqnarray}
where
$f_{\left\{(z_1-\eta_pz_2)^{m_p}=0\right\}}$
is defined as
\begin{eqnarray}\label{data2}
& & 
f_{\{z_1=\eta_pz_2\}},
\left\{\frac{\partial f}{\partial z_1},
\frac{\partial f}{\partial z_2}\right\}_{\{z_1=\eta_pz_2\}},
\ldots,
\left\{
\left(\frac{\partial^{m_p-1}f}{\partial z_1^{j_1}\partial z_2^{j_2}}\right)
_{j_1+j_2=m_p-1}
\right\}_{\{z_1=\eta_pz_2\}}
\end{eqnarray}
(as well as for 
$f_{\{z_1^{m_1}=0\}}$ and
$f_{\{z_2^{m_n}=0\}}$).

The problem of interpolation by lines 
is motivated by applications in mathematical economics and medical imaging
where we have to reconstruct any function $F$ with compact support
from knowledge of its Radon transform
$(RF)(\theta^{(p)},s),\;
(\theta^{(p)},s)\in S^{n-1}\times\mathbb{R}$,
on a finite number of directions
$\theta^{(p)},\,p=1\ldots,n$
(see \cite{hensha}).
\bigskip

Before giving the principal result of this paper, 
we need to specify the following notations.

First, 
$W\subset\mathbb{C}^2$ 
being an open set,
consider
$h(t,w),\,(t,w)\in W$,
that is holomorphic with respect to $t$
(resp. continuous with repect to $w$). For all
$\eta_1,\ldots,\eta_n\in\mathbb{C}$
such that
$(\eta_j,\eta_j)\in W,\,
\forall\,j=1,\ldots,n$
and all
$m_1,\ldots,m_n\in\mathbb{N}$,
we set
\begin{eqnarray}\label{lagen}
\mathcal{L}\left(\eta_1^{m_1},\ldots,\eta_n^{m_n};h(t,w)\right)(X)
& := &
\;\;\;\;\;\;\;\;\;\;\;\;\;\;\;\;\;\;\;\;\;\;\;\;\;\;\;\;\;\;\;\;\;\;\;\;\;\;\;\;\;\;\;\;\;\;\;\;\;\;\;\;
\end{eqnarray}
\begin{eqnarray*}
\;\;
& := &
\prod_{j=1}^n
(X-\eta_j)^{m_j}
\sum_{p=1}^n
\frac{1}{(m_p-1)!}
\frac{\partial^{m_p-1}}{\partial t^{m_p-1}}|_{t=\eta_p}
\left(
\frac{h(t,\eta_p)}
{\prod_{j=1,j\neq p}^n(t-\eta_j)^{m_j}}
\right)
\,.
\end{eqnarray*}

In particular, if
$h(t)$ is holomorphic
on $U\subset\mathbb{C}$
and
$\eta_1,\ldots,\eta_n\in U$,
we set
\begin{eqnarray}\label{lag}
\mathcal{L}\left(\eta_1^{m_1},\ldots,\eta_n^{m_n};h(t)\right)(X)
& := &
\;\;\;\;\;\;\;\;\;\;\;\;\;\;\;\;\;\;\;\;\;\;\;\;\;\;\;\;\;\;\;\;\;\;\;\;\;\;\;\;\;\;\;\;\;\;\;\;\;\;\;\;\;\;\;
\end{eqnarray}
\begin{eqnarray*}
\;\;
& := &
\prod_{j=1}^n
(X-\eta_j)^{m_j}
\sum_{p=1}^n
\frac{1}{(m_p-1)!}
\frac{\partial^{m_p-1}}{\partial t^{m_p-1}}|_{t=\eta_p}
\left(
\frac{h(t)}
{\prod_{j=1,j\neq p}^n(t-\eta_j)^{m_j}}
\right)
\,
\\
\end{eqnarray*}
(in the above sums only appear the terms with
$m_p\geq1$; if 
$m_p=0,\,\forall\,p=1,\ldots,n$, we set by convention
$\mathcal{L}(h):=0$).

In addition, if we choose
$\dfrac{h(t)}{X-t}$, 
then formula (\ref{lag}) gives the Lagrange interpolation polynomial
of $h$ on the $\{\eta_p\}_{1\leq p\leq n}$ with derivatives 
at order $\leq m_p-1$ (see
section \ref{lagrange}):

\begin{eqnarray*}
\mathcal{L}\left(\eta_1^{m_1},\ldots,\eta_n^{m_n}\;;\;
\frac{h(t)}{X-t}\right)(X)
& = &
\;\;\;\;\;\;\;\;\;\;\;\;\;\;\;\;\;\;\;\;\;\;\;\;\;\;\;\;\;\;\;\;\;\;\;\;\;\;\;\;\;\;\;\;\;\;\;\;\;\;\;\;\;\;\;\;\;\;\;\;
\;\;\;\;\;\;
\end{eqnarray*}
\begin{eqnarray*}
& = &
\prod_{j=1}^n(X-\eta_j)^{m_j}
\sum_{p=1}^n
\frac{1}{(m_p-1)!}
\frac{\partial^{m_p-1}}{\partial t^{m_p-1}}|_{t=\eta_p}
\left(\frac{h(t)}{(X-t)\prod_{j=1,j\neq p}^n(t-\eta_j)^{m_j}}\right)
\\
& = & \sum_{p=1}^n\;\prod_{j=1,j\neq p}^n
(X-\eta_j)^{m_j}\;
\sum_{s=0}^{m_p-1}(X-\eta_p)^s\,\frac{1}{s!}
\frac{\partial^s}{\partial t^s}|_{t=\eta_p}
\left(\frac{h(t)}{\prod_{j=1,j\neq p}^n(t-\eta_j)^{m_j}}\right)
\,.
\\
\end{eqnarray*}

In the particular case with
$m_p=1,\,\forall\,p=1,\ldots,n$, we get the classical
Lagrange interpolation polynomial
\begin{eqnarray*}
\mathcal{L}\left(\eta_1,\ldots,\eta_n\;;\;
\frac{h(t)}{X-t}\right)(X)
& = &
\sum_{p=1}^n
\;\prod_{j=1,j\neq p}^n
\;\frac{X-\eta_j}{\eta_p-\eta_j}
\;h(\eta_p)
\,.
\end{eqnarray*}
On the other hand, if 
$m_j=0,\;\forall\,j\neq p$, we get
the Taylor limited expansion on $\eta_p$
at order $m_p-1$:
\begin{eqnarray*}
\mathcal{L}\left(\eta_1^0,\ldots,\eta_p^{m_p},\ldots,\eta_n^0\;;\;
\frac{h(t)}{X-t}\right)(X)
& = &
\sum_{s=0}^{m_p-1}(X-\eta_p)^s\,\frac{1}{s!}
\,h^{(s)}(\eta_p)
\,.
\\
\end{eqnarray*}

\bigskip

Next, for all $p=1,\ldots,n$ and $u_p=0,\ldots,m_p-1$, we set
\begin{eqnarray*}\label{Nup}
N_{u_p}
& := & 
u_p+m_{p+1}+\cdots+m_n
\end{eqnarray*}
and
\begin{eqnarray*}\label{N}
N & := & 
m_1+\cdots+m_n\,.
\end{eqnarray*}

For $f\in\mathcal{O}(\mathbb{B}_2)$ and
$f(z)=\sum_{k_1,k_2\geq0}
a_{k_1,k_2}z_1^{k_1}z_2^{k_2}$
its Taylor expansion, we set, for all
$p=1,\ldots,n-1$ and
$u_p=0,\ldots,m_p-1$,
\begin{eqnarray}\label{r0up}
R_{u_p}^0(f;z,t,w)
& := &
\;\;\;\;\;\;\;\;\;\;\;\;\;\;\;\;\;\;\;\;\;\;\;\;\;\;\;\;\;\;\;\;\;\;\;\;\;\;\;\;\;\;\;\;\;\;\;\;\;\;\;\;\;\;\;\;\;\;\;
\;\;\;\;\;\;\;\;\;\;\;
\end{eqnarray}
\begin{eqnarray*}
\;\;\;\;\;\;\;\;\;
& := &
\frac{1+|w|^2\eta_p/t}{1+|w|^2}
\sum_{k_1+k_2\geq N_{u_p}}
a_{k_1,k_2}\,t^{k_1}
\left(
\frac{z_2+|w|^2z_1/t}{1+|w|^2}
\right)^{k_1+k_2-N_{u_p}}
z_2^{N_{u_p}}
\end{eqnarray*}
and
\begin{eqnarray}\label{r0N}
R_N^0(f;z,t,w)
& := &
\;\;\;\;\;\;\;\;\;\;\;\;\;\;\;\;\;\;\;\;\;\;\;\;\;\;\;\;\;\;\;\;\;\;\;\;\;\;\;\;\;\;\;\;\;\;\;\;\;\;\;\;\;\;\;\;\;\;\;\;
\;\;\;\;\;\;\;\;\;\;\;\;\;\;
\end{eqnarray}
\begin{eqnarray*}
\;\;\;\;\;\;\;\;\;\;\;\;
& := &
(z_1/z_2)^{m_1}
\sum_{k_1+k_2\geq N}
a_{k_1,k_2}\,t^{k_1-m_1}
\left(
\frac{z_2+|w|^2z_1/t}{1+|w|^2}
\right)^{k_1+k_2-N+1}
z_2^{N-1}
\;.
\\\nonumber
\end{eqnarray*}

For all $u_1=0,\ldots,m_1-1$,
\begin{eqnarray}\label{r1u1}
R_{u_1}^1(f;z,t)
& := &
\sum_{k_1\leq u_1,k_2\geq N_{u_1}}
a_{k_1,k_2}\,t^{k_1}
z_2^{k_1+k_2}
\end{eqnarray}
and
\begin{eqnarray}\label{r1N}
R_N^1(f;z,t)
& := &
(z_1/z_2)^{m_1}
\sum_{k_1\leq m_1-1,k_2\geq N-k_1}
a_{k_1,k_2}
\,t^{k_1-m_1}z_2^{k_1+k_2}
\;.
\\\nonumber
\end{eqnarray}

Lastly, for all
$p=2,\ldots,n-1$ and
$u_p=0,\ldots,m_p-1$, we set
\begin{eqnarray}\label{r2up}
& &
R_{u_p}^2(f;z,t)
\;:=\;
\eta_p
\sum_{k_2\leq m_n-1,k_1\geq N_{u_p}-k_2}
a_{k_1,k_2}\,t^{N_{u_p}-1-k_2}
z_1^{k_1+k_2-N_{u_p}}\,z_2^{N_{u_p}}
\end{eqnarray}
and
\begin{eqnarray}\label{r2N}
R_N^2(f;z,t)
& := &
\;\;\;\;\;\;\;\;\;\;\;\;\;\;\;\;\;\;\;\;\;\;\;\;\;\;\;\;\;\;\;\;\;\;\;\;\;\;\;\;\;\;\;\;\;\;\;\;\;\;\;\;\;\;\;\;\;\;\;\;
\;\;\;\;\;\;\;\;\;\;\;\;\;\;\;
\end{eqnarray}
\begin{eqnarray*}
\;\;\;\;\;\;
& := &
(z_1/z_2)^{m_1}
\sum_{k_2\leq m_n-1,k_1\geq N-k_2}
a_{k_1,k_2}
\,t^{N-m_1-1-k_2}
z_1^{k_1+k_2-N+1}z_2^{N-1}
\;.
\end{eqnarray*}
All these functions are well-defined since the
above series are absolutely convergent for all
$z\in\mathbb{B}_2$ and all
$(t,w)$ in a neighborhood of 
$(\eta_p,\eta_p),\,p=1\ldots,n-1$
(see lemma~\ref{serie}).

We can finally set
\begin{eqnarray}\label{formula}
\mathcal{G}
\left(
\eta_1^{m_1},\ldots,\eta_n^{m_n};f
\right)(z)
& := &
\;\;\;\;\;\;\;\;\;\;\;\;\;\;\;\;\;\;\;\;\;\;\;\;\;\;\;\;\;\;\;\;\;\;\;\;\;\;\;\;\;\;\;\;\;\;\;\;\;\;\;\;\;\;\;\;\;\;
\end{eqnarray}

\begin{eqnarray*}
\;\;\;\;\;\;\;\;\;\;
& := &
\sum_{u_1=0}^{m_1-1}
(z_1/z_2)^{u_1}
\mathcal{L}
\left(
\eta_2^{m_2},\ldots,\eta_{n-1}^{m_{n-1}};
\frac{R_{u_1}^0(f;z,t,w)
-R_{u_1}^1(f;z,t)}
{t^{u_1+1}}
\right)
(z_1/z_2)
\\
\\
&  &
+\;
\sum_{p=2}^{n-1}
\sum_{u_p=0}^{m_p-1}
\mathcal{L}
\left(
\eta_p^{u_p+1},\ldots,\eta_{n-1}^{m_{n-1}};
\frac{R_{u_p}^0(f;z,t,w)
-R_{u_p}^2(f;z,t)}
{z_1/z_2-\eta_p}
\right)(z_1/z_2)
\\
\\
&  &
+\;
\mathcal{L}
\left(
0^{m_n};\frac{f(z_1,t)}{z_2-t}
\right)(z_2)
\\
\\
&  &
-\;
\mathcal{L}
\left(
\eta_2^{m_2},\ldots,\eta_{n-1}^{m_{n-1}};
\frac{R_N^0(f;z,t,w)
-R_N^1(f;z,t)-R_N^2(f;z,t)}
{z_1/z_2-t}
\right)(z_1/z_2)
\;.
\\
\end{eqnarray*}
\bigskip

Now we can give the principal result of this paper.
\medskip

\begin{theorem}\label{theorem}

Let be 
$f\in\mathcal{O}\left(\mathbb{B}_2\right)$
and 
$f(z)=\sum_{k_1,k_2\geq0}
a_{k_1,k_2}\,z_1^{k_1}\,z_2^{k_2}$
its Taylor expansion. Then 
for all $z\in\mathbb{B}_2$, we have

\begin{eqnarray*}
f(z)
& = &
\mathcal{G}
\left(
\eta_1^{m_1},\ldots,\eta_n^{m_n};f
\right)(z)
\;+\;
\sum_{k_1+k_2\geq N,\,k_1\geq m_1,\,k_2\geq m_n}
a_{k_1,k_2}\,z_1^{k_1}\,z_2^{k_2}
\;.
\\
\end{eqnarray*}

\end{theorem}

\bigskip

The proof of this result will consist on
the calculation of 
$Res(f,g_n)$ (section \ref{interpolers}, 
proposition \ref{interpoler})
and $PV(f,g_n)$ (section \ref{remainders}, 
proposition \ref{remainder})
from (\ref{prelimdec}). As we will see,
$Res(f,g_n)$ is a holomorphic function that
is constructed from the data
$f_{\{g_n(z)=0\}}$ and that interpolates $f$
on $\{g_n(z)=0\}$ (i.e. it coincides with $f$
with derivatives on the lines
$\{z_1=\eta_pz_2\}$),
while $PV(f,g_n)$ is the sum
of a part that is constructed from
$f_{\{g_n(z)=0\}}$ and
of the above remainder that appears in the statement 
of the theorem. 

It follows that any holomorphic function can be approximated
on the unit ball by the explicit formula 
$\mathcal{G}
\left(
\eta_1^{m_1},\ldots,\eta_n^{m_n};f
\right)$
that is constructed
with the restrictions of $f$ and its derivatives on the
complex lines $\{z_1-\eta_jz_2=0\},\,j=1,\ldots,n-1$
and $\{z_2=0\}$ (see \cite{alabut}, \cite{loganshepp} for
analogous results). 

We see in addition that it is sufficient to assume
that $f\in\mathcal{O}(\mathbb{B}_2)$. In particular,
it is not necessary to assume that $f$ is holomorphic
on a neighborhood of the closed unit ball (condition
that we need
in order to construct interpolation functions
with residual currents).
\bigskip

Although 
$\mathcal{G}
\left(
\eta_1^{m_1},\ldots,\eta_n^{m_n};f
\right)$
is an explicit formula 
(see section \ref{proof}, lemma \ref{explicit}),
it is quite technical in the general case. Nevertheless,
one can entirely specify it in the special but not less natural
case of single lines (i.e.
$m_2=\cdots=m_{n-1}=1$
and 
$m_1=m_n=0$) where theorem \ref{theorem} becomes:
\begin{eqnarray*}
f(z) & = &
\mathcal{G}(\eta_2,\ldots,\eta_{n-1};f)(z)
+\sum_{k_1+k_2\geq N}
a_{k_1,k_2}z_1^{k_1}z_2^{k_2}
\;,
\end{eqnarray*}
with
\begin{eqnarray}\label{singlintro}
\mathcal{G}(\eta_2,\ldots,\eta_{n-1};f)(z)
& := &
\mathcal{G}
\left(\eta_1^0,\eta_2^1,\ldots,\eta_{n-1}^1,\eta_n^0;f\right)(z)
\;\;\;\;\;\;\;\;\;\;\;\;\;
\end{eqnarray}
\begin{eqnarray*}
& = &
\sum_{p=2}^{n-1}
\prod_{j=p+1}^{n-1}
(z_1-\eta_jz_2)
\;\times
\\
& &
\times\;
\sum_{q=p}^{n-1}
\frac{1+\eta_p\overline{\eta_q}}{1+|\eta_q|^2}
\frac{1}
{\prod_{j=p,j\neq q}^{n-1}(\eta_q-\eta_j)}
\sum_{l\geq N-p+1}
\left(
\frac{z_2+\overline{\eta_q}z_1}{1+|\eta_q|^2}
\right)^{l-(N-p+1)}
\frac{1}{l!}
\frac{\partial^l}{\partial v^l}|_{v=0}
[f(\eta_qv,v)]
\\
\\
& - &
\sum_{p=2}^{n-1}
\prod_{j=2,j\neq p}^{n-1}
\frac{z_1-\eta_jz_2}{\eta_p-\eta_j}
\sum_{l\geq N}
\left(
\frac{z_2+\overline{\eta_p}z_1}{1+|\eta_p|^2}
\right)^{l-N+1}
\frac{1}{l!}
\frac{\partial^l}{\partial v^l}|_{v=0}
[f(\eta_pv,v)]
\;.
\\
\end{eqnarray*}

\bigskip

We finish with giving as a consequence
of the theorem some information for
the precision of this approximation
(section \ref{proof}, corollary~\ref{vitessecv}).
$\mathcal{F}$ being a compact subset of
$\mathcal{O}(\mathbb{B}_2)$ (i.e.
a subset of holomorphic functions
that is uniformly bounded on any
compact subset of the ball), one has for all
compact subset
$K\subset\mathbb{B}_2$

\begin{eqnarray}\label{corollarintro}
\sup_{f\in\mathcal{F}}
\sup_{z\in K}
\left|
f(z)-
\mathcal{G}
\left(
\eta_1^{m_1},\ldots,\eta_n^{m_n};f
\right)(z)
\right|
& \leq &
C(K,\mathcal{F})
(1-\varepsilon_K)^N
\;,
\\\nonumber
\end{eqnarray}
where
$C(K,\mathcal{F})$ (resp. $\varepsilon_K$)
depends on $K$ and $\mathcal{F}$
(resp. $K$).
\bigskip

In conclusion some questions follow. 
Can the above precision be made better ?
Or can it be made better with another 
choice of formula ? Since
$\mathcal{G}(\eta_1^{m_1},\ldots\eta_n^{m_n};f)$ 
is linear with respect to the data $f_{\{g_n=0\}}$,
theorem \ref{theorem} can also be interpreted
by the approximation of all $2$-variable
holomorphic function $f$ by a fixed linear superposition
of functions $f_j$ of fewer variables
(that depend on $f$). We know
that one cannot get such an exact 
representation of all function $f$
(see \cite{vitushkin}, \cite{vithen}). 
Therefore we would like to
get some lower bounds for the approximation of any
compact subset of holomorphic functions by such a given family.
We are also motivated in dealing with the approximation by
a nonlinear family with respect to the data
$f_{\{g_n=0\}}$ and getting lower bounds
that only depend on the 
$\varepsilon$-entropy of
$\mathcal{O}(\mathbb{B}_2)$ 
(see \cite{kolmogorov}, \cite{amadeo1}, \cite{amadeo2}).

Other questions follow: can we extend theorem \ref{theorem}
to the case of general analytic subvarieties (and not only
for complex lines) ? Is the assertion still true
with a general domain (and not only the ball) ?
Finally, can we get analogous results
in $\mathbb{C}^m$ ?
\bigskip

I would like to thank G. Henkin for the interesting discussions
about this problem.
\bigskip

\section{A preliminar formula in the general case}\label{prelimsection}

\subsection{General case}

\bigskip

We recall for
$\zeta\in\mathbb{C}^m$
\begin{eqnarray*}
\begin{cases}
\omega'(\zeta)=\;
\sum_{k=1}^m(-1)^{k-1}\zeta_k
\bigwedge_{j=1,j\neq k}^md\zeta_k\,,\\
\omega(\zeta)\;=\;
\bigwedge_{j=1}^md\zeta_j
\end{cases}
\end{eqnarray*}
(as well as
$d\omega'(\zeta)=m\omega(\zeta)$).
The orientation on $\mathbb{C}^m$ is such that
\begin{eqnarray*}
\left(\frac{1}{i}\right)^{m^2}
\int_{\zeta\in\mathbb{C}^m}
\omega(\overline{\zeta})\wedge\omega(\zeta)
& > &
0\,.
\end{eqnarray*}
\bigskip

Now we consider a bounded convex domain in $\mathbb{C}^m$
\begin{eqnarray*}
\Omega & = & \{z\in\mathbb{C}^m,\;\rho(z)<0\}
\end{eqnarray*}
with smooth boundary
$\partial\Omega=\{z\in\mathbb{C}^m,\,\rho(z)=0\}$ and whose orientation satisfies the Stokes formula on
$\overline{\Omega}$ (this orientation is also chosen such that
$\frac{1}{i^{m^2}}
\int_{\zeta\in\partial\Omega}\omega'(\overline{\zeta})\wedge\omega(\zeta)>0$).

Let be
$f,\,g\in\mathcal{O}(\overline{\Omega})$ 
(i.e. 
$\exists\;U\supset\overline{\Omega}$ open set such that
$f,\,g\in\mathcal{O}(U)$ )
with $g\neq0$ and
$P=(P_1,\ldots,P_m)\in\mathcal{O}^m\left(\overline{\Omega}\times\overline{\Omega}\right)$
that satisfies: for all
$(\zeta,z)\in\overline{\Omega}\times\overline{\Omega}$,
\begin{eqnarray}
g(\zeta)-g(z) & = &
<P(\zeta,z),\zeta-z>\\\nonumber
& = &
\sum_{j=1}^m\,P_j(\zeta,z)\,(\zeta_j-z_j)\,.
\end{eqnarray}
Lastly, we set
\begin{eqnarray}
\varphi(\zeta,z) & = & 
<\frac{\partial\rho}{\partial z}(\zeta,z),\zeta-z>\,.
\end{eqnarray}
Now we can give the following preliminar result that 
gives the essential method to prove theorem \ref{theorem}.

\bigskip

\begin{proposition}\label{prelim}

For all $z\in\Omega$, we have
\begin{eqnarray}\label{prelimgen}
\;\;\;\;\;\;
& &
(-1)^{m(m-1)/2}
\;f(z)\;=\;
\lim_{\varepsilon\rightarrow0}
\,g(z)\,\frac{
(m-1)!}{(2\pi i)^m}
\int_{\{\zeta\in\partial\Omega,\,|g(\zeta)|>\varepsilon\}}
\frac{f(\zeta)\;
\omega'\left(\frac{\partial\rho}{\partial z}\right)
\wedge\omega(\zeta)}
{g(\zeta)\;\varphi(\zeta,z)^m}
\end{eqnarray}
\begin{eqnarray*}
& &
+\;
\lim_{\varepsilon\rightarrow0}
\frac{(m-2)!}{(2\pi i)^m}
\int_{\{\zeta\in\partial\Omega,|g(\zeta)|=\varepsilon\}}
\frac{f(\zeta)
\sum_{1\leq k<l\leq m}
(-1)^{k+l-1}
\left(\frac{\partial\rho}{\partial z_k}P_l-\frac{\partial\rho}{\partial z_l}P_k\right)
\bigwedge_{j\neq k,l}
d\left(\frac{\partial\rho}{\partial z_j}\right)
\wedge\omega(\zeta)}
{g(\zeta)\;\varphi(\zeta,z)^{m-1}}
\,.
\end{eqnarray*}

\end{proposition}

\bigskip

\begin{proof}

For all $\varepsilon>0$ and $z\in\Omega$, we consider the following
differential form 
\begin{eqnarray*}
\psi(\zeta,\lambda) & = & 
\frac{(m-1)!}{(2\pi i)^m}\;
f(\zeta)\;\omega'\left(\frac{\lambda}{\varphi(\zeta,z)}
\frac{\partial\rho}{\partial z}(\zeta)
+(1-\lambda)\frac{P(\zeta,z)}{g(\zeta)}\right)
\wedge\omega(\zeta)
\end{eqnarray*}
that is defined on a neighborhood of
\begin{eqnarray*}
\Sigma_{\varepsilon} & = & 
\{\zeta\in\partial\Omega,\;|g(\zeta)|\geq\varepsilon\}\times[0,1]\,,
\end{eqnarray*}
with induced orientation from the one of 
$\Sigma_0=\partial\Omega\times[0,1]$ that satisfies
for all differential form $\chi_1(\zeta)$ and all function
$\chi_2(\lambda)$
\begin{eqnarray*}
\int_{\{(\zeta,\lambda)\in\Sigma_0\}}
\chi_1(\zeta)\wedge\chi_2(\lambda)d\lambda
& = &
\int_{\{\zeta\in\partial\Omega\}}
\chi_1(\zeta)\,\times\,
\int_0^1\chi_2(\lambda)d\lambda\,.
\end{eqnarray*}

The application of the Stokes formula gives
\begin{eqnarray}\label{stokes}
\int_{\partial \Sigma_{\varepsilon}}\psi(\zeta,\lambda) & = & 
\int_{\Sigma_{\varepsilon}}d\psi(\zeta,\lambda)\,,
\end{eqnarray}
with the associate orientation of
$\partial\Sigma_{\varepsilon}$ that is specified in the following lemma.

\begin{lemma}\label{orientation}

We have
\begin{eqnarray*}
\partial\Sigma_{\varepsilon}
& = &
-\,\{\zeta\in\partial\Omega,\,|g(\zeta)|=\varepsilon\}\times[0,1]
-\,\{\zeta\in\partial\Omega,\,|g(\zeta)|\geq\varepsilon\}\times\{1\}
+\,\{\zeta\in\partial\Omega,\,|g(\zeta)|\geq\varepsilon\}\times\{0\}\,.
\end{eqnarray*}
It follows that

\begin{eqnarray*}
\int_{\partial\Sigma_{\varepsilon}}
\psi(\zeta,\lambda)
& = &
-\int_{\{|g(\zeta)|=\varepsilon\}\times[0,1]}
\psi(\zeta,\lambda)
-\int_{\{|g(\zeta)|>\varepsilon\}}
\psi(\zeta,1)
+\int_{\{|g(\zeta)|>\varepsilon\}}
\psi(\zeta,0)\,,\\
\end{eqnarray*}
where the orientation of
$\{\zeta\in\partial\Omega,\,|g(\zeta)|=\varepsilon\}\times[0,1]$
is defined, for all differential form 
$\chi_1(\zeta)$ and all function $\chi_2(\lambda)$, as
\begin{eqnarray*}
\int_{\{\zeta\in\partial\Omega,\,|g(\zeta)|=\varepsilon\}\times[0,1]}
\chi_1(\zeta)\wedge\chi_2(\lambda)d\lambda
& = &
\int_{\{|g(\zeta)|=\varepsilon\}}
\chi_1(\zeta)\,\times\,
\int_0^1\chi_2(\lambda)d\lambda\,,
\end{eqnarray*}
and the orientation of
$\{\zeta\in\partial\Omega,\,|g(\zeta)|=\varepsilon\}$ 
is the one that satisfies the Stokes formula on
$\{\zeta\in\partial\Omega,\,|g(\zeta)|\leq\varepsilon\}$.

\end{lemma}

\begin{proof}

First notice that, for all small enough $\varepsilon>0$,
$\{\zeta\in\partial\Omega,\,|g(\zeta)|>\varepsilon\}$
(resp. 
$\{\zeta\in\partial\Omega,\,|g(\zeta)|=\varepsilon\}$)
is a $(2m-1)$-dimensional (resp. $(2m-2)$-dimensional)
submanifold.

Now consider $\chi(\zeta,\lambda)$ a differential form. It can be written as
\begin{eqnarray*}
\chi(\zeta,\lambda)
& = &
\chi_0(\zeta,\lambda)
+\chi_1(\zeta,\lambda)\wedge d\lambda
\end{eqnarray*}
where $\chi_0,\,\chi_1$ are differential forms of degree zero
with respect to $\lambda$.

For $j=0,1$, one has
\begin{eqnarray}\label{orient1}
\int_{\{|g(\zeta)|>\varepsilon\}\times\{j\}}
\chi(\zeta,\lambda)
& = &
\int_{\{|g(\zeta)|>\varepsilon\}}
\chi_0(\zeta,j)
=\int_{\{|g(\zeta)|>\varepsilon\}}
\chi_{0,2m-1}(\zeta,j)\,,
\end{eqnarray}
where
$\chi_{0,2m-1}(\zeta,j):=
\sum_{k+l=2m-1}\sum_{|K|=k,|L|=l}
\chi_{0,K,L}(\zeta,\lambda)d\zeta_K\wedge d\overline{\zeta}_L$
is the $(2m-1)$-homogeneous part of $\chi_0$ with respect to $\zeta$.

Similarly,
\begin{eqnarray}\label{orient2}
\int_{\{|g(\zeta)|=\varepsilon\}\times[0,1]}
\chi(\zeta,\lambda)
& = &
\int_{\{|g(\zeta)|=\varepsilon\}\times[0,1]}
\chi_{1,2m-2}(\zeta,\lambda)\wedge d\lambda\\\nonumber
& = &
\int_0^1d\lambda
\int_{\{|g(\zeta)|=\varepsilon\}}
\chi_{1,2m-2}(\zeta,\lambda)\,.
\end{eqnarray}

On the other hand,
\begin{eqnarray*}
\int_{\Sigma_{\varepsilon}}
d\chi(\zeta,\lambda)
& = &
\int_{\{|g(\zeta)|>\varepsilon\}\times[0,1]}
d_{\lambda}(\chi_{0,2m-1}(\zeta,\lambda))
+d_{\zeta}(\chi_{1,2m-2}(\zeta,\lambda))\wedge d\lambda
\\
& = &
\int_{\{|g(\zeta)|>\varepsilon\}\times[0,1]}
\left(
(-1)^{2m-1}
\frac{\partial\chi_{0,2m-1}}{\partial\lambda}(\zeta,\lambda)
+d_{\zeta}(\chi_{1,2m-2}(\zeta,\lambda))
\right)
\wedge d\lambda\\
& = &
-\,\int_{\{|g(\zeta)|>\varepsilon\}}
\int_0^1d\lambda\,
\frac{\partial\chi_{0,2m-1}}{\partial\lambda}(\zeta,\lambda)
+\int_0^1d\lambda
\int_{\{|g(\zeta)|>\varepsilon\}}
d_{\zeta}(\chi_{1,2m-2}(\zeta,\lambda))
\\
& = &
-\int_{\{|g(\zeta)|>\varepsilon\}}
\left(
\chi_{0,2m-1}(\zeta,1)
-\chi_{0,2m-1}(\zeta,0)
\right)
-\int_0^1d\lambda
\int_{\{|g(\zeta)|=\varepsilon\}}
\chi_{1,2m-2}(\zeta,\lambda)\,.
\\
\end{eqnarray*}

The lemma follows by (\ref{stokes}), 
(\ref{orient1}) and (\ref{orient2}).

\end{proof}

Now we have
$\psi(\zeta,0)=0$ 
since $P(\zeta,z)$ and $g(\zeta)$ are holomorphic. 

Next, we claim that
\begin{eqnarray}\label{prelim1}
\psi(\zeta,1)
& = &
\frac{(m-1)!}{(2\pi i)^m}\;f(\zeta)\;
\frac{\omega'\left(\frac{\partial\rho}{\partial z}(\zeta)\right)
\wedge\omega(\zeta)}
{\varphi(\zeta,z)^m}\,.
\\\nonumber
\end{eqnarray}
Indeed, we have (since
$d(\frac{1}{\varphi})\wedge d(\frac{1}{\varphi})=0$)
\begin{eqnarray*}
\omega'\left(\frac{1}{\varphi}\frac{\partial\rho}{\partial z}(\zeta)\right)
& = &
\sum_{k=1}^m(-1)^{k-1}\frac{1}{\varphi}\frac{\partial\rho}{\partial z_k}
\bigwedge_{j\neq k}
\left(
\frac{1}{\varphi}d\left(\frac{\partial\rho}{\partial z_j}\right)
+\frac{\partial\rho}{\partial z_j}d\left(\frac{1}{\varphi}\right)
\right)\\
& = &
\sum_{k=1}^m(-1)^{k-1}\frac{1}{\varphi}\frac{\partial\rho}{\partial z_k}
\bigwedge_{j\neq k}
\frac{1}{\varphi}d\left(\frac{\partial\rho}{\partial z_j}\right)
\\
& &
+\;
\sum_{k=1}^m(-1)^{k-1}\frac{1}{\varphi}\frac{\partial\rho}{\partial z_k}
\sum_{l\neq k}
\bigwedge_{j<l,j\neq k}
\frac{1}{\varphi}d\left(\frac{\partial\rho}{\partial z_j}\right)
\wedge\frac{\partial\rho}{\partial z_l}
d\left(\frac{1}{\varphi}\right)\wedge
\bigwedge_{j>l,j\neq k}
\frac{1}{\varphi}d\left(\frac{\partial\rho}{\partial z_j}\right)
\\
\end{eqnarray*}
and the last sum is
\begin{eqnarray*}
& = &
\sum_{1\leq l<k\leq m}
(-1)^{k-1}\frac{1}{\varphi}\frac{\partial\rho}{\partial z_k}
(-1)^{l-1}\frac{\partial\rho}{\partial z_l}
d\left(\frac{1}{\varphi}\right)\wedge
\bigwedge_{j\neq k,l}
\frac{1}{\varphi}d\left(\frac{\partial\rho}{\partial z_j}\right)
\\
& & 
\;+\sum_{1\leq k<l\leq m}
(-1)^{k-1}\frac{1}{\varphi}\frac{\partial\rho}{\partial z_k}
(-1)^{l-2}\frac{\partial\rho}{\partial z_l}
d\left(\frac{1}{\varphi}\right)\wedge
\bigwedge_{j\neq k,l}
\frac{1}{\varphi}d\left(\frac{\partial\rho}{\partial z_j}\right)
\\
\\
& = &
\sum_{1\leq l<k\leq m}
(-1)^{k+l}
\frac{1}{\varphi}\frac{\partial\rho}{\partial z_k}
\frac{\partial\rho}{\partial z_l}
d\left(\frac{1}{\varphi}\right)
\wedge
\bigwedge_{j\neq k,l}
\frac{1}{\varphi}d\left(\frac{\partial\rho}{\partial z_j}\right)
\\
& & 
-\;
\sum_{1\leq l<k\leq m}
(-1)^{k+l}
\frac{1}{\varphi}\frac{\partial\rho}{\partial z_k}
\frac{\partial\rho}{\partial z_l}
d\left(\frac{1}{\varphi}\right)
\wedge
\bigwedge_{j\neq k,l}
\frac{1}{\varphi}d\left(\frac{\partial\rho}{\partial z_j}\right)
\\
\\
& = &
0\,,
\end{eqnarray*}
and this proves (\ref{prelim1}). It follows that
\begin{eqnarray*}
\int_{\{|g(\zeta)|>\varepsilon}
\psi(\zeta,1)
& = &
\frac{(m-1)!}{(2\pi i)^m}
\int_{\{|g(\zeta)|>\varepsilon}
f(\zeta)\;
\frac{\omega'\left(\frac{\partial\rho}{\partial z}(\zeta)\right)
\wedge\omega(\zeta)}
{\varphi(\zeta,z)^m}\,
\end{eqnarray*}
then
\begin{eqnarray}\label{prelimf}
\lim_{\varepsilon\rightarrow0}
\int_{\{|g(\zeta)|>\varepsilon\}}
\psi(\zeta,1)
& = &
\frac{(m-1)!}{(2\pi i)^m}
\int_{\zeta\in\partial\Omega}
f(\zeta)\;
\frac{\omega'\left(\frac{\partial\rho}{\partial z}(\zeta)\right)
\wedge\omega(\zeta)}
{\varphi(\zeta,z)^m}\\\nonumber
\\\nonumber
& = &
(-1)^{m(m-1)/2}\;f(z)\,,
\end{eqnarray}
by the Cauchy-Fantappie formula (see \cite{scv}).
\bigskip

Now we want to specify $\psi(\zeta,\lambda)$ on
$\{|g(\zeta)|=\varepsilon\}\times[0,1]$: 
\begin{eqnarray*}
\omega'\left(
\frac{P}{g}
+\lambda
\left(\frac{1}{\varphi}\frac{\partial\rho}{\partial z}-\frac{P}{g}\right)
\right)
\wedge\omega(\zeta)
& = &
\;\;\;\;\;\;\;\;\;\;\;\;\;\;\;\;\;\;\;\;\;\;\;\;\;\;\;\;\;\;\;\;\;\;\;\;\;\;\;\;\;\;\;\;\;\;\;\;\;\;\;\;\;\;\;\;\;\;\;
\;\;\;\;\;\;\;\;\;\;\;
\end{eqnarray*}
\begin{eqnarray*}
& = &
\sum_{k=1}^m
(-1)^{k-1}
\left(
\frac{P_k}{g}
+\lambda
\left(\frac{1}{\varphi}\frac{\partial\rho}{\partial z_k}-\frac{P_k}{g}\right)
\right)
\bigwedge_{j\neq k}
\left(
\left(\frac{1}{\varphi}\frac{\partial\rho}{\partial z_j}-\frac{P_j}{g}\right)
d\lambda
+\lambda
d\left(\frac{1}{\varphi}\frac{\partial\rho}{\partial z_j}\right)
\right)
\wedge\omega(\zeta)
\\
& = &
\sum_{k<l}
(-1)^{k-1}
\left(
\frac{P_k}{g}
+\lambda
\left(\frac{1}{\varphi}\frac{\partial\rho}{\partial z_k}-\frac{P_k}{g}\right)
\right)
(-1)^{2m-l}
\left(\frac{1}{\varphi}\frac{\partial\rho}{\partial z_l}-\frac{P_l}{g}\right)
\bigwedge_{j\neq k,l}
\left(
\lambda
d\left(\frac{1}{\varphi}\frac{\partial\rho}{\partial z_j}\right)
\right)
\wedge\omega(\zeta)
\wedge d\lambda
\\
& + &
\sum_{k>l}
(-1)^{k-1}
\left(
\frac{P_k}{g}
+\lambda
\left(\frac{1}{\varphi}\frac{\partial\rho}{\partial z_k}-\frac{P_k}{g}\right)
\right)
(-1)^{2m-l-1}
\left(\frac{1}{\varphi}\frac{\partial\rho}{\partial z_l}-\frac{P_l}{g}\right)
\bigwedge_{j\neq k,l}
\left(
\lambda
d\left(\frac{1}{\varphi}\frac{\partial\rho}{\partial z_j}\right)
\right)
\wedge\omega(\zeta)
\wedge d\lambda
\\
\\
& = &
\sum_{k<l}
(-1)^{k+l}
\left(
\frac{1}{\varphi}\frac{\partial\rho}{\partial z_k}
\frac{P_l}{g}
-\frac{1}{\varphi}\frac{\partial\rho}{\partial z_l}
\frac{P_k}{g}
\right)
\bigwedge_{j\neq k,l}
\left(
\frac{1}{\varphi}
d\left(\frac{\partial\rho}{\partial z_j}\right)
+\frac{\partial\rho}{\partial z_j}
d\left(\frac{1}{\varphi}\right)
\right)
\wedge\omega(\zeta)
\wedge\lambda^{m-2}d\lambda
\,.
\\
\end{eqnarray*}
Similarly, we see that, for all
$1\leq k<l\leq m$,
\begin{eqnarray*}
\bigwedge_{j\neq k,l}
\left(
\frac{1}{\varphi}
d\left(\frac{\partial\rho}{\partial z_j}\right)
+\frac{\partial\rho}{\partial z_j}
d\left(\frac{1}{\varphi}\right)
\right)
& = &
\;\;\;\;\;\;\;\;\;\;\;\;\;\;\;\;\;\;\;\;\;\;\;\;\;\;\;\;\;\;\;\;\;\;\;\;\;\;\;\;\;\;\;\;\;\;\;\;\;\;\;\;\;\;
\;\;\;\;\;\;\;\;\;\;\;\;\;\;\;\;\;\;
\end{eqnarray*}
\begin{eqnarray*}
& = &
\frac{1}{\varphi^{m-2}}
\bigwedge_{j\neq k,l}
d\left(\frac{\partial\rho}{\partial z_j}\right)
\\
& &
+\;
\frac{1}{\varphi^{m-3}}
d\left(\frac{1}{\varphi}\right)
\wedge
\left[
\sum_{u<k<l}(-1)^{u-1}
+\sum_{k<u<l}(-1)^{u-2}
+\sum_{k<l<u}(-1)^{u-3}
\right]
\frac{\partial\rho}{\partial z_u}
\bigwedge_{j\neq k,l,u}
d\left(
\frac{\partial\rho}{\partial z_j}
\right)\,,
\\
\end{eqnarray*}
and
\begin{eqnarray*}
\sum_{k<l}(-1)^{k+l}
\left(
\frac{\partial\rho}{\partial z_k}
\frac{P_l}{g}
-\frac{\partial\rho}{\partial z_l}
\frac{P_k}{g}
\right)
\left[
\sum_{u<k<l}(-1)^{u-1}
+\sum_{k<u<l}(-1)^{u-2}
+\sum_{k<l<u}(-1)^{u-3}
\right]
\frac{\partial\rho}{\partial z_u}
\bigwedge_{j\neq k,l,u}
d\left(
\frac{\partial\rho}{\partial z_j}
\right)
=
\end{eqnarray*}
\begin{eqnarray*}
& = &
\sum_{u<k<l}
(-1)^{k+l+u}
\bigwedge_{j\neq k,l,u}
d\left(
\frac{\partial\rho}{\partial z_j}
\right)
\times
\\
& & 
\times
\left[
-
\left(
\frac{\partial\rho}{\partial z_k}
\frac{P_l}{g}
-\frac{\partial\rho}{\partial z_l}
\frac{P_k}{g}
\right)
\frac{\partial\rho}{\partial z_u}
+
\left(
\frac{\partial\rho}{\partial z_u}
\frac{P_l}{g}
-\frac{\partial\rho}{\partial z_l}
\frac{P_u}{g}
\right)
\frac{\partial\rho}{\partial z_k}
-
\left(
\frac{\partial\rho}{\partial z_u}
\frac{P_k}{g}
-\frac{\partial\rho}{\partial z_k}
\frac{P_u}{g}
\right)
\frac{\partial\rho}{\partial z_l}
\right]
\\
\\
& = & 
0\,.
\\
\end{eqnarray*}
It follows that
\begin{eqnarray*}
\omega'\left(
\frac{P}{g}
+\lambda
\left(\frac{1}{\varphi}\frac{\partial\rho}{\partial z}-\frac{P}{g}\right)
\right)
\wedge\omega(\zeta)
& = &
\;\;\;\;\;\;\;\;\;\;\;\;\;\;\;\;\;\;\;\;\;\;\;\;\;\;\;\;\;\;\;\;\;\;\;\;\;\;\;\;\;\;\;\;\;\;\;\;\;\;\;\;\;\;\;\;\;\;\;
\;\;\;\;\;\;\;\;\;\;\;
\end{eqnarray*}
\begin{eqnarray*}
\;\;\;\;\;\;\;\;\;\;\;\;\;\;\;\;\;\;\;\;\;\;
& = &
\frac{1}{\varphi^{m-1}}
\sum_{k<l}
(-1)^{k+l}
\left(
\frac{\partial\rho}{\partial z_k}
\frac{P_l}{g}
-\frac{\partial\rho}{\partial z_l}
\frac{P_k}{g}
\right)
\bigwedge_{j\neq k,l}
d\left(\frac{\partial\rho}{\partial z_j}\right)
\wedge\omega(\zeta)
\wedge\lambda^{m-2}d\lambda
\\
\end{eqnarray*}
then
\begin{eqnarray}\label{prelimres}
\int_{\{|g(\zeta)|=\varepsilon\}\times[0,1]}
\psi(\zeta,\lambda)
& = &
\;\;\;\;\;\;\;\;\;\;\;\;\;\;\;\;\;\;\;\;\;\;\;\;\;\;\;\;\;\;\;\;\;\;\;\;\;\;\;\;\;\;\;\;
\;\;\;\;\;\;\;\;\;\;\;\;\;\;\;\;\;\;
\end{eqnarray}
\begin{eqnarray*}
& = &
\frac{(m-1)!}{(2\pi i)^m}
\int_{\{|g(\zeta)|=\varepsilon\}}
\frac{f(\zeta)}{\varphi^{m-1}}
\sum_{k<l}
(-1)^{k+l}
\left(
\frac{\partial\rho}{\partial z_k}
\frac{P_l}{g}
-\frac{\partial\rho}{\partial z_l}
\frac{P_k}{g}
\right)
\bigwedge_{j\neq k,l}
d\left(\frac{\partial\rho}{\partial z_j}\right)
\wedge\omega(\zeta)
\times
\int_0^1\lambda^{m-2}d\lambda
\\
& = &
\frac{(m-2)!}{(2\pi i)^m}
\int_{\{|g(\zeta)|=\varepsilon\}}
\frac{f(\zeta)
\sum_{k<l}
(-1)^{k+l}
\left(
\frac{\partial\rho}{\partial z_k}(\zeta)
P_l(\zeta,z)
-\frac{\partial\rho}{\partial z_l}(\zeta)
P_k(\zeta,z)
\right)}
{g(\zeta)\;\varphi(\zeta,z)^{m-1}}
\bigwedge_{j\neq k,l}
d\left(\frac{\partial\rho}{\partial z_j}(\zeta)\right)
\wedge\omega(\zeta)\,.
\\
\end{eqnarray*}

\bigskip

Finally, we want to specify
$d\psi(\zeta,\lambda)$
on
$\{|g(\zeta)|>\varepsilon\}\times[0,1]$. Since 
$f,\,g$ and $P$ are holomorphic, one has
\begin{eqnarray*}
d\psi(\zeta,\lambda)
& = &
\frac{m!}{(2\pi i)^m}\;
f(\zeta)\;\omega\left(\frac{\lambda}{\varphi(\zeta,z)}
\frac{\partial\rho}{\partial z}(\zeta)
+(1-\lambda)\frac{P(\zeta,z)}{g(\zeta)}\right)
\wedge\omega(\zeta)\\
& = &
\frac{m!}{(2\pi i)^m}\;
f(\zeta)\;
\omega
\left(
\lambda
\left(
\frac{1}{\varphi(\zeta,z)}
\frac{\partial\rho}{\partial z}(\zeta)
-\frac{P(\zeta,z)}{g(\zeta)}
\right)
\right)
\wedge\omega(\zeta)
\,
\\
\end{eqnarray*}
and
\begin{eqnarray}\nonumber
\omega
\left(
\lambda
\left(
\frac{1}{\varphi}
\frac{\partial\rho}{\partial z}
-\frac{P}{g}
\right)
\right)
\wedge\omega(\zeta)
& = &
\bigwedge_{k=1}^m
\left(
\left(
\frac{1}{\varphi}
\frac{\partial\rho}{\partial z_k}
-\frac{P_k}{g}
\right)
d\lambda
+\lambda
d\left(
\frac{1}{\varphi}
\frac{\partial\rho}{\partial z_k}
\right)
\right)
\wedge\omega(\zeta)
\\\nonumber
& = &
\sum_{k=1}^m
(-1)^{2m-k}
\left(
\frac{1}{\varphi}
\frac{\partial\rho}{\partial z_k}
-\frac{P_k}{g}
\right)
\bigwedge_{l\neq k}
d\left(
\frac{1}{\varphi}
\frac{\partial\rho}{\partial z_l}
\right)
\wedge\omega(\zeta)
\wedge
\lambda^{m-1}d\lambda
\\\nonumber
& = &
-\,
\omega'
\left(
\frac{1}{\varphi}
\frac{\partial\rho}{\partial z}
\right)
\wedge\omega(\zeta)
\wedge\lambda^{m-1}d\lambda
\\\nonumber
& & 
+\;
\sum_{k=1}^m
(-1)^{k-1}
\frac{P_k}{g}
\bigwedge_{l\neq k}
d\left(
\frac{1}{\varphi}
\frac{\partial\rho}{\partial z_l}
\right)
\wedge\omega(\zeta)
\wedge
\lambda^{m-1}d\lambda\,.
\\\label{prelimpv1}
& = &
-\,
\frac{1}{\varphi^m}
\,\omega'
\left(
\frac{\partial\rho}{\partial z}
\right)
\wedge\omega(\zeta)
\wedge\lambda^{m-1}d\lambda
\\\nonumber
& & 
+\;
\frac{1}{\varphi^{m-1}}
\sum_{k=1}^m
(-1)^{k-1}
\frac{P_k}{g}
\bigwedge_{l\neq k}
d\left(
\frac{\partial\rho}{\partial z_l}
\right)
\wedge\omega(\zeta)
\wedge
\lambda^{m-1}d\lambda
\\\nonumber
& & 
+\;
\frac{1}{\varphi^{m-2}}
\,
d\left(\frac{1}{\varphi}\right)
\wedge
\sum_{k=1}^m
(-1)^{k-1}
\frac{P_k}{g}
\;\times
\\\nonumber
& &
\;\times\;
\left[
\sum_{l<k}(-1)^{l-1}
+\sum_{l>k}(-1)^{l-2}
\right]
\frac{\partial\rho}{\partial z_l}
\bigwedge_{j\neq k,l}
d\left(
\frac{\partial\rho}{\partial z_j}
\right)
\wedge\omega(\zeta)
\wedge
\lambda^{m-1}d\lambda
\,,
\end{eqnarray}
the last equality coming from (\ref{prelim1}).
On the other hand, since
$\varphi(\zeta,z)
=\sum_{q=1}^m(\zeta_q-z_q)
\frac{\partial\rho}{\partial z_q}$,
one has
\begin{eqnarray*}
\frac{1}{\varphi^{m-2}}
d\left(
\frac{1}{\varphi}\right)
\wedge
\sum_{k=1}^m
(-1)^{k-1}
\frac{P_k}{g}
\left[
\sum_{l<k}(-1)^{l-1}
+\sum_{l>k}(-1)^l
\right]
\frac{\partial\rho}{\partial z_l}
\bigwedge_{j\neq k,l}
d\left(
\frac{\partial\rho}{\partial z_j}
\right)
\wedge\omega(\zeta)
& = &
\end{eqnarray*}
\begin{eqnarray*}
& = &
\frac{1}{\varphi^m}
\sum_{q=1}^m
(\zeta_q-z_q)
d\left(\frac{\partial\rho}{\partial z_q}\right)
\wedge
\sum_{k<l}
(-1)^{k+l}
\left(
\frac{P_k}{g}
\frac{\partial\rho}{\partial z_l}
-
\frac{P_l}{g}
\frac{\partial\rho}{\partial z_k}
\right)
\bigwedge_{j\neq k,l}
d\left(
\frac{\partial\rho}{\partial z_j}
\right)
\wedge\omega(\zeta)
\\
\\
& = &
\frac{1}{\varphi^m}
\sum_{k<l}
(-1)^{k+l}
(\zeta_k-z_k)
\left(
\frac{P_k}{g}
\frac{\partial\rho}{\partial z_l}
-
\frac{P_l}{g}
\frac{\partial\rho}{\partial z_k}
\right)
(-1)^{k-1}
\bigwedge_{j\neq l}
d\left(
\frac{\partial\rho}{\partial z_j}
\right)
\wedge\omega(\zeta)
\\
& & 
+\;
\frac{1}{\varphi^m}
\sum_{k<l}
(-1)^{k+l}
(\zeta_l-z_l)
\left(
\frac{P_k}{g}
\frac{\partial\rho}{\partial z_l}
-
\frac{P_l}{g}
\frac{\partial\rho}{\partial z_k}
\right)
(-1)^{l-2}
\bigwedge_{j\neq k}
d\left(
\frac{\partial\rho}{\partial z_j}
\right)
\wedge\omega(\zeta)
\\
\\
& = &
\frac{1}{\varphi^m}
\sum_{k>l}
(-1)^{k-1}
(\zeta_l-z_l)
\left(
\frac{P_l}{g}
\frac{\partial\rho}{\partial z_k}
-
\frac{P_k}{g}
\frac{\partial\rho}{\partial z_l}
\right)
\bigwedge_{j\neq k}
d\left(
\frac{\partial\rho}{\partial z_j}
\right)
\wedge\omega(\zeta)
\\
& &
+\;
\frac{1}{\varphi^m}
\sum_{k<l}
(-1)^k
(\zeta_l-z_l)
\left(
\frac{P_k}{g}
\frac{\partial\rho}{\partial z_l}
-
\frac{P_l}{g}
\frac{\partial\rho}{\partial z_k}
\right)
\bigwedge_{j\neq k}
d\left(
\frac{\partial\rho}{\partial z_j}
\right)
\wedge\omega(\zeta)
\end{eqnarray*}
\begin{eqnarray*}
& = &
\frac{1}{\varphi^m}
\sum_{k=1}^m
(-1)^k
\sum_{l=1}^m
(\zeta_l-z_l)
\left(
\frac{P_k}{g}
\frac{\partial\rho}{\partial z_l}
-
\frac{P_l}{g}
\frac{\partial\rho}{\partial z_k}
\right)
\bigwedge_{j\neq k}
d\left(
\frac{\partial\rho}{\partial z_j}
\right)
\wedge\omega(\zeta)
\\\nonumber
\\\nonumber
& = &
\frac{1}{\varphi^m}
\sum_{k=1}^m
(-1)^k
\left(
\frac{P_k}{g}
\,\varphi
-
\frac{g(\zeta)-g(z)}{g(\zeta)}
\frac{\partial\rho}{\partial z_k}
\right)
\bigwedge_{j\neq k}
d\left(
\frac{\partial\rho}{\partial z_j}
\right)
\wedge\omega(\zeta)
\\\nonumber
\\\nonumber
& = &
\frac{1}{\varphi^{m-1}}
\sum_{k=1}^m
(-1)^k
\frac{P_k}{g}
\bigwedge_{j\neq k}
d\left(
\frac{\partial\rho}{\partial z_j}
\right)
\wedge\omega(\zeta)
\;+\;
\frac{1}{\varphi^m}
\left(1-\frac{g(z)}{g(\zeta)}\right)
\omega'
\left(
\frac{\partial\rho}{\partial z}
\right)
\wedge\omega(\zeta)\,.
\\
\end{eqnarray*}
It follows from (\ref{prelimpv1}) that
\begin{eqnarray}\nonumber
\int_{\{|g(\zeta)|>\varepsilon\}\times[0,1]}
d\psi(\zeta,\lambda)
& = &
-\,
\frac{m!}{(2\pi i)^m}
\int_{\{|g(\zeta)|>\varepsilon\}}
\frac{f(\zeta)\;g(z)}{g(\zeta)\,\varphi(\zeta,z)}
\,\omega'
\left(
\frac{\partial\rho}{\partial z}
\right)
\wedge\omega(\zeta)
\;\times\;
\int_0^1\lambda^{m-1}d\lambda
\\\label{prelimpv2}
& = &
-\,g(z)\,
\frac{(m-1)!}{(2\pi i)^m}
\int_{\{|g(\zeta)|>\varepsilon\}}
\frac{f(\zeta)}{g(\zeta)\,\varphi(\zeta,z)}
\,\omega'
\left(
\frac{\partial\rho}{\partial z}
\right)
\wedge\omega(\zeta)\,.
\\\nonumber
\end{eqnarray}

Lastly, we deduce by 
lemma \ref{orientation}, (\ref{stokes}), 
(\ref{prelimf}), (\ref{prelimres}) and
(\ref{prelimpv2}) that, for all
$z\in\mathbb{B}_2$,
\begin{eqnarray*}
-\,\lim_{\varepsilon\rightarrow0}
g(z)\,
\frac{(m-1)!}{(2\pi i)^m}
\int_{\{|g(\zeta)|>\varepsilon\}}
\frac{f(\zeta)}{g(\zeta)\,\varphi(\zeta,z)}
\,\omega'
\left(
\frac{\partial\rho}{\partial z}
\right)
\wedge\omega(\zeta)
& = &
\;\;\;\;\;\;\;\;\;\;\;\;\;\;\;\;\;\;\;\;\;\;\;\;\;\;\;\;\;\;\;\;\;\;\;\;\;\;\;\;\;\;\;\;\;\;\;\;\;\;\;\;\;\;\;
\end{eqnarray*}
\begin{eqnarray*}
& = &
-\lim_{\varepsilon\rightarrow0}
\frac{(m-2)!}{(2\pi i)^m}
\int_{\{|g(\zeta)|=\varepsilon\}}
\frac{f(\zeta)
\sum_{k<l}
(-1)^{k+l}
\left(
\frac{\partial\rho}{\partial z_k}
P_l
-\frac{\partial\rho}{\partial z_l}
P_k
\right)}
{g(\zeta)\;\varphi(\zeta,z)^{m-1}}
\bigwedge_{j\neq k,l}
d\left(\frac{\partial\rho}{\partial z_j}\right)
\wedge\omega(\zeta)
\\
\\
& & 
-\;
(-1)^{m(m-1)/2}\;f(z)
\\
\end{eqnarray*}
and the proof of the proposition is achieved.

\end{proof}

\bigskip

\subsection{Case of theorem \ref{theorem}}\label{tranche}

Now consider
$\mathbb{C}^2$ with
$\Omega=\mathbb{B}_2$ and
$\partial\Omega=\mathbb{S}_2=\{|z_1|^2+|z_2|^2=1\}$ 
where
$\rho(\zeta)=\|\zeta\|^2-1=<\overline{\zeta},\zeta>-1$.
Moreover,
\begin{eqnarray*}
\begin{cases}
\frac{\partial\rho}{\partial z}(\zeta)=\overline{\zeta}\,,\\
\varphi(\zeta,z)=<\overline{\zeta},\zeta-z>=1-<\overline{\zeta},z>\,.\\
\end{cases}
\end{eqnarray*}
We also choose (see (\ref{defgn}) in introduction)
\begin{eqnarray*}
g_n(z) 
& = &
z_1^{m_1}\prod_{j=2}^{n-1}(z_1-\eta_jz_2)^{m_j}\,z_2^{m_n}
\end{eqnarray*}
with $m_j\geq0,\;j=1,\ldots,n$,
\begin{eqnarray*}
0=|\eta_1|<|\eta_2|\leq\cdots\leq|\eta_{n-1}|<|\eta_n|=+\infty
\end{eqnarray*}
and associate 
$P_n(\zeta,z)=(P_n^1(\zeta,z),P_n^2(\zeta,z))$
(we will specify $P_n(\zeta,z)$ in section~\ref{interpolers}).
One could use proposition \ref{prelim} in the following. Nevertheless,
in order to prove theorem~\ref{theorem}, we need another set than
$\{|g_n(\zeta)|>\varepsilon\}$.
For all $p=1,\ldots,n$ we set
\begin{eqnarray}\label{alpha}
\alpha_p & := &
\frac{|\eta_p|}{\sqrt{1+|\eta_p|^2}}
\end{eqnarray}
(with 
$\alpha_n:=1$). Then
$0=\alpha_1<\alpha_2\leq\cdots\leq\alpha_{n-1}<\alpha_n$.

Now there are
$\widetilde{n}\leq n$
and 
$1=q_1<\cdots<q_{\widetilde{n}}=n$
such that
\begin{eqnarray}\label{defq}
0=\alpha_1=\alpha_{q_1}<\alpha_{q_1+1} & =  \cdots  = & \alpha_{q_2}<\cdots\\\nonumber
<\alpha_{q_l+1} & =\cdots= & \alpha_{q_{l+1}}<\cdots\\\nonumber
<\alpha_{q_{\widetilde{n}-2}+1} & =\cdots= & \alpha_{q_{\widetilde{n}-1}}
<\alpha_{q_{\widetilde{n}}}=\alpha_n=1\,.
\end{eqnarray}
\bigskip

Then for all small enough $\varepsilon>0$ we can set
\begin{eqnarray}\label{deftranche}
\widetilde{\Sigma}_{\varepsilon} & := & 
\bigcup_{l=1}^{\widetilde{n}-1}
\left\{\zeta\in\mathbb{S}_2,\;
\alpha_{q_l}+\varepsilon<|\zeta_1|<\alpha_{q_{l+1}}-\varepsilon
\right\}\,
\end{eqnarray}
with boundary $\partial\widetilde{\Sigma}_{\varepsilon}$ 
and orientation that satisfies Stokes formula. 
For the proof of theorem~\ref{theorem} we will use 
the following result that is similar to
proposition \ref{prelim}
with $\widetilde{\Sigma}_{\varepsilon}$:

\begin{proposition}\label{proptranche}

For all 
$f\in\mathcal{O}\left(\overline{\mathbb{B}_2}\right)$ and all $z\in\mathbb{B}_2$,

\begin{eqnarray}\label{prelimtranche}
f(z) & = &
\;\;\;\;\;\;\;\;\;\;\;\;\;\;\;\;\;\;\;\;\;\;\;\;\;\;\;\;\;\;\;\;\;\;\;\;\;\;\;\;\;\;\;\;\;\;\;\;
\;\;\;\;\;\;\;\;\;\;\;\;\;\;\;\;\;\;\;\;\;\;\;\;\;\;\;\;\;\;\;\;\;\;\;\;\;\;\;\;
\end{eqnarray}
\begin{eqnarray*}
& = &
\lim_{\varepsilon\rightarrow0}
\frac{1}{(2\pi i)^2}
\left[
\int_{|\zeta_1|=1-\varepsilon}
-\sum_{l=2}^{\widetilde{n}-1}
\left(\int_{|\zeta_1|=\alpha_{q_l}+\varepsilon}-\int_{|\zeta_1|=\alpha_{q_l}-\varepsilon}\right)
-\int_{|\zeta_1|=\varepsilon}
\right]
\frac{f(\zeta)\det\left(\overline{\zeta},P_n(\zeta,z)\right)}{g_n(\zeta)\left(1-<\overline{\zeta},z>\right)}
\,\omega(\zeta)\\
\\
& & 
-\,\lim_{\varepsilon\rightarrow0}\,
\frac{g_n(z)}{(2\pi i)^2}\,
\sum_{l=1}^{\widetilde{n}-1}\,
\int_{\alpha_{q_l}+\varepsilon<|\zeta_1|<\alpha_{q_{l+1}}-\varepsilon}\;
\frac{f(\zeta)\;\omega'\left(\overline{\zeta}\right)\wedge\omega(\zeta)}
{g_n(\zeta)\,\left(1-<\overline{\zeta},z>\right)^2}\,.
\end{eqnarray*}

\end{proposition}

\bigskip

\begin{proof}

The proof is similar to the one of proposition \ref{prelim}
with
\begin{eqnarray*}
\widetilde{\psi}(\zeta,\lambda)
& = &
\frac{1}{(2\pi i)^2}
\;f(\zeta)
\;\omega'\left(
\lambda
\frac{\overline{\zeta}}{1-<\overline{\zeta},z>}
+(1-\lambda)\frac{P_n(\zeta,z)}{g_n(\zeta)}\right)
\wedge\omega(\zeta)\,.
\end{eqnarray*}

In particular, we see that,
for all small enough $\varepsilon>0$,
$\widetilde{\psi}$ is well-defined on
a neighborhood of
$\overline{\widetilde{\Sigma}_{\varepsilon}}$.
Indeed, if for 
$j=1,\ldots,n-1$,
$\zeta_1-\eta_j\zeta_2$ vanishes on
$\mathbb{S}_2$, then
\begin{eqnarray*}
\begin{cases}
|\zeta_1|=|\eta_j|\,|\zeta_2|\\
|\zeta_1|^2+|\zeta_2|^2=1
\end{cases}
& \Rightarrow &
|\zeta_1|^2=|\eta_j|^2(1-|\zeta_1|^2)
\end{eqnarray*}
thus
$|\zeta_1|=\alpha_j$.

The only difference is to specify the orientation on
$\widetilde{\Sigma}_{\varepsilon}$ and
$\partial\widetilde{\Sigma}_{\varepsilon}$.
First, the orientation on
$\widetilde{\Sigma}_{\varepsilon}$
is induced by the one on
$\mathbb{S}_2$.
Next, for all
$0<\beta_1<\beta_2<1$,
we have as in lemma \ref{orientation}
\begin{eqnarray*}
\partial
\left(
\{\beta_1<|\zeta_1|<\beta_2\}
\times[0,1]
\right)
& = &
\left(
\partial
\{\beta_1<|\zeta_1|<\beta_2\}
\right)
\times[0,1]
\,-\,
\{\beta_1<|\zeta_1|<\beta_2\}
\times\left(
\partial[0,1]
\right)
\,
\end{eqnarray*}
where the orientation of
$\partial\{\beta_1<|\zeta_1|<\beta_2\}$
is fixed by the Stokes formula on
$\{\beta_1<|\zeta_1|<\beta_2\}$,
i.e. for all
$(2n-2)$-form $\chi$,

\begin{eqnarray*}
\int_{\{\beta_1<|\zeta_1|<\beta_2\}}
d\chi(\zeta)
& = &
\int_{\{|\zeta_1|=\beta_2\}}
\chi(\zeta)
-\int_{\{|\zeta_1|=\beta_1\}}
\chi(\zeta)\,.
\\
\end{eqnarray*}
It follows that
\begin{eqnarray*}
\partial\widetilde{\Sigma}_{\varepsilon}
& = &
\sum_{l=1}^{\widetilde{n}-1}
\left(
\{|\zeta_1|=\alpha_{q_{l+1}}-\varepsilon\}
\,-\,
\{|\zeta_1|=\alpha_{q_l}+\varepsilon\}
\right)\,.
\end{eqnarray*}

Since
\begin{eqnarray*}
\lim_{\varepsilon\rightarrow0}
\int_{\widetilde{\Sigma}_{\varepsilon}}
\psi(\zeta,1)
& = &
\frac{1}{(2\pi i)^2}
\int_{\mathbb{S}_2}
f(\zeta)\,
\frac{\omega'(\overline{\zeta})\wedge\omega(\zeta)}
{(1-<\overline{\zeta},z>^2}
\;=\;
-\,f(z)\,,
\\
\end{eqnarray*}
we get
\begin{eqnarray*}
-\,\lim_{\varepsilon\rightarrow0}\,
\frac{g_n(z)}{(2\pi i)^2}\,
\sum_{l=1}^{\widetilde{n}-1}\,
\int_{\alpha_{q_l}+\varepsilon<|\zeta_1|<\alpha_{q_{l+1}}-\varepsilon}\;
\frac{f(\zeta)\;\omega'\left(\overline{\zeta}\right)\wedge\omega(\zeta)}
{g_n(\zeta)\,\left(1-<\overline{\zeta},z>\right)^2}
& = &
\;\;\;\;\;\;\;\;\;\;\;\;\;\;\;\;\;\;\;\;\;\;\;\;\;\;\;\;\;\;\;\;
\end{eqnarray*}
\begin{eqnarray*}
& = &
\lim_{\varepsilon\rightarrow0}
\sum_{l=1}^{\widetilde{n}-1}
\frac{1}{(2\pi i)^2}
\left[
\int_{|\zeta_1|=\alpha_{q_{l+1}}-\varepsilon}
-\int_{|\zeta_1|=\alpha_{q_l}+\varepsilon}
\right]
\frac{-f(\zeta)\,
(\overline{\zeta}_1P_n^2
-\overline{\zeta}_2P_n^1)}
{g_n(\zeta)\,(1-<\overline{\zeta},z>)}
\,\omega(\zeta)
+f(z)\,
\\
\end{eqnarray*}
and the proposition follows.

\end{proof}

\bigskip

\section{Some preliminar results on the Lagrange interpolation formula}\label{lagrange}

\bigskip

Let $W\subset\mathbb{C}$ be an open set, $f\,\in\mathcal{O}(W)$ and
$\eta_1,\ldots,\eta_n\,\in W$ be different complex numbers (with associate multiplicities $m_1,\ldots,m_n\,\in\mathbb{N}$). Consider the following Lagrange interpolation polynomial for 
$f$ on the $\eta_j$:
\begin{eqnarray}\label{lagrange1}
L_{f,\eta^m}(X) & := & 
\mathcal{L}
\left(
\eta_1^{m_1},\ldots,\eta_n^{m_n};
\frac{f(t)}{X-t}
\right)
\;\;\;\;\;\;\;\;\;\;\;\;\;\;\;\;\;\;\;\;\;\;\;\;\;\;\;\;\;\;\;\;\;\;
\;\;\;\;\;\;\;\;\;
\end{eqnarray}
\begin{eqnarray*}
& = &
\prod_{j=1}^n(X-\eta_j)^{m_j}
\sum_{p=1}^n
\frac{1}{(m_p-1)!}
\frac{\partial^{m_p-1}}{\partial t^{m_p-1}}|_{t=\eta_p}\left(\frac{f(t)}{(X-t)\prod_{j=1,j\neq p}^n(t-\eta_j)^{m_j}}\right)
\\
& = & 
\sum_{p=1}^n\;\prod_{j=1,j\neq p}^n
(X-\eta_j)^{m_j}\;
\sum_{s=0}^{m_p-1}(X-\eta_p)^s\,\frac{1}{s!}
\frac{\partial^s}{\partial t^s}|_{t=\eta_p}\left(\frac{f(t)}{\prod_{j=1,j\neq p}^n(t-\eta_j)^{m_j}}\right)
\,.
\\
\end{eqnarray*}
%

%

\bigskip

One has the following preliminar result:

\begin{lemma}\label{prelimlagrange}

$L_{f,\eta^m}$ is the unique polynomial $P\in\mathbb{C}[X]$ with degree at most $N-1$
(where 
$N=m_1+\cdots+m_n$)
that satisfies, for all $p=1,\ldots,n$ and all $s=1,\ldots,m_p-1$,
\begin{eqnarray*}
P^{(s)}(\eta_p)
& = &
f^{(s)}(\eta_p)\,.
\end{eqnarray*}

\end{lemma}

\begin{proof}


First, we have for all
$p=1,\ldots,n$ and all $l=1,\ldots,m_p-1$,
\begin{eqnarray*}
L_{f,\eta^m}^{(l)}(\eta_p)
& = &
\;\;\;\;\;\;\;\;\;\;\;\;\;\;\;\;\;\;\;\;\;\;\;\;\;\;\;\;\;\;\;\;\;\;\;\;\;\;\;\;\;\;\;\;\;\;\;\;\;\;\;\;\;\;\;\;\;\;
\;\;\;\;\;\;\;\;\;\;\;\;\;\;\;\;\;\;\;\;\;\;\;\;\;\;\;\;\;\;\;\;\;\;\;\;\;\;\;\;
\end{eqnarray*}
\begin{eqnarray*}
& = &
\frac{\partial^l}{\partial X^l}|_{X=\eta_p}
\prod_{j=1,j\neq p}^n
(X-\eta_j)^{m_j}
\sum_{s=0}^{m_p-1}(X-\eta_p)^s\frac{1}{s!}
\frac{\partial^s}{\partial t^s}|_{t=\eta_p}
\left(\frac{f(t)}{\prod_{j=1,j\neq p}^n(t-\eta_j)^{m_j}}\right)
\\
& = &
\sum_{u=0}^l
\frac{l!}{u!\,(l-u)!}
\frac{\partial^{l-u}}{\partial X^{l-u}}|_{X=\eta_p}
\left(
\prod_{j=1,j\neq p}^n
(X-\eta_j)^{m_j}
\right)
\;\times
\\
& & 
\;\;\;\;\;\;\;\;\;\;\;\;\;\;\;\;\;\;\;\;\;\;\;\;\;\;\;\;\;\;\;\;\;\;\;\;\;\;\;\;\;\;\;\;\;\;\;\;\;\;\;\;\;\;
\times\;
u!\;
\frac{1}{u!}
\frac{\partial^u}{\partial t^u}|_{t=\eta_p}
\left(\frac{f(t)}{\prod_{j=1,j\neq p}^n(t-\eta_j)^{m_j}}\right)
\\
\\
& = &
\frac{\partial^u}{\partial t^u}|_{t=\eta_p}
\left[
\prod_{j=1,j\neq p}^n(t-\eta_j)^{m_j}
\times
\frac{f(t)}{\prod_{j=1,j\neq p}^n(t-\eta_j)^{m_j}}
\right]
\\
\\
& = &
f^{(l)}(\eta_p)\,.
\\
\end{eqnarray*}

Finally, let be
$P\in\mathbb{C}[X]$
another polynomial of degree at most
$N-1$ that satisfies
$L_{f,\eta^m}^{(s)}(\eta_p)=f^{(s)}(\eta_p)$,
for all
$p=1,\ldots,n$ and all
$s=0,\ldots,m_p-1$.
It follows that 
$P-L_{f,\eta^m}$ is divisible by
$\prod_{j=1}^n(X-\eta_j)$
and of degree at most
$N-1$ then
$P-L_{f,\eta^m}=0$ and this proves the lemma.

\end{proof}

\bigskip

We have an additional result when $f$ is a polynomial function.

\begin{lemma}\label{lagrangeuclide}

Consider the Euclidean division of $P\in\mathbb{C}[X]$ with degree $k$ by 
$G(X):=\prod_{j=1}^n(X-\eta_j)^{m_j}$
\begin{eqnarray*}
P & = & G\cdot Q(P,G)\,+\,R(P,G)\,,
\end{eqnarray*}
where $Q(P,G)$ (resp. $R(P,G)$) is the quotient (resp. remainder). Then
\begin{eqnarray*}
R(P,G) & = & L_{P,\eta^m}\,.
\end{eqnarray*}
In particular,
\begin{eqnarray}\label{preliminfini}
& &
\frac{P(X)}{G(X)}
\,-\,
\sum_{p=1}^n
\frac{1}{(m_p-1)!}
\frac{\partial^{m_p-1}}{\partial t^{m_p-1}}|_{t=\eta_p}
\left[
\frac{P(t)}{(X-t)
\prod_{j=1,j\neq p}^n
(t-\eta_j)^{m_j}}
\right]
\;=\;
\\\nonumber
\\\nonumber
& & 
\;\;\;=\;
Q(P,G)
\,.
\\\nonumber
\end{eqnarray}

If $P(X)=X^k$ we have in addition
\begin{eqnarray}
Q(X^k,G) & = &
\sum_{u=0}^{k-N}X^{k-N-u}
\sum_{v_1+\cdots+v_n=u}\,\prod_{j=1}^n\frac{(v_j+m_j-1)!}{v_j!\,(m_j-1)!}\,\eta_j^{v_j}\,.
\end{eqnarray}

\end{lemma}

\bigskip

\begin{proof}

First, we have $G^{(s)}(\eta_p)=0$ for all $p=1,\ldots,n$ and all $s=0,\ldots,m_p-1$ then
$$R(P,G)^{(s)}(\eta_p)=P^{(s)}(\eta_p)=L_{P,\eta^m}^{(s)}(\eta_p)\,.$$
The first assertion follows since $L_{P,\eta^m}$ and $R(P,G)$ have degree at most $N-1$.
\bigskip

Next, (\ref{preliminfini}) follows from the Euclidean division of $P$
by $G$.
\bigskip

Now consider $P(X)=X^k$. If $k<N$ then $Q(X^k,G)=0$ and the second assertion is obvious. If $k\geq N$ we can write
$G(X)=\prod_{j=1}^N(X-\eta'_j)$ (where the $\eta'_j$ are not necessarily different) and we prove the assertion by induction on $N$.

If $N=1$ we have
\begin{eqnarray*}
X^k & = & X^k-{\eta'_1}^k+{\eta'_1}^k\\
& = & (X-\eta'_1)\sum_{u=0}^{k-1}X^{k-1-u}{\eta'_1}^u+{\eta'_1}^k\,.
\end{eqnarray*}
Now assume that it is true for $N-1$ and let be $\eta'_N\in\mathbb{C}$. We have similarly
\begin{eqnarray*}
X^k & = & (X-\eta'_N)\sum_{v_N=0}^{k-1}{\eta'_N}^{v_N}X^{k-1-v_N}+{\eta'_N}^k\\
\end{eqnarray*}
Since
\begin{eqnarray*}
Q\left(X^{k-1-v_N},\prod_{j=1}^{N-1}(X-\eta'_j)\right)
& = & \sum_{w=0}^{k-N-v_N}X^{k-N-v_N-w}\sum_{v_1+\cdots+v_{N-1}=w}\prod_{j=1}^{N-1}{\eta'_j}^{v_j}\\
& = & \sum_{u=v_N}^{k-N}X^{k-N-u}\sum_{v_1+\cdots+v_{N-1}=u-v_N}\prod_{j=1}^{N-1}{\eta'_j}^{v_j}\,,
\end{eqnarray*}
it follows that
\begin{eqnarray*}
X^k & = & 
(X-\eta'_N)\prod_{j=1}^{N-1}(X-\eta'_j)
\sum_{u=0}^{k-N}X^{k-N-u}
\sum_{v_N=0}^{\min(k-1,u)}{\eta'_N}^{v_N}
\sum_{v_1+\cdots+v_{N-1}=u-v_N}\prod_{j=1}^{N-1}{\eta'_j}^{v_j}\\
& & \;+\;(X-\eta'_N)R+{\eta'_N}^k\\
\\
& = & G(X)\,\sum_{u=0}^{k-N}X^{k-N-u}\sum_{v_N=0}^u\;
\sum_{v_1+\cdots+v_{N-1}+v_N=u}{\eta'_N}^{v_N}\prod_{j=1}^{N-1}{\eta'_j}^{v_j}
\,+\,(X-\eta'_N)R+{\eta'_N}^k\,,\\
\end{eqnarray*}
with $\deg R\leq N-2$ and this proves the assertion.
\bigskip

To finish the proof, we notice that for all $u=0,\ldots,k-N$
\begin{eqnarray*}
\sum_{v_1+\cdots+v_N=u}\prod_{j=1}^N{\eta'_j}^{v_j} 
& = & \sum_{v_{1,1}+\cdots+v_{n,m_n}=u}\,\prod_{l=1}^n\eta_l^{v_{l,1}+\cdots+v_{l,m_l}}\\
& = & \sum_{w_1+\cdots+w_n=u}\,\prod_{l=1}^n
\left(\eta_l^{w_l}\,card\{{v_{l,1}+\cdots+v_{l,m_l}=w_l}\}\right)\\
& = & \sum_{w_1+\cdots+w_n=u}\,\prod_{l=1}^n
\frac{(w_l+m_l-1)!}{w_l!\,(m_l-1)!}\,\eta_l^{w_l}\,,
\end{eqnarray*}
the last equality coming from the following lemma.

\end{proof}

\begin{lemma}\label{combinatoire}

For all
$m\geq1$ and $q\geq0$,
\begin{eqnarray*}
card\{(v_1,\ldots,v_m)\in\mathbb{N}^m,\;v_1+\cdots+v_m=q\}
& = & \frac{(q+m-1)!}{q!\,(m-1)!}\,.
\end{eqnarray*}

\end{lemma}

\begin{proof}

Consider the following formal series

$$\sum_{v_1,\ldots,v_m\geq 0}X_1^{v_1}\cdots X_m^{v_m}\,$$ 
The coefficient of order $q$ after evaluation
$X_1=\cdots=X_m=X$
is exactly 
$card\{v_1+\cdots+v_m=q\}$.
On the other hand, we have
$$\prod_{j=1}^m\left(\sum_{v_j\geq 0}X_j^{v_j}\right)
=\prod_{j=1}^m\frac{1}{1-X_j}$$
that gives after evaluation
\begin{eqnarray*}
\frac{1}{(1-X)^m} & = &
\frac{1}{(m-1)!}\frac{d^{m-1}}{dX^{m-1}}\left(\frac{1}{1-X}\right)\\
& = & 
\frac{1}{(m-1)!}\sum_{k\geq m-1}k(k-1)\cdots(k-m+2)X^{k-m+1},
\\
\end{eqnarray*}
whose coefficient of order $q$ is
$\dfrac{(q+m-1)\cdots(q+1)}{(m-1)!}$.

\end{proof}

\bigskip

\section{Calculation of the remainder}\label{remainders}

\bigskip

We set
\begin{eqnarray}\label{ueta}
& &
U_{\eta}
\;=\;
\{z\in\mathbb{C}^2,\,
z_1\neq0,\,z_2\neq0\;
\mbox{and}\;
\forall\,p=2,\ldots,n-1,\,
z_1-\eta_pz_2\neq0
\}
\,.
\\\nonumber
\end{eqnarray}
We also remind 
$g_n(z)=z_1^{m_1}\prod_{j=2}^{n-1}(z_1-\eta_jz_2)^{m_j}z_2^{m_n},\;
m_1,\ldots,m_n\in\mathbb{N}$.
So we can give the following result that we will prove in this section.

\begin{proposition}\label{remainder}

For all $z\in U_{\eta}$,
we have
\begin{eqnarray}\label{remainder1}
-\;
\lim_{\varepsilon\rightarrow0}\frac{g_n(z)}{(2\pi i)^2}
\int_{\Sigma_{\varepsilon}}
\frac{\zeta_1^{k_1}\zeta_2^{k_2}\;\omega'\left(\overline{\zeta}\right)\wedge\omega(\zeta)}
{g_n(\zeta)\;\left(1-<\overline{\zeta},z>\right)^2}
& = & 
\;\;\;\;\;\;\;\;\;\;\;\;\;\;\;\;\;\;\;\;\;\;\;\;\;\;\;\;\;\;\;\;\;\;\;\;\;\;\;
\\\nonumber
\end{eqnarray}
\begin{eqnarray*}
& = &
{\bf 1}_{k_1+k_2\geq N,\,k_1\geq m_1,k_2\geq m_n}
\,z_1^{k_1}\,z_2^{k_2}\\
\\
& - &
{\bf 1}_{k_1+k_2\geq N}\;
\sum_{p=2}^{n-1}
z_1^{m_1}\prod_{j=2,j\neq p}^{n-1}(z_1-\eta_jz_2)^{m_j}z_2^{m_n}
\sum_{s=0}^{m_p-1}z_2^{m_p-1-s}(z_1-\eta_pz_2)^s
\\
& & 
\;\times\,\frac{1}{s!}\frac{\partial^s}{\partial t^s}|_{t=\eta_p}
\left[\frac{t^{k_1}}{t^{m_1}\prod_{j=2,j\neq p}^{n-1}(t-\eta_j)^{m_j}}
\left(\frac{z_2+|\eta_p|^2z_1/t}{1+|\eta_p|^2}\right)^{k_1+k_2-N+1}\right]\\
\\
& + & 
{\bf 1}_{k_1\leq m_1-1,k_2\geq N-k_1}
\sum_{p=2}^{n-1}z_1^{m_1}\prod_{j=2,j\neq p}^{n-1}(z_1-\eta_jz_2)^{m_j}z_2^{m_n}
\sum_{s=0}^{m_p-1}z_2^{m_p-1-s}(z_1-\eta_pz_2)^s\\
& & 
\;\times\,
\frac{1}{s!}\frac{\partial^s}{\partial t^s}|_{t=\eta_p}
\left[
\frac{t^{k_1}z_2^{k_1+k_2-N+1}}
{t^{m_1}\prod_{j=2,j\neq p}^{n-1}(t-\eta_j)^{m_j}}
\right]\\
\\
& + & 
{\bf 1}_{k_2\leq m_n-1,k_1\geq N-k_2}
\sum_{p=2}^{n-1}z_1^{m_1}\prod_{j=2,j\neq p}^{n-1}(z_1-\eta_jz_2)^{m_j}z_2^{m_n}
\sum_{s=0}^{m_p-1}z_2^{m_p-1-s}(z_1-\eta_pz_2)^s
\\
& & 
\;\times\,
\frac{1}{s!}\frac{\partial^s}{\partial t^s}|_{t=\eta_p}
\left[
\frac{t^{N-1-k_2}z_1^{k_1+k_2-N+1}}
{t^{m_1}\prod_{j=2,j\neq p}^{n-1}(t-\eta_j)^{m_j}}
\right]\,,
\\
\end{eqnarray*}
where
\begin{eqnarray*}
{\bf 1}_{k_1+k_2\geq N}
& := &
\begin{cases}
1\;\mbox{ if }\;k_1+k_2\geq N,\\
0\;\mbox{ otherwise}
\end{cases}
\end{eqnarray*}
(likewise for
${\bf 1}_{k_1+k_2\geq N,\,k_1\geq m_1,k_2\geq m_n}$,
${\bf 1}_{k_1\leq m_1-1,k_2\geq N-k_1}$
and
${\bf 1}_{k_2\leq m_n-1,k_1\geq N-k_2}$).

\end{proposition}

\bigskip

We begin with the following lemma.

\begin{lemma}\label{infini}

For all 
$r\in\,[0,1]$ such that 
$r\neq\alpha_p,\;\forall\,p=1,\ldots,n$ and all $k_1\geq0$, we have
\begin{eqnarray}
\frac{1}{2\pi i}\int_{|\zeta_1|=+\infty}
\frac{\zeta_1^{k_1-m_1+1}\;d\zeta_1}{\prod_{j=2}^{n-1}
(\zeta_1-\eta_j\zeta_2)^{m_j}\;\left(\zeta_1-\frac{r^2z_1\zeta_2}{\zeta_2-(1-r^2z_2)}\right)^2}
& = &
\;\;\;\;\;\;\;\;\;\;\;\;\;\;\;\;
\end{eqnarray}
\begin{eqnarray*}
\;\;\;\;\;\;\;\;\;\;\;\;\;\;\;
& = &
{\bf{1}}_{k_1\geq m_1+\cdots+m_{n-1}}\;
\zeta_2^{k_1-(m_1+\cdots+m_{n-1})}\,P\left(\frac{r^2z_1}{\zeta_2-(1-r^2)z_2}\right)
\,,
\end{eqnarray*}
where $P\in\mathbb{C}[X]$.

\end{lemma}

\begin{proof}

First, notice that, for all
$z\in\mathbb{B}_2$ and all
$\zeta\in\mathbb{S}_2$, we have by the
Cauchy-Schwarz inequality
$|1-<\overline{\zeta},z>|
\geq1-\|z\|\,\|\zeta\|=1-\|z\|>0$.
In addition, 
\begin{eqnarray}\label{zetaz2}
|(1-r^2)z_2|
=\sqrt{1-r^2}|\,\overline{\zeta}_2z_2|
\leq|\sqrt{1-r^2}\,\|z\|<|\zeta_2|
\end{eqnarray}
and
\begin{eqnarray}\label{zetaz1}
\left|
\frac{r^2z_1\zeta_2}{\zeta_2-(1-r^2)z_2}
\right|
=r\frac{|\overline{\zeta_1}z_1|}{|1-\overline{\zeta_2}z_2|}
<r=|\zeta_1|
\end{eqnarray}
since
$|1-\overline{\zeta}_2z_2|-|\overline{\zeta}_1z_1|\geq
1-|\overline{\zeta}_1z_1|-|\overline{\zeta}_2z_2|
\geq1-\|\zeta\|\,\|z\|>0$.
In particular, 
$\frac{r^2z_1\zeta_2}{\zeta_2-(1-r^2)z_2}$ 
is residue with respect to
$\zeta_1$ in the above integral.
\bigskip

Next, by
$\int_{|\zeta_1|=+\infty}$ above, we mean 
$\lim_{R\rightarrow+\infty}\int_{|\zeta_1|=R}$ 
that exists by the residue theorem
(it is also 
$\int_{|\zeta_1|=R}$ for
$R$ large enough). If
$k_1<m_1+\cdots+m_{n-1}$ then
\begin{eqnarray*}
\deg_{\zeta_1}
\left(\frac{\zeta_1^{k_1-m_1+1}}{\prod_{j=2}^{n-1}
(\zeta_1-\eta_j\zeta_2)^{m_j}\;\left(\zeta_1-\frac{r^2z_1\zeta_2}{\zeta_2-(1-r^2z_2)}\right)^2}\right)
& \leq & -2
\end{eqnarray*}
and 
\begin{eqnarray*}
\frac{1}{2\pi i}\int_{|\zeta_1|=+\infty}
\frac{\zeta_1^{k_1-m_1+1}\;d\zeta_1}{\prod_{j=2}^{n-1}
(\zeta_1-\eta_j\zeta_2)^{m_j}\;\left(\zeta_1-\frac{r^2z_1\zeta_2}{\zeta_2-(1-r^2z_2)}\right)^2}
& = & 0\,.
\end{eqnarray*}

Now if $k_1\geq m_1+\cdots+m_{n-1}$, in particular
$k_1-m_1+1\geq0$ and the above integral is
\begin{eqnarray*}\nonumber
& & 
\frac{\partial}{\partial\zeta_1}|_{\zeta_1=\frac{r^2z_1\zeta_2}{\zeta_2-(1-r^2)z_2}}
\left[
\frac{\zeta_1^{k_1-m_1+1}}{\prod_{j=2}^{n-1}(\zeta_1-\eta_j\zeta_2)^{m_j}}
\right]\\\nonumber
& & +\;
\sum_{p=2}^{n-1}
\frac{1}{(m_p-1)!}\frac{\partial^{m_p-1}}{\partial\zeta_1^{m_p-1}}|_{\zeta_1=\eta_p\zeta_2}
\left[\frac{\zeta_1^{k-m_1+1}}{\prod_{j=2,j\neq p}^{n-1}(\zeta_1-\eta_j\zeta_2)^{m_j}
(\zeta_1-\frac{r^2z_1\zeta_2}{\zeta_2-(1-r^2)z_2})^2}
\right]
=\\\nonumber
\\\nonumber
& = &
\frac{1}{\zeta_2}
\frac{\partial}{\partial x}|_{x=\frac{r^2z_1}{\zeta_2-(1-r^2)z_2}}
\left[
\frac{(x\zeta_2)^{k_1-m_1+1}}{\prod_{j=2}^{n-1}(x\zeta_2-\eta_j\zeta_2)^{m_j}}
\right]\\\nonumber
& & +\;
\sum_{p=2}^{n-1}
\frac{1}{(m_p-1)!}\frac{1}{\zeta_2^{m_p-1}}\frac{\partial^{m_p-1}}{\partial t^{m_p-1}}|_{t=\eta_p}
\left[\frac{(t\zeta_2)^{k-m_1+1}}{\prod_{j=2,j\neq p}^{n-1}(t\zeta_2-\eta_j\zeta_2)^{m_j}
(t\zeta_2-\frac{r^2z_1\zeta_2}{\zeta_2-(1-r^2)z_2})^2}
\right]\nonumber
\end{eqnarray*}
\begin{eqnarray}\label{infini2}
=\;\zeta_2^{k_1-(m_1+\cdots+m_{n-1})}
\frac{\partial}{\partial x}|_{x=\frac{r^2z_1}{\zeta_2-(1-r^2)z_2}}
\left[
\frac{x^{k_1-m_1+1}}{\prod_{j=2}^{n-1}(x-\eta_j)^{m_j}}
\right]
\;\;\;\;\;\;\;\;\;\;\;\;\;\;
\end{eqnarray}
\begin{eqnarray*}
& &
+\;\;
\zeta_2^{k_1-(m_1+\cdots+m_{n-1})}
\sum_{p=2}^{n-1}
\frac{1}{(m_p-1)!}\frac{\partial^{m_p-1}}{\partial t^{m_p-1}}|_{t=\eta_p}
\left[\frac{t^{k-m_1+1}}{\prod_{j=2,j\neq p}^{n-1}(t-\eta_j)^{m_j}
(t-\frac{r^2z_1}{\zeta_2-(1-r^2)z_2})^2}
\right]\\
\\
& = & 
\zeta_2^{k_1-(m_1+\cdots+m_{n-1})}\;\times
\;\;\;\;\;\;\;\;\;\;\;\;\;\;\;\;\;\;\;\;\;\;\;\;\;\;\;\;\;\;\;\;\;\;\;\;\;\;\;\;
\;\;\;\;\;\;\;\;\;\;\;\;\;\;\;\;\;\;\;\;\;\;\;\;\;\;\;\;\;\;\;\;\;\;\;\;\;\;\;\;\;\;\;\;
\end{eqnarray*}
\begin{eqnarray*}
\frac{\partial}{\partial x}|_{x=\frac{r^2z_1}{\zeta_2-(1-r^2)z_2}}
\left[
\frac{x^{k_1-m_1+1}}{\prod_{j=2}^{n-1}(x-\eta_j)^{m_j}}
-\sum_{p=2}^{n-1}
\frac{1}{(m_p-1)!}\frac{\partial^{m_p-1}}{\partial t^{m_p-1}}|_{t=\eta_p}
\left(\frac{t^{k-m_1+1}}{\prod_{j=2,j\neq p}^{n-1}(t-\eta_j)^{m_j}(x-t)}
\right)
\right]\,.\nonumber
\end{eqnarray*}

\bigskip

Now by lemma \ref{lagrangeuclide} 
and (\ref{preliminfini}),
we have

\begin{eqnarray*}
\frac{X^{k_1-m_1+1}}{\prod_{j=2}^{n-1}(X-\eta_j)^{m_j}}
-\sum_{p=2}^{n-1}
\frac{1}{(m_p-1)!}\frac{\partial^{m_p-1}}{\partial t^{m_p-1}}|_{t=\eta_p}
\left(\frac{t^{k-m_1+1}}{\prod_{j=2,j\neq p}^{n-1}(t-\eta_j)^{m_j}(X-t)}
\right)
& = &
\end{eqnarray*}
\begin{eqnarray*}
\;\;\;\;\;\;\;\;\;\;\;\;\;\;\;\;\;\;\;\;\;\;
& = &
Q\left(X^{k_1-m_1+1},\prod_{j=2}^{n-1}(X-\eta_j)^{m_j}\right)\,,
\end{eqnarray*}
where $Q$ (resp. $R$) is the quotient (resp. remainder)
of the Euclidean division of 
$X^{k_1-m_1+1}$ by
$\prod_{j=2}^{n-1}(X-\eta_j)^{m_j}$. 
It follows that this is a polynomial, as well as
\begin{eqnarray*}
\frac{\partial Q}{\partial X}|_{X=\frac{r^2z_1}{\zeta_2-(1-r^2)z_2}}
\end{eqnarray*}
and this proves the lemma.

Notice that it is true as long as 
$\dfrac{r^2z_1}{\zeta_2-(1-r^2)z_2}\neq\eta_p,
\,\forall\,p=2,\ldots,n-1$
then as soon as
$|\zeta_2|^2=1-r^2\neq|(1-r^2)z_2+r^2z_1/\eta_p|^2$.
The lemma is proved for all 
$(r,z)$ in a dense open set of
$[0,1]\times\mathbb{B}_2$ then
for all
$r\neq\alpha_p,\,p=1,\ldots,n$ since
the functions that appear in the statement
are continuous with respect to
$r$ and $z$.

\end{proof}

\bigskip

Now we can give the proof of proposition~\ref{remainder}.

\begin{proof}

We have to calculate for all $k_1,\,k_2\geq0$ :
\begin{eqnarray}\label{infini0}
-\,\lim_{\varepsilon\rightarrow0}\frac{1}{(2\pi i)^2}
\int_{\Sigma_{\varepsilon}}
\frac{\zeta_1^{k_1}\zeta_2^{k_2}\;\omega'\left(\overline{\zeta}\right)\wedge\omega(\zeta)}
{g_n(\zeta)\;\left(1-\overline{\zeta_1}z_1-\overline{\zeta_2}z_2\right)^2}
& = &
\;\;\;\;\;\;\;\;\;\;\;\;\;\;\;\;\;\;\;\;\;\;\;\;\;\;\;\;\;\;\;
\end{eqnarray}
\begin{eqnarray*}
& = &
\lim_{\varepsilon\rightarrow0}
\sum_{l=1}^{\widetilde{n}-1}
\int_{\alpha_{q_l}+\varepsilon}^{\alpha_{q_{l+1}}-\varepsilon}2rdr
\;\times
\\
& &
\frac{1}{2\pi i}\int_{|\zeta_2|=\sqrt{1-r^2}}
\frac{\zeta_2^{k_2-m_n-1}d\zeta_2}{(1-\overline{\zeta_2}z_2)^2}
\frac{1}{2\pi i}
\int_{|\zeta_1|=r}
\frac{\zeta_1^{k_1-m_1-1}d\zeta_1}
{\prod_{j=2}^{n-1}(\zeta_1-\eta_j\zeta_2)^{m_j}(1-\frac{\overline{\zeta_1}z_1}{1-\overline{\zeta_2}z_2})^2}\,,
\end{eqnarray*}
since
\begin{eqnarray*}
\omega'\left(\overline{\zeta}\right)\wedge\omega(\zeta)
& = &
-2rdr\wedge\frac{d\zeta_1}{\zeta_1}\wedge\frac{d\zeta_2}{\zeta_2}\,
\end{eqnarray*}
with the following parametrization of
$\zeta\in\mathbb{S}_2$ :
\begin{eqnarray*}
\begin{cases}
\zeta_1=re^{i\theta_1},\;0\leq\theta_1<2\pi,\\
\zeta_2=\sqrt{1-r^2}\,e^{i\theta_2},\;0\leq\theta_2<2\pi,\\
\end{cases}
& &
0<r<1\,.
\end{eqnarray*}
\bigskip

Now fix $l=1,\ldots,\widetilde{n}-1$ and $\alpha_{q_l}<r<\alpha_{q_{l+1}}$. Then for all
$|\zeta_1|=r$ and all $|\zeta_2|=\sqrt{1-r^2}$, one has
$|\eta_{q_l}|^2<r^2(1+|\eta_{q_l}|^2)$ thus 
$|\eta_{q_l}|\sqrt{1-r^2}<r$ (similarly
$|\eta_{q_{l+1}}|\sqrt{1-r^2}>r$). It follows that
\begin{eqnarray}
\begin{cases}
|\eta_2\zeta_2|\leq\cdots\leq|\eta_{q_l}\zeta_2|<|\zeta_1|\,,\\
|\zeta_1|<|\eta_{q_l+1}\zeta_2|=\cdots=|\eta_{q_{l+1}}\zeta_2|\leq\cdots\leq|\eta_{n-1}\zeta_2|\,.
\end{cases}
\end{eqnarray}
This yields to
\begin{eqnarray*}
\frac{1}{2\pi i}
\int_{|\zeta_1|=r}
\frac{\zeta_1^{k_1-m_1+1}\;d\zeta_1}
{\prod_{j=2}^{n-1}(\zeta_1-\eta_j\zeta_2)^{m_j}(\zeta_1-\frac{r^2z_1\zeta_2}{\zeta_2-(1-r^2)z_2})^2}
& = &
\;\;\;\;\;\;\;\;\;\;\;\;\;\;\;\;\;\;\;\;\;\;\;\;\;\;\;\;\;\;\;\;\;\;\;\;\;\;\;\;\;\;
\end{eqnarray*}
\begin{eqnarray*}
& = &
\frac{1}{2\pi i}
\int_{|\zeta_1|=+\infty}
\frac{\zeta_1^{k_1-m_1+1}\;d\zeta_1}
{\prod_{j=2}^{n-1}(\zeta_1-\eta_j\zeta_2)^{m_j}(\zeta_1-\frac{r^2z_1\zeta_2}{\zeta_2-(1-r^2)z_2})^2}\\
& &
-\,\sum_{p=q_l+1}^{n-1}
\frac{1}{(m_p-1)!}
\frac{\partial^{m_p-1}}{\partial\zeta_1^{m_p-1}}|_{\zeta_1=\eta_p\zeta_2}
\left(
\frac{\zeta_1^{k_1-m_1+1}}
{\prod_{j=2,j\neq p}^{n-1}(\zeta_1-\eta_j\zeta_2)^{m_j}(\zeta_1-\frac{r^2z_1\zeta_2}{\zeta_2-(1-r^2)z_2})^2}
\right)\\
\\
& = &
{\bf{1}}_{k_1\geq m_1+\cdots+m_{n-1}}\,\zeta_2^{k_1-(m_1+\cdots+m_{n-1})}
P\left(\frac{r^2z_1}{\zeta_2-(1-r^2)z_2}\right)\\
& - &
\zeta_2^{k_1-(m_1+\cdots+m_{n-1})}
\sum_{p=q_l+1}^{n-1}
\frac{1}{(m_p-1)!}
\frac{\partial^{m_p-1}}{\partial t^{m_p-1}}|_{t=\eta_p}
\left(
\frac{t^{k_1-m_1+1}}
{\prod_{j=2,j\neq p}^{n-1}(t-\eta_j)^{m_j}(t-\frac{r^2z_1}{\zeta_2-(1-r^2)z_2})^2}
\right)
\end{eqnarray*}
by lemma~\ref{infini}.
\bigskip

One can deduce that
\begin{eqnarray*}
\frac{1}{2\pi i}\int_{|\zeta_2|=\sqrt{1-r^2}}
\frac{\zeta_2^{k_2-m_n+1}d\zeta_2}{(\zeta_2-(1-r^2)z_2)^2}
\frac{1}{2\pi i}
\int_{|\zeta_1|=r}
\frac{\zeta_1^{k_1-m_1+1}d\zeta_1}
{\prod_{j=2}^{n-1}(\zeta_1-\eta_j\zeta_2)^{m_j}(\zeta_1-\frac{r^2z_1\zeta_2}{\zeta_2-(1-r^2)z_2})^2}
& = &
\end{eqnarray*}
\begin{eqnarray}\label{infini3}
& = &
{\bf{1}}_{k_1\geq m_1+\cdots+m_{n-1}}\,
\frac{1}{2\pi i}\int_{|\zeta_2|=\sqrt{1-r^2}}
\frac{\zeta_2^{k_1+k_2-N+1}P(\frac{r^2z_1}{\zeta_2-(1-r^2)z_2})}
{(\zeta_2-(1-r^2)z_2)^2}
\,d\zeta_2
\end{eqnarray}
\begin{eqnarray*}
-\sum_{p=q_l+1}^{n-1}
\frac{1}{(m_p-1)!}
\frac{\partial^{m_p-1}}{\partial t^{m_p-1}}|_{t=\eta_p}
\left[
\frac{t^{k_1-m_1+1}}
{\prod_{j=2,j\neq p}^{n-1}(t-\eta_j)^{m_j}}
\frac{1}{2\pi i}\int_{|\zeta_2|=\sqrt{1-r^2}}
\frac{\zeta_2^{k_1+k_2-N+1}\;d\zeta_2}{(t(\zeta_2-(1-r^2)z_2)-r^2z_1)^2}
\right].\nonumber
\end{eqnarray*}

\bigskip

Now we see in the first integral that the only possible residues are
$\zeta_2=(1-r^2)z_2$ and $\zeta_2=0$. If
$k_1+k_2<N$ (and $P\neq0$ otherwise the integral is zero) then 
\begin{eqnarray*}
\deg\left(\frac{\zeta_2^{k_1+k_2-N+1}}{(\zeta_2-(1-r^2)z_2)^2}
P\left(\frac{r^2z_1}{\zeta_2-(1-r^2)z_2}\right)\right)
& \leq & -2
\end{eqnarray*}
and
\begin{eqnarray*}
\frac{1}{2\pi i}\int_{|\zeta_2|=\sqrt{1-r^2}}
\frac{\zeta_2^{k_1+k_2-N+1}P(\frac{r^2z_1}{\zeta_2-(1-r^2)z_2})d\zeta_2}
{(\zeta_2-(1-r^2)z_2)^2}
=\frac{1}{2\pi i}\int_{|\zeta_2|=+\infty}
\frac{\zeta_2^{k_1+k_2-N+1}P(\frac{r^2z_1}{\zeta_2-(1-r^2)z_2})d\zeta_2}
{(\zeta_2-(1-r^2)z_2)^2}
=0\,.
\end{eqnarray*}
\bigskip

If $k_1+k_2\geq N$ then the only residue is
$\zeta_2=(1-r^2)z_2$
and we get by (\ref{infini2})
\begin{eqnarray*}
\frac{1}{2\pi i}\int_{|\zeta_2|=\sqrt{1-r^2}}
\frac{\zeta_2^{k_1+k_2-N+1}P(\frac{r^2z_1}{\zeta_2-(1-r^2)z_2})}
{(\zeta_2-(1-r^2)z_2)^2}\,d\zeta_2
& = &
\;\;\;\;\;\;\;\;\;\;\;\;\;\;\;\;\;\;\;\;\;\;\;\;\;\;\;\;\;\;\;\;\;\;\;\;\;\;\;\;\;\;\;\;\;\;\;\;\;
\end{eqnarray*}
\begin{eqnarray*}
& = &
\lim_{\varepsilon'\rightarrow0}
\frac{1}{2\pi i}
\int_{|\zeta_2-(1-r^2)z_2|=\varepsilon'}
\frac{\zeta_2^{k_1+k_2-N+1}P(\frac{r^2z_1}{\zeta_2-(1-r^2)z_2})}
{(\zeta_2-(1-r^2)z_2)^2}\,d\zeta_2\,,
\\
\\
& = &
\lim_{\varepsilon'\rightarrow0}
\frac{1}{2\pi i}
\int_{|\zeta_2-(1-r^2)z_2|=\varepsilon'}
\frac{\zeta_2^{k_1+k_2-N+1}}{(\zeta_2-(1-r^2)z_2)^2}
\frac{\partial}{\partial x}|_{x=\frac{r^2z_1}{\zeta_2-(1-r^2)z_2}}
\left[
\frac{x^{k_1-m_1+1}}{\prod_{j=2}^{n-1}(x-\eta_j)^{m_j}}
\right]
\,d\zeta_2\\
& &
+\;
\lim_{\varepsilon'\rightarrow0}
\frac{1}{2\pi i}
\int_{|\zeta_2-(1-r^2)z_2|=\varepsilon'}
\frac{\zeta_2^{k_1+k_2-N+1}}{(\zeta_2-(1-r^2)z_2)^2}\,d\zeta_2
\;\times\\
& & 
\;\;\;\;\;\;\;\;\;
\times\;
\sum_{p=2}^{n-1}
\frac{1}{(m_p-1)!}\frac{\partial^{m_p-1}}{\partial t^{m_p-1}}|_{t=\eta_p}
\left[\frac{t^{k_1-m_1+1}}{\prod_{j=2,j\neq p}^{n-1}(t-\eta_j)^{m_j}
(t-\frac{r^2z_1}{\zeta_2-(1-r^2)z_2})^2}
\right]\\
\\
& = &
\lim_{\varepsilon'\rightarrow0}
\;(r^2z_1)^{k_1-m_1}\;
\frac{\partial}{\partial t}|_{t=1}
\;t^{k_1-m_1+1}
\;\times
\end{eqnarray*}
\begin{eqnarray*}
\frac{1}{2\pi i}
\int_{|\zeta_2-(1-r^2)z_2|=\varepsilon'}
\frac{\zeta_2^{k_1+k_2-N+1}\;d\zeta_2}
{(\zeta_2-(1-r^2)z_2)^{k_1-(m_1+\cdots+m_{n-1})+2}
\prod_{j=2}^{n-1}(tr^2z_1-\eta_j(\zeta_2-(1-r^2)z_2))^{m_j}}\\
\end{eqnarray*}
\begin{eqnarray*}
& + &
\lim_{\varepsilon'\rightarrow0}
\sum_{p=2}^{n-1}
\frac{1}{(m_p-1)!}
\;\times
\\
& &
\frac{\partial^{m_p-1}}{\partial t^{m_p-1}}|_{t=\eta_p}
\left[
\frac{t^{k_1-m_1+1}}{\prod_{j=2,j\neq p}^{n-1}(t-\eta_j)^{m_j}}
\frac{1}{2\pi i}
\int_{|\zeta_2-(1-r^2)z_2|=\varepsilon'}
\frac{\zeta_2^{k_1+k_2-N+1}\;d\zeta_2}{(t(\zeta_2-(1-r^2)z_2)-r^2z_1)^2}
\right]
\\
\end{eqnarray*}
(one can switch integral and derivative for all
fixed $z$ and all
$\varepsilon'$ small enough since the above functions,
as well as all
their derivatives with respect to $t$,
are integrable).
\bigskip

Now $z_1,\,r,\,t\neq0$ being fixed (since
$z\in U_{\eta}$), one can choose for $\zeta_2$
a small enough neighborhood of
$(1-r^2)z_2$ such that the function
\begin{eqnarray*}
\zeta_2 & \mapsto &
\frac{\zeta_2^{k_1+k_2-N+1}}{(t(\zeta_2-(1-r^2)z_2)-r^2z_1)^2}
\end{eqnarray*}
has no singularity. It follows that
\begin{eqnarray*}
\frac{1}{2\pi i}\int_{|\zeta_2|=\sqrt{1-r^2}}
\frac{\zeta_2^{k_1+k_2-N+1}P(\frac{r^2z_1}{\zeta_2-(1-r^2)z_2})}
{(\zeta_2-(1-r^2)z_2)^2}\,d\zeta_2
& = &
\;\;\;\;\;\;\;\;\;\;\;\;\;\;\;\;\;\;\;\;\;\;\;\;\;\;\;\;\;\;\;\;\;\;\;\;\;\;\;\;\;\;\;\;\;\;\;\;\;
\end{eqnarray*}
\begin{eqnarray*}
& = &
(r^2z_1)^{k_1-m_1}\;
\frac{\partial}{\partial t}|_{t=1}
\;t^{k_1-m_1+1}
\frac{1}{(k_1-(m_1+\cdots+m_{n-1})+1)!}
\;\times
\\
& &
\times\;
\frac{\partial^{k_1-(m_1+\cdots+m_{n-1})+1}}{\partial\zeta_2^{k_1-(m_1+\cdots+m_{n-1})+1}}|_{\zeta_2=(1-r^2)z_2}
\left[
\frac{\zeta_2^{k_1+k_2-N+1}}{\prod_{j=2}^{n-1}(tr^2z_1-\eta_j(\zeta_2-(1-r^2)z_2))^{m_j}}
\right]
\end{eqnarray*}
\begin{eqnarray*}
& = &
(r^2z_1)^{k_1-m_1}\;
\frac{\partial}{\partial t}|_{t=1}
\;t^{k_1-m_1+1}
\;\times
\\
& &
\times\;
\sum_{v_1+\cdots+v_{n-1}=k_1-(m_1+\cdots+m_{n-1})+1,\,v_1\leq k_1+k_2-N+1}
\frac{(k_1+k_2-N+1)!}{v_1!\,(k_1+k_2-N+1-v_1)!}
\;\times
\\
& &
\times\;
((1-r^2)z_2)^{k_1+k_2-N+1-v_1}\;
\prod_{j=2}^{n-1}
\;\frac{(v_j+m_j-1)!}{v_j!\,(m_j-1)!}
\;\frac{\eta_j^{v_j}}{(tr^2z_1)^{v_j+m_j}}
\\
\end{eqnarray*}
\begin{eqnarray*}
& = &
\sum_{v_1+\cdots+v_{n-1}=k_1-(m_1+\cdots+m_{n-1})+1,\,v_1\leq k_1+k_2-N+1}
\frac{(k_1+k_2-N+1)!}{v_1!\,(k_1+k_2-N+1-v_1)!}
\;\times\\
\\
& & 
\times\;
(r^2z_1)^{k_1-m_1-(v_2+m_2+\cdots+v_{n-1}+m_{n-1})}
((1-r^2)z_2)^{k_1+k_2-N+1-v_1}
\;\times
\\
& &
\times\;
\prod_{j=2}^{n-1}
\;\frac{(v_j+m_j-1)!}{v_j!\,(m_j-1)!}
\;\eta_j^{v_j}
\;\frac{d}{dt}|_{t=1}
\left(t^{k_1-m_1+1-(v_2+m_2+\cdots+v_{n-1}+m_{n-1})}\right)
\end{eqnarray*}
\begin{eqnarray*}
& = &
\sum_{v_1+\cdots+v_{n-1}=k_1-(m_1+\cdots+m_{n-1})+1,
\,v_1\leq k_1+k_2-N+1,\,v_1\geq1}
\frac{(k_1+k_2-N+1)!}{v_1!\,(k_1+k_2-N+1-v_1)!}
\;\times
\\
& &
\times\;
(r^2z_1)^{v_1-1}
((1-r^2)z_2)^{k_1+k_2-N+1-v_1}
\;\prod_{j=2}^{n-1}
\;\frac{(v_j+m_j-1)!}{v_j!\,(m_j-1)!}
\;\eta_j^{v_j}
\;v_1
\end{eqnarray*}
\begin{eqnarray}\label{infini3a}
& &
=
\sum_{v_1+\cdots+v_{n-1}=k_1-(m_1+\cdots+m_{n-1}),\,v_1\leq k_1+k_2-N}
\frac{(k_1+k_2-N+1)!}{v_1!\,(k_1+k_2-N-v_1)!}
\;\times
\\\nonumber
& &
\times\;
(r^2z_1)^{v_1}
((1-r^2)z_2)^{k_1+k_2-N-v_1}
\;\prod_{j=2}^{n-1}
\;\frac{(v_j+m_j-1)!}{v_j!\,(m_j-1)!}
\;\eta_j^{v_j}\,.
\;\;\;\;\;\;\;\;\;\;\;\;\;\;\;\;\;\;\;\;\;\;\;\;
\end{eqnarray}

\bigskip

Now consider the other part of (\ref{infini3}). In order to calculate the following integral, for all
$t$ close to $\eta_p,\;p=q_l+1,\ldots,n-1$,
\begin{eqnarray*}
\frac{1}{2\pi i}\int_{|\zeta_2|=\sqrt{1-r^2}}
\frac{\zeta_2^{k_1+k_2-N+1}\;d\zeta_2}{(\zeta_2-(1-r^2)z_2-r^2z_1/t)^2}
\;,
\end{eqnarray*}
the following lemma will be useful.
\bigskip

\begin{lemma}\label{residup}

Let $K\subset\mathbb{B}_2$ 
be a compact subset.
There exists
$\varepsilon_K>0$ such that,
for all $z\in K$,
for all $p=q_{l}+1,\ldots,n-1$, 
for all $t$ close to $\eta_p$
and all
$r<\alpha_{q_{l+1}}$,
\begin{eqnarray}\label{residup1}
\left|\left(1-r^2\right)z_2+\frac{r^2z_1}{t}\right|
& \leq &
\left(1-\varepsilon_K^2\right)
\sqrt{1-r^2}\,.
\end{eqnarray}
In particular, for all
$z\in\mathbb{B}_2$,
for all
$t$ close to $\eta_p$ and
all
$r<\alpha_{q_{l+1}}$,
one has
\begin{eqnarray*}
\left|\left(1-r^2\right)z_2+\frac{r^2z_1}{t}\right|
& < &
\sqrt{1-r^2}\,.
\\
\end{eqnarray*}

These assertions are still true for
all $p=2,\ldots,n-1$,
for all
$t$ close enough to $\eta_p$
and $r$ to $\alpha_p$.

%

\end{lemma}

\bigskip

\begin{proof}

One has by the Cauchy-Schwarz inequality
\begin{eqnarray*}
|(1-r^2)z_2+r^2z_1/t|
& \leq &
\|z\|
\,\sqrt{(1-r^2)^2+\frac{r^4}{|t|^2}}
\;\leq\;
\|z\|\,
\sqrt{1-2r^2+r^4\left(1+\frac{1}{|t|^2}\right)}\,.
\end{eqnarray*}
Since $z\in K\subset\mathbb{B}_2$, 
there is $\varepsilon_K>0$ such that
$\|z\|\leq1-\varepsilon_K$. It follows that, for all 
$t$ close enough to $\eta_p$,
\begin{eqnarray*}
|(1-r^2)z_2+r^2z_1/t|
& \leq &
(1-\varepsilon_K)\;(1+\varepsilon_K)
\sqrt{1-2r^2+\frac{r^4}{\alpha_p^2}}\\
& \leq &
(1-\varepsilon_K^2)
\,\sqrt{1-2r^2+\frac{r^4}{\alpha_{q_{l+1}}^2}}\\
& < &
\sqrt{1-r^2}\;.
\\
\end{eqnarray*}

Similarly, 
since
\begin{eqnarray*}
\lim_{t\rightarrow\eta_p,r\rightarrow\alpha_p}
\sqrt{1-2r^2+r^4\left(1+\frac{1}{|t|^2}\right)}
\;=\;
\sqrt{1-\alpha_p^2}
\;=\;
\lim_{r\rightarrow\alpha_p}\sqrt{1-r^2}\,,
\end{eqnarray*}
one has, for all
$t$ close enough to $\eta_p$
and $r$ to $\alpha_p$,
\begin{eqnarray*}
|(1-r^2)z_2+r^2z_1/t|
& \leq &
(1-\varepsilon_K)
\sqrt{1-2r^2+r^4\left(1+\frac{1}{|t|^2}\right)}
\\
& \leq &
(1-\varepsilon_K)^2
\sqrt{1-r^2}\,.
\\
\end{eqnarray*}

\end{proof}

\bigskip

In particular, 
$\zeta_2=(1-r^2)z_2+r^2z_1/t$ is residue in the above integral. If
$k_1+k_2<N$ then
\begin{eqnarray*}
\frac{1}{2\pi i}\int_{|\zeta_2|=\sqrt{1-r^2}}
\frac{\zeta_2^{k_1+k_2-N+1}\;d\zeta_2}{(\zeta_2-(1-r^2)z_2-r^2z_1/t)^2}
=
\frac{1}{2\pi i}\int_{|\zeta_2|=+\infty}
\frac{\zeta_2^{k_1+k_2-N+1}\;d\zeta_2}{(\zeta_2-(1-r^2)z_2-r^2z_1/t)^2}
=0\,.
\\
\end{eqnarray*}

\bigskip

If $k_1+k_2\geq N$ one has
\begin{eqnarray*}
\frac{1}{2\pi i}\int_{|\zeta_2|=\sqrt{1-r^2}}
\frac{\zeta_2^{k_1+k_2-N+1}\;d\zeta_2}{(\zeta_2-(1-r^2)z_2-r^2z_1/t)^2}
=
(k_1+k_2-N+1)
((1-r^2)z_2+r^2z_1/t)^{k_1+k_2-N}\,.
\end{eqnarray*}

\bigskip

It follows by (\ref{infini3}) and (\ref{infini3a})
that, for all
$k_1,\,k_2\geq0$,
\begin{eqnarray*}
\frac{1}{2\pi i}\int_{|\zeta_2|=\sqrt{1-r^2}}
\frac{\zeta_2^{k_2-m_n+1}d\zeta_2}{(\zeta_2-(1-r^2)z_2)^2}
\frac{1}{2\pi i}
\int_{|\zeta_1|=r}
\frac{\zeta_1^{k_1-m_1+1}d\zeta_1}
{\prod_{j=2}^{n-1}(\zeta_1-\eta_j\zeta_2)^{m_j}(\zeta_1-\frac{r^2z_1\zeta_2}{\zeta_2-(1-r^2)z_2})^2}
& = &
\end{eqnarray*}
\begin{eqnarray*}
& = &
{\bf 1}_{k_1+k_2\geq N}
\sum_{v_1+\cdots+v_{n-1}=k_1-(m_1+\cdots+m_{n-1}),
\,v_1\leq k_1+k_2-N}
\frac{(k_1+k_2-N+1)!}{v_1!\,(k_1+k_2-N-v_1)!}
\;\times
\\
& &
\times\;
(r^2z_1)^{v_1}
((1-r^2)z_2)^{k_1+k_2-N-v_1}
\;\prod_{j=2}^{n-1}
\;\frac{(v_j+m_j-1)!}{v_j!\,(m_j-1)!}
\;\eta_j^{v_j}
\\
& - &
{\bf 1}_{k_1+k_2\geq N}
\sum_{p=q_l+1}^{n-1}
\frac{1}{(m_p-1)!}
\;\times
\\
& &
\times\;
\frac{\partial^{m_p-1}}{\partial t^{m_p-1}}|_{t=\eta_p}
\left(
\frac{(k_1+k_2-N+1)t^{k_1-m_1-1}
((1-r^2)z_2+r^2z_1/t)^{k_1+k_2-N}}
{\prod_{j=2,j\neq p}^{n-1}(t-\eta_j)^{m_j}}
\right)
\end{eqnarray*}(notice that the condition 
${\bf 1}_{k_1\geq m_1+\cdots+m_{n-1}}$ in (\ref{infini3})
is satisfied since if
$k_1<m_1+\cdots+m_{n-1}$ then 
$\sum_{v_1+\cdots+v_{n-1}=k_1-(m_1+\cdots+m_{n-1})}=0$).
\bigskip

This is valid for all 
$l=1,\ldots,\widetilde{n}-1$ and all
$\alpha_{q_l}<r<\alpha_{q_{l+1}}$. 
It follows from (\ref{infini0})
that, for all $k_1+k_2\geq N$ (that we will assume
in the following since for
$k_1+k_2<N$ the propostition is proved),
\begin{eqnarray*}
-\,\lim_{\varepsilon\rightarrow0}\frac{1}{(2\pi i)^2}
\int_{\Sigma_{\varepsilon}}
\frac{\zeta_1^{k_1}\zeta_2^{k_2}\;\omega'\left(\overline{\zeta}\right)\wedge\omega(\zeta)}
{g_n(\zeta)\;\left(1-\overline{\zeta_1}z_1-\overline{\zeta_2}z_2\right)^2}
& = &
\;\;\;\;\;\;\;\;\;\;\;\;\;\;\;\;\;\;\;\;\;\;\;\;\;\;\;\;\;\;\;\;\;\;\;\;\;\;\;\;\;\;\;\;\;\;\;\;\;\;
\end{eqnarray*}
\begin{eqnarray*}
& = &
\sum_{v_1+\cdots+v_{n-1}=k_1-(m_1+\cdots+m_{n-1}),
\,v_1\leq k_1+k_2-N}
\frac{(k_1+k_2-N+1)!}{v_1!\,(k_1+k_2-N-v_1)!}
z_1^{v_1}
z_2^{k_1+k_2-N-v_1}
\\
& & 
\times\,
\prod_{j=2}^{n-1}
\;\frac{(v_j+m_j-1)!}{v_j!\,(m_j-1)!}
\;\eta_j^{v_j}
\lim_{\varepsilon\rightarrow0}
\;\sum_{l=1}^{\widetilde{n}-1}
\int_{\alpha_{q_l}+\varepsilon}^{\alpha_{q_{l+1}}-\varepsilon}
(r^2)^{v_1}(1-r^2)^{k_1+k_2-N-v_1}
\;2rdr\\
\\
& - & 
\lim_{\varepsilon\rightarrow0}
\;\sum_{l=1}^{\widetilde{n}-1}\sum_{p=q_l+1}^{n-1}
\int_{\alpha_{q_l}+\varepsilon}^{\alpha_{q_{l+1}}-\varepsilon}
\;2rdr
\;\times\\
& & 
\frac{1}{(m_p-1)!}
\frac{\partial^{m_p-1}}{\partial t^{m_p-1}}|_{t=\eta_p}
\left(
\frac{(k_1+k_2-N+1)t^{k_1-m_1-1}
((1-r^2)z_2+r^2z_1/t)^{k_1+k_2-N}}
{\prod_{j=2,j\neq p}^{n-1}(t-\eta_j)^{m_j}}
\right).
\end{eqnarray*}

\bigskip

Now notice that, in the last above sum, 
by (\ref{defq}) we have
$p\geq q_l+1$ if and only if 
$\alpha_p>\alpha_{q_l}$
that is equivalent to 
$\alpha_p\geq\alpha_{q_{l+1}}$,
so $l=1,\ldots,l_p-1$ where $l_p$ is defined such that  
$\alpha_{q_{l_p}}=\alpha_p$. This allows us to get
\begin{eqnarray*}
& &
\sum_{v_1+\cdots+v_{n-1}=k_1-(m_1+\cdots+m_{n-1}),
v_1\leq k_1+k_2-N}
z_1^{v_1}
z_2^{k_1+k_2-N-v_1}
\prod_{j=2}^{n-1}
\,\frac{(v_j+m_j-1)!}{v_j!\,(m_j-1)!}
\,\eta_j^{v_j}
\;\times
\\
& & 
\;\times\;
\frac{(k_1+k_2-N+1)!}{v_1!\,(k_1+k_2-N-v_1)!}
\;\int_0^1
x^{v_1}(1-x)^{k_1+k_2-N-v_1}
\;dx
\\
\\
& & 
-\;
\sum_{p=2}^{n-1}
\;\int_0^{\alpha_p}
\;2rdr
\;\frac{1}{(m_p-1)!}
\;\times\\
& & 
\;\;\times\;
\frac{\partial^{m_p-1}}{\partial t^{m_p-1}}|_{t=\eta_p}
\left(
\frac{(k_1+k_2-N+1)t^{k_1-m_1-1}
((1-r^2)z_2+r^2z_1/t)^{k_1+k_2-N}}
{\prod_{j=2,j\neq p}^{n-1}\;(t-\eta_j)^{m_j}}
\right)
\end{eqnarray*}
\begin{eqnarray}\label{infini4}
& &
=
\sum_{v_1+\cdots+v_{n-1}=k_1-(m_1+\cdots+m_{n-1}),
\,v_1\leq k_1+k_2-N}
z_1^{v_1}
\,z_2^{k_1+k_2-N-v_1}\;\times
\;\;\;\;\;\;\;\;\;\;\;\;
\\\nonumber
& &
\;\;
\times\;
\prod_{j=2}^{n-1}
\;\frac{(v_j+m_j-1)!}{v_j!\,(m_j-1)!}
\;\eta_j^{v_j}
\end{eqnarray}
\begin{eqnarray*}
-\;\sum_{p=2}^{n-1}
\frac{1}{(m_p-1)!}
\frac{\partial^{m_p-1}}{\partial t^{m_p-1}}|_{t=\eta_p}
\left[\frac{t^{k_1-m_1-1}
\left((z_2+\alpha_p^2(z_1/t-z_2))^{k_1+k_2-N+1}-z_2^{k_1+k_2-N+1}\right)}
{\prod_{j=2,j\neq p}^{n-1}\;(t-\eta_j)^{m_j}\;\;(z_1/t-z_2)}\right]\,,
\\
\end{eqnarray*}
the last equality coming from the following idendity (that can be proved by induction on
$v_1=0,\ldots,k_1+k_2-N+1$ with integrating by parts)
\begin{eqnarray*}
\int_0^1
x^{v_1}(1-x)^{k_1+k_2-N-v_1}
\;dx
& = &
\frac{v_1!\,(k_1+k_2-N-v_1)!}{(k_1+k_2-N+1)!}
\\
\end{eqnarray*}
(on the other hand, notice that, since
$z\in U_{\eta}$, then for all 
$p=2,\ldots,n-1$ and all
$t$ close enough to $\eta_p$,
$z_1/t-z_2\neq0$).
\bigskip

Now assume that
$k_1\leq m_1-1$. Then 
$k_1-(m_1+\cdots+m_{n-1})<0$ and 
$k_2\geq N-k_1\geq0$. We get from above
\begin{eqnarray*}
-\,\lim_{\varepsilon\rightarrow0}\frac{1}{(2\pi i)^2}
\;\int_{\Sigma_{\varepsilon}}
\frac{\zeta_1^{k_1}\zeta_2^{k_2}\;\omega'\left(\overline{\zeta}\right)\wedge\omega(\zeta)}
{g_n(\zeta)\;\left(1-<\overline{\zeta},z>\right)^2}
& = &
\;\;\;\;\;\;\;\;\;\;\;\;\;\;\;\;\;\;\;\;\;\;\;\;\;\;\;\;\;\;\;\;\;\;\;\;\;\;\;\;\;\;\;\;\;\;\;\;\;\;\;\;\;\;\;\;\;\;\;
\;\;\;\;
\end{eqnarray*}
\begin{eqnarray}\label{infini5c}
& = &
\sum_{p=2}^{n-1}
\frac{1}{(m_p-1)!}
\frac{\partial^{m_p-1}}{\partial t^{m_p-1}}|_{t=\eta_p}
\left[
\frac{t^{k_1}\;z_2^{k_1+k_2-N+1}}{(z_1-tz_2)\;t^{m_1}\,\prod_{j=2,j\neq p}^{n-1}\;(t-\eta_j)^{m_j}}
\right]
\;\;\;\;
\end{eqnarray}
\begin{eqnarray*}
& & 
\;\;\;\;
-\;
\sum_{p=2}^{n-1}
\frac{1}{(m_p-1)!}
\;\times
\\\nonumber
& &
\;\;\;\;
\frac{\partial^{m_p-1}}{\partial t^{m_p-1}}|_{t=\eta_p}
\left[
\frac{t^{k_1}}{(z_1-tz_2)\;t^{m_1}\,\prod_{j=2,j\neq p}^{n-1}\;(t-\eta_j)^{m_j}}
\left(\frac{z_2+|\eta_p|^2z_1/t}{1+|\eta_p|^2}\right)^{k_1+k_2-N+1}
\right]\,.\\
\end{eqnarray*}

\bigskip

Else, we have
$k_1\geq m_1$. Now assume that
$k_2\geq m_n$ then 
$k_1+k_2-N\geq k_1-(m_1+\cdots+m_{n-1})$ and 
we get
\begin{eqnarray*}
\sum_{v_1+\cdots+v_{n-1}=k_1-(m_1+\cdots+m_{n-1})}
z_1^{v_1}
\,z_2^{k_1+k_2-N-v_1}
\;\prod_{j=2}^{n-1}
\;\frac{(v_j+m_j-1)!}{v_j!\,(m_j-1)!}
\;\eta_j^{v_j}
& = &
\;\;\;\;\;\;\;\;\;\;\;\;\;\;\;\;\;\;\;\;\;\;
\end{eqnarray*}
\begin{eqnarray*}
& = &
z_2^{k_1+k_2-N}
\;\sum_{u=0}^{k_1-m_1-(m_2+\cdots+m_{n-1})}
(z_1/z_2)^{k_1-m_1-(m_2+\cdots+m_{n-1})-u}
\;\times
\\
& &
\times\;
\sum_{v_2+\cdots+v_{n-1}=u}
\;\prod_{j=2}^{n-1}
\;\frac{(v_j+m_j-1)!}{v_j!\,(m_j-1)!}
\;\eta_j^{v_j}
\\
& = &
z_2^{k_1+k_2-N}\;
Q\left(X^{k_1-m_1},\prod_{j=2}^{n-1}(X-\eta_j)^{m_j}\right)
|_{X\;=\;z_1/z_2}
\end{eqnarray*}
by lemma \ref{lagrangeuclide}. On the other hand, we have
\begin{eqnarray*}
\sum_{p=2}^{n-1}
\frac{1}{(m_p-1)!}\;
\frac{\partial^{m_p-1}}{\partial t^{m_p-1}}|_{t=\eta_p}
\left[\frac{t^{k_1-m_1}}
{(z_1/z_2-t)\;\prod_{j=2,j\neq p}^{n-1}\;(t-\eta_j)^{m_j}}\right]
& = &
\;\;\;\;\;\;\;\;\;\;\;\;\;\;\;\;\;\;\;
\end{eqnarray*}
\begin{eqnarray*}
\;\;\;\;\;\;\;\;\;\;\;\;\;\;\;\;\;\;\;\;\;\;\;\;\;\;\;\;\;\;\;\;\;\;\;\;\;\;\;\;\;\;\;\;
& = &
\left[
\frac{R\left(X^{k_1-m_1},\prod_{j=2}^{n-1}(X-\eta_j)^{m_j}\right)}
{\prod_{j=2}^{n-1}(X-\eta_j)^{m_j}}
\right]
|_{X\;=\;z_1/z_2}.
\end{eqnarray*}
It follows by (\ref{infini4}) that
\begin{eqnarray*}
-\,\lim_{\varepsilon\rightarrow0}\frac{1}{(2\pi i)^2}
\;\int_{\Sigma_{\varepsilon}}
\frac{\zeta_1^{k_1}\zeta_2^{k_2}\;\omega'\left(\overline{\zeta}\right)\wedge\omega(\zeta)}
{g_n(\zeta)\;\left(1-<\overline{\zeta},z>\right)^2}
& = &
z_2^{k_1+k_2-N}
\frac{(z_1/z_2)^{k_1-m_1}}{\prod_{j=2}^{n-1}(z_1/z_2-\eta_j)^{m_j}}
\end{eqnarray*}
\begin{eqnarray*}
\;\;\;\;\;\;\;\;\;\;\;\;\;\;\;
& & 
-\;\sum_{p=2}^{n-1}
\frac{1}{(m_p-1)!}
\frac{\partial^{m_p-1}}{\partial t^{m_p-1}}|_{t=\eta_p}
\left[
\frac{t^{k_1-m_1}(z_2+\alpha_p^2(z_1/t-z_2))^{k_1+k_2-N+1}}{(z_1-tz_2)\;\prod_{j=2,j\neq p}^{n-1}\;(t-\eta_j)^{m_j}}
\right]\\
\end{eqnarray*}
\begin{eqnarray}\label{infini5a}
& &
\;\;\;\;
=\;
\frac{z_1^{k_1}z_2^{k_2}}{g_n(z)}
\;-\;\sum_{p=2}^{n-1}
\frac{1}{(m_p-1)!}
\;\times
\\\nonumber
& &
\frac{\partial^{m_p-1}}{\partial t^{m_p-1}}|_{t=\eta_p}
\left[
\frac{t^{k_1-m_1}}{(z_1-tz_2)\;\prod_{j=2,j\neq p}^{n-1}\;(t-\eta_j)^{m_j}}
\left(\frac{z_2+|\eta_p|^2z_1/t}{1+|\eta_p|^2}\right)^{k_1+k_2-N+1}
\right]\,.
\end{eqnarray}
\bigskip

Lastly, we have
$k_2\leq m_n-1$ then 
$k_1\geq N-k_2\geq m_1$ and
$k_1+k_2-N\leq k_1-(m_1+\cdots+m_{n-1})-1$. It follows that
\begin{eqnarray*}
\sum_{v_1+\cdots+v_{n-1}=k_1-(m_1+\cdots+m_{n-1}),\,v_1\leq k_1+k_2-N}
z_1^{v_1}
\,z_2^{k_1+k_2-N-v_1}
\;\prod_{j=2}^{n-1}
\;\frac{(v_j+m_j-1)!}{v_j!\,(m_j-1)!}
\;\eta_j^{v_j}
\;=
\end{eqnarray*}
\begin{eqnarray*}
& = &
z_2^{k_1+k_2-N}
\;Q\left(X^{k_1-m_1},\prod_{j=2}^{n-1}(X-\eta_j)^{m_j}\right)
|_{X\;=\;z_1/z_2}\\
& &
-\;
\sum_{v_1+\cdots+v_{n-1}=k_1-(m_1+\cdots+m_{n-1}),\,v_1\geq k_1+k_2-N+1}
z_1^{v_1}
\,z_2^{k_1+k_2-N-v_1}
\;\prod_{j=2}^{n-1}
\;\frac{(v_j+m_j-1)!}{v_j!\,(m_j-1)!}
\;\eta_j^{v_j}
\end{eqnarray*}
\begin{eqnarray*}
& = &
z_2^{k_1+k_2-N}
\;Q\left(X^{k_1-m_1},\prod_{j=2}^{n-1}(X-\eta_j)^{m_j}\right)
|_{X\;=\;z_1/z_2}\\
& &
-\;
\frac{z_1^{k_1+k_2-N+1}}{z_2}
\sum_{v_1+\cdots+v_{n-1}=m_n-1-k_2}
(z_1/z_2)^{v_1}\;\prod_{j=2}^{n-1}
\;\frac{(v_j+m_j-1)!}{v_j!\,(m_j-1)!}
\;\eta_j^{v_j}\,.
\\
\end{eqnarray*}
Notice by lemma \ref{lagrangeuclide} 
(since
$N-m_1\geq m_n\geq k_2+1$)
that
\begin{eqnarray*}
\sum_{u=0}^{m_n-1-k_2}
(z_1/z_2)^{m_n-1-k_2-u}
\sum_{v_2+\cdots+v_{n-1}=u}
\;\prod_{j=2}^{n-1}
\;\frac{(v_j+m_j-1)!}{v_j!\,(m_j-1)!}
\;\eta_j^{v_j}
& = &
\;\;\;\;\;\;\;\;\;\;\;\;\;\;\;\;\;\;\;\;\;\;\;\;
\end{eqnarray*}
\begin{eqnarray*}
& = &
Q\left(
X^{N-m_1-1-k_2},
\prod_{j=2}^{n-1}(X-\eta_j)^{m_j}
\right)|_{X\;=\;z_1/z_2}
\\
\\
& = &
\frac{(z_1/z_2)^{N-m_1-1-k_2}}{\prod_{j=2}^{n-1}(z_1/z_2-\eta_j)^{m_j}}
\;-\;\sum_{p=2}^{n-1}
\frac{1}{(m_p-1)!}
\frac{\partial^{m_p-1}}{\partial t^{m_p-1}}|_{t=\eta_p}
\left[
\frac{t^{N-m_1-1-k_2}}{(z_1/z_2-t)\;\prod_{j=2,j\neq p}^{n-1}\;(t-\eta_j)^{m_j}}
\right]\,.
\\
\end{eqnarray*}
It follows from (\ref{infini4}) that
\begin{eqnarray*}
-\,\lim_{\varepsilon\rightarrow0}\frac{1}{(2\pi i)^2}
\;\int_{\Sigma_{\varepsilon}}
\frac{\zeta_1^{k_1}\zeta_2^{k_2}\;\omega'\left(\overline{\zeta}\right)\wedge\omega(\zeta)}
{g_n(\zeta)\;\left(1-<\overline{\zeta},z>\right)^2}
& = &
\;\;\;\;\;\;\;\;\;\;\;\;\;\;\;\;\;\;\;\;\;\;\;\;\;\;\;\;\;\;\;\;\;\;\;\;\;\;\;\;\;\;\;\;\;\;\;\;\;\;\;\;\;\;\;\;\;\;\;
\;\;\;\;
\\
\end{eqnarray*}
\begin{eqnarray*}
\;\;\;\;\;\;\;
& = &
z_2^{k_1+k_2-N}
\frac{(z_1/z_2)^{k_1-m_1}}{\prod_{j=2}^{n-1}(z_1/z_2-\eta_j)^{m_j}}
-\;
\frac{z_1^{k_1+k_2-N+1}}{z_2}\frac{(z_1/z_2)^{N-m_1-1-k_2}}{\prod_{j=2}^{n-1}(z_1/z_2-\eta_j)^{m_j}}\\
& + &
\frac{z_1^{k_1+k_2-N+1}}{z_2}
\;\sum_{p=2}^{n-1}
\frac{1}{(m_p-1)!}
\frac{\partial^{m_p-1}}{\partial t^{m_p-1}}|_{t=\eta_p}
\left[
\frac{t^{N-m_1-1-k_2}}{(z_1/z_2-t)\;\prod_{j=2,j\neq p}^{n-1}\;(t-\eta_j)^{m_j}}
\right]
\\
& - & 
\sum_{p=2}^{n-1}
\frac{1}{(m_p-1)!}
\;\times
\\
& &
\frac{\partial^{m_p-1}}{\partial t^{m_p-1}}|_{t=\eta_p}
\left[
\frac{t^{k_1-m_1}}{(z_1-tz_2)\;\prod_{j=2,j\neq p}^{n-1}\;(t-\eta_j)^{m_j}}
\left(\frac{z_2+|\eta_p|^2z_1/t}{1+|\eta_p|^2}\right)^{k_1+k_2-N+1}
\right]
\\
\end{eqnarray*}
\begin{eqnarray}\label{infini5b}
& = &
\sum_{p=2}^{n-1}
\frac{1}{(m_p-1)!}
\frac{\partial^{m_p-1}}{\partial t^{m_p-1}}|_{t=\eta_p}
\left[
\frac{t^{N-m_1-1-k_2}\;z_1^{k_1+k_2-N+1}}{(z_1-tz_2)\;\prod_{j=2,j\neq p}^{n-1}\;(t-\eta_j)^{m_j}}
\right]
\;\;\;\;\;\;\;\;
\end{eqnarray}
\begin{eqnarray*}
\;\;\;\;\;\;\;
& & 
-\;\sum_{p=2}^{n-1}
\frac{1}{(m_p-1)!}
\frac{\partial^{m_p-1}}{\partial t^{m_p-1}}|_{t=\eta_p}
\left[
\frac{t^{k_1-m_1}}{(z_1-tz_2)\;\prod_{j=2,j\neq p}^{n-1}\;(t-\eta_j)^{m_j}}
\left(\frac{z_2+|\eta_p|^2z_1/t}{1+|\eta_p|^2}\right)^{k_1+k_2-N+1}
\right]\,.\\
\end{eqnarray*}

\bigskip

Finally, we get from (\ref{infini4}), (\ref{infini5c}), 
(\ref{infini5a}) and (\ref{infini5b})
\begin{eqnarray*}
-\;
\lim_{\varepsilon\rightarrow0}
\,\frac{g_n(z)}{(2\pi i)^2}
\;\int_{\Sigma_{\varepsilon}}
\frac{\zeta_1^{k_1}\zeta_2^{k_2}\;\omega'\left(\overline{\zeta}\right)\wedge\omega(\zeta)}
{g_n(\zeta)\;\left(1-<\overline{\zeta},z>\right)^2}
& = &
\;\;\;\;\;\;\;\;\;\;\;\;\;\;\;\;\;\;\;\;\;\;\;\;\;\;\;\;\;\;\;\;\;\;\;\;\;\;\;\;\;\;\;\;\;\;\;\;
\\
\end{eqnarray*}
\begin{eqnarray*}
& = &
{\bf 1}_{k_1+k_2\geq N,k_1\geq m_1,k_2\geq m_n}
\,z_1^{k_1}\,z_2^{k_2}
\\
\\
& & 
-\;
{\bf 1}_{k_1+k_2\geq N}\,
\sum_{p=2}^{n-1}
z_1^{m_1}
\prod_{j=2,j\neq p}^{n-1}(z_1-\eta_jz_2)^{m_j}
\;z_2^{m_n}
\;\times
\\
& &
\;\;\times\;
\frac{(z_1-\eta_pz_2)^{m_p}}{(m_p-1)!}
\frac{\partial^{m_p-1}}{\partial t^{m_p-1}}|_{t=\eta_p}
\frac{t^{k_1}\,
\left(\frac{z_2+|\eta_p|^2z_1/t}{1+|\eta_p|^2}\right)^{k_1+k_2-N+1}}
{(z_1-tz_2)\;t^{m_1}\prod_{j=2,j\neq p}^{n-1}\;(t-\eta_j)^{m_j}}
\end{eqnarray*}
\begin{eqnarray*}
& & 
\;\;\;\;+\;
{\bf 1}_{k_1\leq m_1-1,k_2\geq N-k_1}\,
\sum_{p=2}^{n-1}
z_1^{m_1}
\prod_{j=2,j\neq p}^{n-1}(z_1-\eta_jz_2)^{m_j}
\;z_2^{m_n}
\;\times
\\
& & 
\;\;\;\;\;\;
\times\;
\frac{(z_1-\eta_pz_2)^{m_p}}{(m_p-1)!}
\frac{\partial^{m_p-1}}{\partial t^{m_p-1}}|_{t=\eta_p}
\frac{t^{k_1}\;z_2^{k_1+k_2-N+1}}
{(z_1-tz_2)\;t^{m_1}\,\prod_{j=2,j\neq p}^{n-1}\;(t-\eta_j)^{m_j}}
\\
\\
& & 
\;\;\;\;+\;
{\bf 1}_{k_2\leq m_n-1,k_1\geq N-k_2}\,
\sum_{p=2}^{n-1}
z_1^{m_1}
\prod_{j=2,j\neq p}^{n-1}(z_1-\eta_jz_2)^{m_j}
\;z_2^{m_n}
\;\times
\\
& & 
\;\;\;\;\;\;
\times\;
\frac{(z_1-\eta_pz_2)^{m_p}}{(m_p-1)!}
\frac{\partial^{m_p-1}}{\partial t^{m_p-1}}|_{t=\eta_p}
\frac{t^{N-1-k_2}\,z_1^{k_1+k_2-N+1}}
{(z_1-tz_2)\;t^{m_1}\,\prod_{j=2,j\neq p}^{n-1}\;(t-\eta_j)^{m_j}}
\\
\end{eqnarray*}
and the proof of the proposition is achieved.

\end{proof}

\bigskip

\section{Calculation of the interpolation part}\label{interpolers}

\bigskip

In this section, we will prove the following result
that allows us to complete the proof of theorem \ref{theorem}.
We assume that there exists 
$m_p\geq1,\;2\leq p\leq n-1$.

\begin{proposition}\label{interpoler}

Let be 
$f\in\mathcal{O}\left(\mathbb{B}_2\right)$
and
$f(z)=\sum_{k_1,k_2\geq0}
a_{k_1,k_2}z_1^{k_1}z_2^{k_2}$
its Taylor expansion. 
For all
$z\in\mathbb{B}_2$,
\begin{eqnarray}\label{interpolerprop1}
\lim_{\varepsilon\rightarrow0}
\frac{1}{(2\pi i)^2}
\int_{\partial\widetilde{\Sigma}_{\varepsilon}}
\frac{\zeta_1^{k_1}\zeta_2^{k_2}\det\left(\overline{\zeta},P_n(\zeta,z)\right)}
{g_n(\zeta)\left(1-<\overline{\zeta},z>\right)}
\,\omega(\zeta)\;=
\;\;\;\;\;\;\;\;\;\;\;\;\;\;\;\;\;\;\;\;\;\;\;\;\;\;\;\;
\end{eqnarray}

\begin{eqnarray*}
& = &
\sum_{u_1=0}^{m_1-1}z_1^{u_1}
\prod_{j=2}^{n-1}(z_1-\eta_jz_2)^{m_j}z_2^{m_n}
\sum_{p=2}^{n-1}
\frac{1}{(m_p-1)!}
\frac{\partial^{m_p-1}}{\partial t^{m_p-1}}|_{t=\eta_p}
\;\frac{1}{t^{u_1+1}\prod_{j=2,j\neq p}^{n-1}(t-\eta_j)^{m_j}}
\\
& &
\times\;
\{\;
{\bf 1}_{k_1+k_2\geq N_{u_1}}
\frac{1}{1+|\eta_p|^2}
t^{k_1}
\left(
\frac{z_2+|\eta_p|^2z_1/t}{1+|\eta_p|^2}
\right)^{k_1+k_2-N_{u_1}}
\\
& &
\;\;-\;\;\;
{\bf 1}_{k_1\leq u_1,\,k_2\geq N_{u_1}-k_1}
\,t^{k_1}
z_2^{k_1+k_2-N_{u_1}}
\;\}
\end{eqnarray*}
\begin{eqnarray*}
& + & 
\sum_{p=2}^{n-1}\sum_{u_p=0}^{m_p-1}
(z_1-\eta_pz_2)^{u_p}
\prod_{j=p+1}^{n-1}(z_1-\eta_jz_2)^{m_j}z_2^{m_n}
\;\times
\\
& \times &
\{\;
{\bf 1}_{k_1+k_2\geq N_{u_p}}
\{\,
\frac{1}{u_p!}
\frac{\partial^{u_p}}{\partial t^{u_p}}|_{t=\eta_p}
\,\frac{1+|\eta_p|^2\eta_p/t}{1+|\eta_p|^2}
\frac{\,t^{k_1}
\left(
\frac{z_2+|\eta_p|^2z_1/t}{1+|\eta_p|^2}
\right)^{k_1+k_2-N_{u_p}}}
{\prod_{j=p+1}^{n-1}(t-\eta_j)^{m_j}}
\\
& + &
\sum_{q=p+1}^{n-1}
\frac{1}{(m_q-1)!}
\frac{\partial^{m_q-1}}{\partial t^{m_q-1}}|_{t=\eta_q}
\,\frac{1+|\eta_q|^2\eta_p/t}{1+|\eta_q|^2}
\frac{\,t^{k_1}
\left(
\frac{z_2+|\eta_q|^2z_1/t}{1+|\eta_q|^2}
\right)^{k_1+k_2-N_{u_p}}}
{(t-\eta_p)^{u_p+1}\prod_{j=p+1,j\neq q}^{n-1}(t-\eta_j)^{m_j}}
\;\}
\end{eqnarray*}
\begin{eqnarray*}
& - & 
{\bf 1}_{k_2\leq m_n-1,k_1\geq N_{u_p}-k_2}
\,\eta_p\,
z_1^{k_1+k_2-N_{u_p}}
\;\times
\{\;
\frac{1}{u_p!}
\frac{\partial^{u_p}}{\partial t^{u_p}}|_{t=\eta_p}
\,
\frac{t^{N_{u_p}-1-k_2}}
{\prod_{j=p+1}^{n-1}(t-\eta_j)^{m_j}}
\\
& &
+
\sum_{q=p+1}^{n-1}
\frac{1}{(m_q-1)!}
\frac{\partial^{m_q-1}}{\partial t^{m_q-1}}|_{t=\eta_q}
\,\frac{t^{N_{u_p}-1-k_2}}
{(t-\eta_p)^{u_p+1}\prod_{j=p+1,j\neq q}^{n-1}(t-\eta_j)^{m_j}}
\;\}
\end{eqnarray*}
\begin{eqnarray*}
& + &
{\bf 1}_{k_1\geq0,k_2\leq m_n-1}
\,z_1^{k_1}z_2^{k_2}\;.
\;\;\;\;\;\;\;\;\;\;\;\;\;\;\;\;\;\;\;\;\;\;\;\;\;\;\;\;\;\;\;\;\;\;\;\;\;\;\;\;\;\;\;\;\;\;\;\;\;\;\;\;\;\;
\;\;\;\;\;\;\;\;\;\;\;\;\;\;\;\;\;\;\;\;\;\;\;\;\;\;\;
\\
\end{eqnarray*}

\end{proposition}

\bigskip

Before giving the proof of this proposition, we need to specify
$P_n(\zeta,z)$.

\begin{lemma}\label{Pn}

For all 
$(\zeta,z)\in\mathbb{C}^2$, 
one can choose
\begin{eqnarray}\label{Pn1}
P_n^1(\zeta,z) 
& = &
\sum_{u_1=0}^{m_1-1}
\zeta_1^{m_1-1-u_1}z_1^{u_1}
\prod_{j=2}^{n-1}(z_1-\eta_jz_2)^{m_j}z_2^{m_n}
\;\;\;\;\;\;\;\;\;\;\;\;\;\;\;\;\;\;\;\;\;\;\;\;
\end{eqnarray}
\begin{eqnarray*}
+
\sum_{p=2}^{n-1}\sum_{u_p=0}^{m_p-1}
\zeta_1^{m_1}
\prod_{j=2}^{p-1}(\zeta_1-\eta_j\zeta_2)^{m_j}
(\zeta_1-\eta_p\zeta_2)^{m_p-1-u_p}
(z_1-\eta_pz_2)^{u_p}
\prod_{j=p+1}^{n-1}(z_1-\eta_jz_2)^{m_j}
z_2^{m_n}\,,
\end{eqnarray*}

\begin{eqnarray}\label{Pn2}
P_n^2(\zeta,z)
& = &
\sum_{u_n=0}^{m_n-1}
\zeta_1^{m_1}
\prod_{j=2}^{n-1}(\zeta_1-\eta_j\zeta_2)^{m_j}
\zeta_2^{m_n-1-u_n}z_2^{u_n}
\;\;\;\;\;\;\;\;\;\;\;\;\;\;\;\;\;\;\;\;\;\;\;\;
\end{eqnarray}
\begin{eqnarray*}
-
\sum_{p=2}^{n-1}\sum_{u_p=0}^{m_p-1}
\eta_p\,
\zeta_1^{m_1}
\prod_{j=2}^{p-1}(\zeta_1-\eta_j\zeta_2)^{m_j}
(\zeta_1-\eta_p\zeta_2)^{m_p-1-u_p}
(z_1-\eta_pz_2)^{u_p}
\prod_{j=p+1}^{n-1}(z_1-\eta_jz_2)^{m_j}
z_2^{m_n}\,.
\end{eqnarray*}

\end{lemma}

\bigskip

\begin{proof}

First, we prove the following fact by induction on $m$: 
$h_1,\ldots,h_m$ being given, consider the associate
$P_{h_1},\ldots,P_{h_m}$. Then one can choose

\begin{eqnarray}\label{prelimPn}
P_{\prod_{j=1}^mh_j}(\zeta,z)
& = &
\sum_{p=1}^m\;
\prod_{j=1}^{p-1}h_j(\zeta)
\;P_{h_p}(\zeta,z)
\prod_{j=p+1}^mh_j(z)\,,
\end{eqnarray}
i.e.
\begin{eqnarray*}
\prod_{j=1}^mh_j(\zeta)-\prod_{j=1}^mh_j(z)
& = &
\sum_{p=1}^m\;
\prod_{j=1}^{p-1}h_j(\zeta)
\prod_{j=p+1}^mh_j(z)
\;<P_{h_p}(\zeta,z),\zeta-z>
\,.
\\
\end{eqnarray*}
This is obvious for $m=1$. Now consider
$h_1,\ldots,h_m$, we have
\begin{eqnarray*}
\prod_{j=1}^mh_j(\zeta)-\prod_{j=1}^mh_j(z)
& = &
(h_m(\zeta)-h_m(z))
\prod_{j=1}^{m-1}h_j(\zeta)
\;+
\left(
\prod_{j=1}^{m-1}h_j(\zeta)
-\prod_{j=1}^{m-1}h_j(z)
\right)
h_m(z)
\\
& = &
\prod_{j=1}^{m-1}h_j(\zeta)
\,<P_{h_m}(\zeta,z),\zeta-z>
\\
& &
+\;
h_m(z)\,
\sum_{p=1}^{m-1}\;
\prod_{j=1}^{p-1}h_j(\zeta)
\prod_{j=p+1}^{m-1}h_j(z)
\;<P_{h_p}(\zeta,z),\zeta-z>
\end{eqnarray*}
and this proves (\ref{prelimPn}).

On the other hand, we see that,
for all
$\eta\in\mathbb{C}$,
\begin{eqnarray*}
P_{(t\mapsto t_1-\eta t_2)}(\zeta,z)
=
\left(
\begin{array}{c}
1\\
-\eta\\
\end{array}
\right)\,,
\end{eqnarray*}
as well as

\begin{eqnarray*}
P_{(t\mapsto t_2)}(\zeta,z)
=
\left(
\begin{array}{c}
0\\
1\\
\end{array}
\right)\,.
\end{eqnarray*}

One can deduce from (\ref{prelimPn}) 
the following choice about 
$g_n(z)=\prod_{u_1=1}^{m_1}z_1
\,\prod_{p=2}^{n-1}
\prod_{u_p=1}^{m_p}(z_1-\eta_pz_2)
\,\prod_{u_n=1}^{m_n}z_2$:
\begin{eqnarray*}
P_n(\zeta,z)
& = &
\;\sum_{u_1=1}^{m_1}
\zeta_1^{u_1-1}
\left(
\begin{array}{c}
1\\
0\\
\end{array}
\right)
z_1^{m_1-u_1}
\prod_{j=2}^{n-1}
(z_1-\eta_jz_2)^{m_j}\,z_2^{m_n}
\\
\\
& & 
+\;
\sum_{p=2}^{n-1}
\sum_{u_p=1}^{m_p}
\zeta_1^{m_1}\,
\prod_{j=2}^{p-1}(\zeta_1-\eta_j\zeta_2)^{m_j}
(\zeta_1-\eta_p\zeta_2)^{u_p-1}
\;\times
\\
& & 
\;\;\;\;\;\;\;\;\;\;\;\;\;\;\;\;\;\;\;\;\;\;\;\;\times\;
\left(
\begin{array}{c}
1\\
-\eta_p\\
\end{array}
\right)
(z_1-\eta_pz_2)^{m_p-u_p}
\prod_{j=p+1}^{n-1}(z_1-\eta_jz_2)^{m_j}
\,z_2^{m_n}
\\
\\
& & 
+\;
\sum_{u_n=1}^{m_n}
\zeta_1^{m_1}\,
\prod_{j=2}^{n-1}
(\zeta_1-\eta_j\zeta_2)^{m_j}
\,\zeta_2^{u_n-1}
\left(
\begin{array}{c}
0\\
1\\
\end{array}
\right)
z_2^{m_n-u_n}
\,,
\\
\end{eqnarray*}
and the lemma is proved.

\end{proof}

\bigskip

Now we can give the proof of proposition
\ref{interpoler}.

\begin{proof}

We have by lemma \ref{Pn}
\begin{eqnarray*}
\frac{\det(\overline{\zeta},P_n(\zeta,z))}
{g_n(\zeta)}
& = &
\overline{\zeta}_1
\sum_{u_n=0}^{m_n-1}
\frac{z_2^{u_n}}{\zeta_2^{u_n+1}}
\\
& - &
\sum_{p=2}^{n-1}\sum_{u_p=0}^{m_p-1}
(\overline{\zeta}_2+\eta_p\overline{\zeta}_1)
\frac{(z_1-\eta_pz_2)^{u_p}
\prod_{j=p+1}^{n-1}(z_1-\eta_jz_2)^{m_j}
z_2^{m_n}}
{(\zeta_1-\eta_p\zeta_2)^{u_p+1}
\prod_{j=p+1}^{n-1}(\zeta_1-\eta_j\zeta_2)^{m_j}
\zeta_2^{m_n}}
\\
& - &
\overline{\zeta}_2
\sum_{u_1=0}^{m_1-1}
\frac{z_1^{u_1}
\prod_{j=2}^{n-1}(z_1-\eta_jz_2)^{m_j}
z_2^{m_n}}
{\zeta_1^{u_1+1}
\prod_{j=2}^{n-1}(\zeta_1-\eta_j\zeta_2)^{m_j}
\zeta_2^{m_n}}
\,.
\\
\end{eqnarray*}
Then we want to calculate, for all
$k_1,\;k_2\geq0$:
\begin{eqnarray}\label{interpoler1a}
\lim_{\varepsilon\rightarrow0}
\frac{1}{(2\pi i)^2}
\int_{\partial\widetilde{\Sigma}_{\varepsilon}}
\frac{\zeta_1^{k_1}\zeta_2^{k_2}\det\left(\overline{\zeta},P_n(\zeta,z)\right)}
{g_n(\zeta)\left(1-<\overline{\zeta},z>\right)}
\,\omega(\zeta)
& = &
\;\;\;\;\;\;\;\;\;\;\;\;\;\;\;\;\;\;\;\;\;\;\;\;\;\;\;\;\;\;
\end{eqnarray}
\begin{eqnarray*}
& = &
\lim_{\varepsilon\rightarrow0}
\frac{1}{(2\pi i)^2}
\left[
\int_{r=1-\varepsilon}
-\sum_{l=2}^{\widetilde{n}-1}
\left(\int_{r=\alpha_{q_l}+\varepsilon}
-\int_{r=\alpha_{q_l}-\varepsilon}\right)
-\int_{r=\varepsilon}
\right]
\;\;\;\;\;\;\;\;\;\;\;\;\;\;\;\;\;\;\;\;\;\;\;\;\;\;\;\;\;\;\;\;\;\;
\\
& &
\frac{\omega(\zeta)}
{(\zeta_2-(1-r^2)z_2)
(\zeta_1-\frac{r^2z_1\zeta_2}{\zeta_2-(1-r^2)^2z_2})}
\;\times\;\{\;
\sum_{u_n=0}^{m_n-1}
z_2^{u_n}r^2
\zeta_1^{k_1}\zeta_2^{k_2-u_n}
\end{eqnarray*}
\begin{eqnarray*}
& - &
\sum_{p=2}^{n-1}\sum_{u_p=0}^{m_p-1}
(z_1-\eta_pz_2)^{u_p}
\prod_{j=p+1}^{n-1}(z_1-\eta_jz_2)^{m_j}
z_2^{m_n}
\frac{
\zeta_1^{k_1}\zeta_2^{k_2-m_n}
((1-r^2)\zeta_1+\eta_pr^2\zeta_2)}
{(\zeta_1-\eta_p\zeta_2)^{u_p+1}
\prod_{j=p+1}^{n-1}(\zeta_1-\eta_j\zeta_2)^{m_j}}
\\
& - &
\sum_{u_1=0}^{m_1-1}
z_1^{u_1}
\prod_{j=2}^{n-1}(z_1-\eta_jz_2)^{m_j}
z_2^{m_n}
\frac{(1-r^2)\;
\zeta_1^{k_1-u_1}\zeta_2^{k_2-m_n}}
{\prod_{j=2}^{n-1}(\zeta_1-\eta_j\zeta_2)^{m_j}}
\;\;\}\;.
\\
\end{eqnarray*}

The proof will be a consequence of 
lemma \ref{etan}, lemma \ref{etap} 
and lemma \ref{eta1}.

\end{proof}

\bigskip

\begin{lemma}\label{etan}

For all 
$z\in\mathbb{B}_2$,
for all
$u_n=0,\ldots,m_n-1$ and all
$k_1,\,k_2\geq0$, 
we have
\begin{eqnarray*}
\lim_{\varepsilon\rightarrow0}
\frac{1}{(2\pi i)^2}
\int_{\partial\widetilde{\Sigma}_{\varepsilon}}
\frac{r^2
\zeta_1^{k_1}\zeta_2^{k_2-u_n}
\,\omega(\zeta)}
{(\zeta_2-(1-r^2)z_2)
(\zeta_1-\frac{r^2z_1\zeta_2}{\zeta_2-(1-r^2)^2z_2})}
& = &
{\bf 1}_{k_2=u_n}\;
z_1^{k_1}\,z_2^{u_n}\,.
\\
\end{eqnarray*}
It follows that
\begin{eqnarray*}
\sum_{u_n=0}^{m_n-1}
z_2^{u_n}
\lim_{\varepsilon\rightarrow0}
\frac{1}{(2\pi i)^2}
\int_{\partial\widetilde{\Sigma}_{\varepsilon}}
\frac{r^2
\zeta_1^{k_1}\zeta_2^{k_2-u_n}
\,\omega(\zeta)}
{(\zeta_2-(1-r^2)z_2)
(\zeta_1-\frac{r^2z_1\zeta_2}{\zeta_2-(1-r^2)^2z_2})}
& = &
\;\;\;\;\;\;\;\;\;\;\;\;
\end{eqnarray*}
\begin{eqnarray*}
\;\;\;\;\;\;\;\;\;\;\;\;\;\;\;\;\;\;\;\;\;\;\;\;\;\;\;\;\;\;\;\;\;\;\;\;\;\;\;\;\;\;\;\;\;\;\;\;\;\;\;\;\;\;\;\;\;\;
\;\;\;\;\;\;\;\;\;\;\;\;\;\;\;\;\;\;\;\;\;\;\;\;
& = &
{\bf 1}_{k_2\leq m_n-1}\,
z_1^{k_1}z_2^{k_2}
\;.
\end{eqnarray*}

\end{lemma}

\begin{proof}

First,
for all
$l=2,\ldots,\widetilde{n}-1$
and all
$u_n\geq0$
\begin{eqnarray*}
\lim_{\varepsilon\rightarrow0}
\int_{r=\alpha_{q_l}\pm\varepsilon}
\frac{
\overline{\zeta}_1
\zeta_1^{k_1}\zeta_2^{k_2-u_n}}
{(1-<\overline{\zeta},z>)}
\,\omega(\zeta)
& = &
\int_{r=\alpha_{q_l}}
\frac{
\overline{\zeta}_1
\zeta_1^{k_1}\zeta_2^{k_2-u_n}}
{(1-<\overline{\zeta},z>)}
\,\omega(\zeta)
\end{eqnarray*}
then
\begin{eqnarray*}
\lim_{\varepsilon\rightarrow0}
\frac{1}{(2\pi i)^2}
\int_{\partial\widetilde{\Sigma}_{\varepsilon}}
\frac{
\overline{\zeta}_1
\zeta_1^{k_1}\zeta_2^{k_2-u_n}}
{(1-<\overline{\zeta},z>)}
\,\omega(\zeta)
& = &
\;\;\;\;\;\;\;\;\;\;\;\;\;\;\;\;\;\;\;\;\;\;\;\;\;\;\;\;\;\;\;\;\;\;\;\;\;\;\;\;\;\;\;\;\;\;\;\;\;\;\;\;\;\;\;\;\;\;
\end{eqnarray*}
\begin{eqnarray*}
& = &
\lim_{\varepsilon\rightarrow0}
\frac{1}{(2\pi i)^2}
\left[
\int_{r=1-\varepsilon}
-\int_{r=\varepsilon}
\right]
\frac{
r^2\zeta_1^{k_1}\zeta_2^{k_2-u_n}
\,\omega(\zeta)}
{(\zeta_2-(1-r^2)z_2)
(\zeta_1-\frac{r^2z_1\zeta_2}{\zeta_2-(1-r^2)^2z_2})}
\,
\;\;\;\;\;\;\;\;\;\;\;\;\;\;\;\;\;\;\;\;\;\;\;\;\;\;\;\;
\end{eqnarray*}
\begin{eqnarray*}
& = &
\lim_{\varepsilon\rightarrow0}
\frac{1}{2\pi i}
\left[
\int_{|\zeta_2|=\sqrt{1-(1-\varepsilon)^2}}
-\int_{|\zeta_2|=\sqrt{1-\varepsilon^2}}
\right]
\frac{
\zeta_2^{k_2-u_n}
d\zeta_2}
{(\zeta_2-(1-r^2)z_2)}
\frac{1}{2\pi i}
\int_{|\zeta_1|=r}
\frac{r^2\zeta_1^{k_1}d\zeta_1}
{(\zeta_1-\frac{r^2z_1\zeta_2}{\zeta_2-(1-r^2)^2z_2})}
\\
& = &
\lim_{\varepsilon\rightarrow0}
\frac{1}{2\pi i}
\left[
\int_{|\zeta_2|=\sqrt{2\varepsilon-\varepsilon^2}}
-\int_{|\zeta_2|=\sqrt{1-\varepsilon^2}}
\right]
\frac{
r^2(r^2z_1)^{k_1}
\zeta_2^{k_1+k_2-u_n}d\zeta_2}
{(\zeta_2-(1-r^2)z_2)^{k_1+1}}
\,.
\end{eqnarray*}
This integral is zero if 
$k_2<u_n$. It follows that, for all
$u_n=0,\ldots,m_n-1$ and all
$k_2\geq u_n$,
\begin{eqnarray*}
\lim_{\varepsilon\rightarrow0}
\frac{1}{(2\pi i)^2}
\int_{\partial\widetilde{\Sigma}_{\varepsilon}}
\frac{
\overline{\zeta}_1
\zeta_1^{k_1}\zeta_2^{k_2-u_n}}
{(1-<\overline{\zeta},z>)}
\,\omega(\zeta)
& = &
\;\;\;\;\;\;\;\;\;\;\;\;\;\;\;\;\;\;\;\;\;\;\;\;\;\;\;\;\;\;\;\;\;\;\;\;\;\;\;\;\;\;\;\;\;\;\;\;\;\;\;\;\;\;\;
\;\;\;\;\;\;\;\;\;\;\;\;
\end{eqnarray*}
\begin{eqnarray*}
& = &
\lim_{\varepsilon\rightarrow0}
\frac{(k_1+k_2-u_n)!}
{k_1!\,(k_2-u_n)!}
z_1^{k_1}z_2^{k_2-u_n}
\left(
(1-\varepsilon)^{2(k_1+1)}
(2\varepsilon-\varepsilon^2)^{k_2-u_n}
-\varepsilon^{2(k_1+1)}
(1-\varepsilon^2)^{k_2-u_n}
\right)
\\
& = &
\begin{cases}
z_1^{k_1}
\;\mbox{if}\;k_2=u_n\,,\\
0\;\;\mbox{otherwise}\\
\end{cases}
\\
\end{eqnarray*}
and this proves the first part 
of the lemma.

The second part follows since
\begin{eqnarray*}
\sum_{u_n=0}^{m_n-1}
{\bf 1}_{k_2=u_n}
z_1^{k_1}z_2^{u_n}
\;=\;
z_1^{k_1}z_2^{k_n}
\sum_{u_n=0}^{m_n-1}
{\bf 1}_{k_2=u_n}
\;=\;
{\bf 1}_{k_2\leq m_n-1}
\,z_1^{k_1}z_2^{k_n}
\,.
\end{eqnarray*}

\end{proof}

\bigskip

Next, we have the following result.

\begin{lemma}\label{etap}

For all
$z\in U_{\eta}$,
for all
$p=2,\ldots,n-1$ and all 
$u_p=0,\ldots,m_p-1$, 
we have
\begin{eqnarray*}
\lim_{\varepsilon\rightarrow0}
\frac{1}{(2\pi i)^2}
\int_{\partial\widetilde{\Sigma}_{\varepsilon}}
\frac{
(\overline{\zeta}_2+\eta_p\overline{\zeta}_1)
\,\zeta_1^{k_1}\zeta_2^{k_2-m_n}\,\omega(\zeta)}
{(\zeta_1-\eta_p\zeta_2)^{u_p+1}
\prod_{j=p+1}^{n-1}
(\zeta_1-\eta_j\zeta_2)^{m_j}
(1-<\overline{\zeta},z>)}
& = &
\;\;\;\;\;\;\;\;\;\;\;\;\;\;\;\;\;\;\;\;\;
\end{eqnarray*}
\begin{eqnarray*}
& = &
-\;
{\bf 1}_{k_1+k_2\geq N_{u_p}}
\;\times
\{\;
\frac{1}{u_p!}
\frac{\partial^{u_p}}{\partial t^{u_p}}|_{t=\eta_p}
\left[
\;\frac{1+|\eta_p|^2\eta_p/t}{1+|\eta_p|^2}
\;
\frac{t^{k_1}
\left(
\frac{z_2+|\eta_p|^2z_1/t}{1+|\eta_p|^2}
\right)^{k_1+k_2-N_{u_p}}}
{\prod_{j=p+1}^{n-1}(t-\eta_j)^{m_j}}
\right]
\\
& + &
\sum_{q=p+1}^{n-1}
\frac{1}{(m_q-1)!}
\frac{\partial^{m_q-1}}{\partial t^{m_q-1}}|_{t=\eta_q}
\left[
\;\frac{1+|\eta_q|^2\eta_p/t}{1+|\eta_q|^2}
\;\frac{t^{k_1}
\left(
\frac{z_2+|\eta_p|^2z_1/t}{1+|\eta_p|^2}
\right)^{k_1+k_2-N_{u_p}}}
{(t-\eta_p)^{u_p+1}\prod_{j=p+1,j\neq q}^{n-1}(t-\eta_j)^{m_j}}
\right]
\;\}
\end{eqnarray*}
\begin{eqnarray*}
& + & 
{\bf 1}_{k_2\leq m_n-1,k_1\geq N_{u_p}-k_2}
\,\eta_p\,z_1^{k_1+k_2-N_{u_p}}
\;\times
\{\;
\frac{1}{u_p!}
\frac{\partial^{u_p}}{\partial t^{u_p}}|_{t=\eta_p}
\left[
\frac{t^{N_{u_p}-1-k_2}}
{\prod_{j=p+1}^{n-1}(t-\eta_j)^{m_j}}
\right]
\\
&  &
+\;
\sum_{q=p+1}^{n-1}
\frac{1}{(m_q-1)!}
\frac{\partial^{m_q-1}}{\partial t^{m_q-1}}|_{t=\eta_q}
\left[
\frac{t^{N_{u_p}-1-k_2}}
{(t-\eta_p)^{u_p+1}\prod_{j=p+1,j\neq q}^{n-1}(t-\eta_j)^{m_j}}
\right]
\;\}\;.
\end{eqnarray*}

\end{lemma}

\bigskip

\begin{proof}

First, consider
$l_p$ such that
$\alpha_{q_{l_p}}=\alpha_p$. Then
\begin{eqnarray}\label{etap0p1}
& &
\lim_{\varepsilon\rightarrow0}
\frac{1}{(2\pi i)^2}
\int_{\partial\widetilde{\Sigma}_{\varepsilon}}
\frac{
(\overline{\zeta}_2+\eta_p\overline{\zeta}_1)
\,\zeta_1^{k_1}\zeta_2^{k_2-m_n}\,\omega(\zeta)}
{(\zeta_1-\eta_p\zeta_2)^{u_p+1}
\prod_{j=p+1}^{n-1}
(\zeta_1-\eta_j\zeta_2)^{m_j}
(1-<\overline{\zeta},z>)}
\;=
\end{eqnarray}
\begin{eqnarray*}
& = &
\lim_{\varepsilon\rightarrow0}
\frac{1}{(2\pi i)^2}
\left[
\int_{r=1-\varepsilon}
-\sum_{l=l_p}^{\widetilde{n}-1}
\left(
\int_{r=\alpha_{q_l}+\varepsilon}-\int_{r=\alpha_{q_l}-\varepsilon}
\right)
-\int_{r=\varepsilon}
\right]
\\
& &
\;\;\;\;\;\;\;\;\;\;\;\;\;\;\;\;\;\;\;\;\;\;\;\;
\times\;
\frac{
(\overline{\zeta}_2+\eta_p\overline{\zeta}_1)
\,\zeta_1^{k_1}\zeta_2^{k_2-m_n}\,\omega(\zeta)}
{(\zeta_1-\eta_p\zeta_2)^{u_p+1}
\prod_{j=p+1}^{n-1}
(\zeta_1-\eta_j\zeta_2)^{m_j}
(1-<\overline{\zeta},z>)}
\,.
\\
\end{eqnarray*}

On the other hand, we know
by (\ref{preliminfini}) in lemma \ref{lagrangeuclide}
that,
for all
$r\in[0,1]$ such that
$r\neq\alpha_s,\;
\forall\,s=1,\ldots,n-1$
and all
$k_1\geq0$,
we have
\begin{eqnarray*}
\frac{1}{2\pi i}
\int_{|\zeta_1|=+\infty}
\frac{\zeta_1^{k_1}
((1-r^2)\zeta_1+\eta_pr^2\zeta_2)
\;d\zeta_1}
{(\zeta_1-\eta_p\zeta_2)^{u_p+1}
\prod_{j=p+1}^{n-1}(\zeta_1-\eta_j\zeta_2)^{m_j}
(\zeta_1-\frac{r^2z_1\zeta_2}{\zeta_2-(1-r^2)z_2})}
& = &
\;\;\;\;\;\;\;\;\;\;\;\;
\end{eqnarray*}
\begin{eqnarray*}
& = &
{\bf 1}_{k_1\geq u_p+m_{p+1}+\cdots+m_{n-1}}\times\{\;
\frac{(\frac{r^2z_1\zeta_2}{\zeta_2-(1-r^2)z_2})^{k_1}
((1-r^2)\frac{r^2z_1\zeta_2}{\zeta_2-(1-r^2)z_2}
+\eta_pr^2\zeta_2)}
{(\frac{r^2z_1\zeta_2}{\zeta_2-(1-r^2)z_2}
-\eta_p\zeta_2)^{u_p+1}
\prod_{j=p+1}^{n-1}
(\frac{r^2z_1\zeta_2}{\zeta_2-(1-r^2)z_2}
-\eta_j\zeta_2)^{m_j}}
\\
& - &
\frac{1}{u_p!}
\frac{\partial^{u_p}}{\partial\zeta_1^{u_p}}|_{\zeta_1=\eta_p\zeta_2}
\left[
\frac{\zeta_1^{k_1}((1-r^2)\zeta_1+\eta_pr^2\zeta_2)}
{\prod_{j=p+1}^{n-1}(\zeta_1-\eta_j\zeta_2)^{m_j}
(\frac{r^2z_1\zeta_2}{\zeta_2-(1-r^2)z_2}-\zeta_1)}
\right]
\end{eqnarray*}
\begin{eqnarray*}
& - &
\sum_{v=p+1}^{n-1}
\frac{1}{(m_v-1)!}
\;\times
\\
& &
\frac{\partial^{m_v-1}}{\partial\zeta_1^{m_v-1}}|_{\zeta_1=\eta_v\zeta_2}
\left[
\frac{\zeta_1^{k_1}((1-r^2)\zeta_1+\eta_pr^2\zeta_2)}
{(\zeta_1-\eta_p\zeta_2)^{u_p+1}
\prod_{j=p+1,j\neq v}^{n-1}(\zeta_1-\eta_j\zeta_2)^{m_j}
(\frac{r^2z_1\zeta_2}{\zeta_2-(1-r^2)z_2}-\zeta_1)}
\right]
\;\}
\end{eqnarray*}
\begin{eqnarray}\label{infinip}
& = &
{\bf 1}_{k_1\geq u_p+m_{p+1}+\cdots+m_{n-1}}
\,\zeta_2^{k_1-(u_p+m_{p+1}+\cdots+m_{n-1})}
\,P_r\left(\frac{r^2z_1}{\zeta_2-(1-r^2)z_2}\right)\,,
\end{eqnarray}
with the following quotient
\begin{eqnarray*}
P_r(X)
& = &
Q\left(
X^{k_1}\left((1-r^2)X+\eta_pr^2\right)
\,,\,
(X-\eta_p)^{u_p+1}\prod_{j=p+1}^{n-1}(X-\eta_j)^{m_j}\right)
\,.
\end{eqnarray*}

It follows that, for all
$l=l_p,\ldots,\widetilde{n}-1$
and for all
$r=\alpha_{q_l}+\varepsilon$
(in particular, 
$r>\alpha_p$),
\begin{eqnarray*}
\frac{1}{2\pi i}
\int_{|\zeta_1|=r}
\frac{\zeta_1^{k_1}((1-r^2)\zeta_1+\eta_pr^2\zeta_2)
\;d\zeta_1}
{(\zeta_1-\eta_p\zeta_2)^{u_p+1}
\prod_{j=p+1}^{n-1}(\zeta_1-\eta_j\zeta_2)^{m_j}
(\zeta_1-\frac{r^2z_1\zeta_2}{\zeta_2-(1-r^2)z_2})}
& = &
\;\;\;\;\;\;\;\;\;\;\;\;\;\;\;\;\;\;\;
\end{eqnarray*}
\begin{eqnarray*}
& = &
{\bf 1}_{k_1\geq u_p+m_{p+1}+\cdots+m_{n-1}}
\,\zeta_2^{k_1-(u_p+m_{p+1}+\cdots+m_{n-1})}
\,P_r\left(\frac{r^2z_1}{\zeta_2-(1-r^2)z_2}\right)
\\
& - &
\sum_{v\leq n-1,\alpha_v>\alpha_{q_l}}
\frac{1}{(m_v-1)!}
\;\times
\\
& &
\frac{\partial^{m_v-1}}{\partial\zeta_1^{m_v-1}}|_{\zeta_1=\eta_v\zeta_2}
\frac{\zeta_1^{k_1}((1-r^2)\zeta_1+\eta_pr^2\zeta_2)}
{(\zeta_1-\eta_p\zeta_2)^{u_p+1}
\prod_{j=p+1,j\neq v}^{n-1}(\zeta_1-\eta_j\zeta_2)^{m_j}
(\zeta_1-\frac{r^2z_1\zeta_2}{\zeta_2-(1-r^2)z_2})}
\end{eqnarray*}
\begin{eqnarray*}
& = &
\zeta_2^{k_1-(u_p+m_{p+1}+\cdots+m_{n-1})}
\times\{\;
{\bf 1}_{k_1\geq u_p+m_{p+1}+\cdots+m_{n-1}}
\,P_r\left(\frac{r^2z_1}{\zeta_2-(1-r^2)z_2}\right)
\;\;\;
\\
& - &
\sum_{v\leq n-1,\alpha_v>\alpha_{q_l}}
\frac{1}{(m_v-1)!}
\;\times
\\
& &
\frac{\partial^{m_v-1}}{\partial t^{m_v-1}}|_{t=\eta_v}
\frac{t^{k_1}((1-r^2)t+\eta_pr^2)}
{(t-\eta_p)^{u_p+1}
\prod_{j=p+1,j\neq v}^{n-1}(t-\eta_j)^{m_j}
(t-\frac{r^2z_1}{\zeta_2-(1-r^2)z_2})}\;\}
\,.
\\
\end{eqnarray*}
Then
\begin{eqnarray*}
\frac{1}{(2\pi i)^2}
\int_{r=\alpha_{q_l}+\varepsilon}
\frac{
(\overline{\zeta}_2+\eta_p\overline{\zeta}_1)
\,\zeta_1^{k_1}\zeta_2^{k_2-m_n}\,\omega(\zeta)}
{(\zeta_1-\eta_p\zeta_2)^{u_p+1}
\prod_{j=p+1}^{n-1}
(\zeta_1-\eta_j\zeta_2)^{m_j}
(1-<\overline{\zeta},z>)}
& = &
\;\;\;\;\;\;\;\;\;\;\;\;\;\;\;\;\;\;\;\;\;
\end{eqnarray*}
\begin{eqnarray*}
& = &
{\bf 1}_{k_1\geq u_p+m_{p+1}+\cdots+m_{n-1}}
\frac{1}{2\pi i}
\int_{|\zeta_2|=\sqrt{1-r^2}}
\frac{\zeta_2^{k_1+k_2-N_{u_p}}}
{\zeta_2-(1-r^2)z_2}
P_r\left(\frac{r^2z_1}{\zeta_2-(1-r^2)z_2}\right)
\,d\zeta_2
\\
& - &
\sum_{v\leq n-1,\alpha_v>\alpha_{q_l}}
\frac{1}{(m_v-1)!}
\frac{\partial^{m_v-1}}{\partial t^{m_v-1}}|_{t=\eta_v}
\frac{t^{k_1}((1-r^2)+\eta_pr^2/t)}
{(t-\eta_p)^{u_p+1}
\prod_{j=p+1,j\neq v}^{n-1}(t-\eta_j)^{m_j}}
\;\times
\\
& &
\;\;\;\;\;\;\;\;\;\;\;\;\;\;\;\;\;\;\;\;\;\;\;\;\;\;\;\;\;\;\;\;\;\;\;\;\;\;
\times\;
\frac{1}{2\pi i}
\int_{|\zeta_2|=\sqrt{1-r^2}}
\frac{\zeta_2^{k_1+k_2-N_{u_p}}}
{\zeta_2-(1-r^2)z_2-r^2z_1/t}
\,d\zeta_2
\;.
\end{eqnarray*}
Similarly, for all
$l=l_p,\ldots,\widetilde{n}-1$ and all
$r=\alpha_{q_l}-\varepsilon$,
\begin{eqnarray*}
\frac{1}{(2\pi i)^2}
\int_{r=\alpha_{q_l}-\varepsilon}
\frac{
(\overline{\zeta}_2+\eta_p\overline{\zeta}_1)
\,\zeta_1^{k_1}\zeta_2^{k_2-m_n}\,\omega(\zeta)}
{(\zeta_1-\eta_p\zeta_2)^{u_p+1}
\prod_{j=p+1,j\neq v}^{n-1}
(\zeta_1-\eta_j\zeta_2)^{m_j}
(1-<\overline{\zeta},z>)}
& = &
\;\;\;\;\;\;\;\;\;\;\;\;\;\;\;\;\;\;\;
\end{eqnarray*}
\begin{eqnarray*}
& = &
{\bf 1}_{k_1\geq u_p+m_{p+1}+\cdots+m_{n-1}}
\frac{1}{2\pi i}
\int_{|\zeta_2|=\sqrt{1-r^2}}
\frac{\zeta_2^{k_1+k_2-N_{u_p}}}
{\zeta_2-(1-r^2)z_2}
P_r\left(\frac{r^2z_1}{\zeta_2-(1-r^2)z_2}\right)
\,d\zeta_2
\\
& - &
\sum_{v\leq n-1,\alpha_v\geq\alpha_l,v\neq p}
\frac{1}{(m_v-1)!}
\frac{\partial^{m_v-1}}{\partial t^{m_v-1}}|_{t=\eta_v}
\frac{t^{k_1}((1-r^2)+\eta_pr^2/t)}
{(t-\eta_p)^{u_p+1}
\prod_{j=p+1,j\neq v}^{n-1}(t-\eta_j)^{m_j}}
\;\times
\\
& &
\;\;\;\;\;\;\;\;\;\;\;\;\;\;\;\;\;\;\;\;\;\;\;\;\;\;\;\;\;\;\;\;\;\;\;\;\;\;\;\;\;\;
\times\;
\frac{1}{2\pi i}
\int_{|\zeta_2|=\sqrt{1-r^2}}
\frac{\zeta_2^{k_1+k_2-N_{u_p}}}
{\zeta_2-(1-r^2)z_2-r^2z_1/t}
\,d\zeta_2
\\
& - &
{\bf 1}_{l=l_p}\,
\frac{1}{u_p!}
\frac{\partial^{u_p}}{\partial t^{u_p}}|_{t=\eta_p}
\frac{t^{k_1}((1-r^2)+\eta_pr^2/t)}
{\prod_{j=p+1}^{n-1}(t-\eta_j)^{m_j}}
\frac{1}{2\pi i}
\int_{|\zeta_2|=\sqrt{1-r^2}}
\frac{\zeta_2^{k_1+k_2-N_{u_p}}}
{\zeta_2-(1-r^2)z_2-r^2z_1/t}
d\zeta_2\,.
\\
\end{eqnarray*}

Now we know by lemma \ref{residup}
that, for all
$z\in\mathbb{B}_2$,
for all
$r$ close to 
$\alpha_s,\;s=2,\ldots,n-1$
and all
$t$ close to
$\eta_s$,
one has
$|(1-r^2)z_2+r^2z_1/t|<\sqrt{1-r^2}$.
It follows that
\begin{eqnarray*}
\lim_{\varepsilon\rightarrow0}
\sum_{l=l_p}^{\widetilde{n}-1}
\left(
\int_{r=\alpha_{q_l}+\varepsilon}
-\int_{r=\alpha_{q_l}-\varepsilon}
\right)
& = &
{\bf 1}_{k_1+k_2\geq N_{u_p}}
\sum_{l=l_p}^{\widetilde{n}-1}
\sum_{p+1\leq v\leq n-1,\alpha_v=\alpha_{q_l}}
\frac{1}{(m_v-1)!}
\;\times
\end{eqnarray*}
\begin{eqnarray*}
\times
\frac{\partial^{m_v-1}}{\partial t^{m_v-1}}|_{t=\eta_v}
\left[
\frac{t^{k_1}
((1-\alpha_{q_l}^2)+\eta_p\alpha_{q_l}^2/t)}
{(t-\eta_p)^{u_p+1}
\prod_{j=p+1,j\neq p}^{n-1}(t-\eta_j)^{m_j}}
\left((1-\alpha_{q_l}^2)z_2+\alpha_{q_l}^2z_1/t
\right)^{k_1+k_2-N_{u_p}}
\right]
\end{eqnarray*}
\begin{eqnarray*}
+\;
{\bf 1}_{k_1+k_2\geq N_{u_p}}
\frac{1}{u_p!}
\frac{\partial^{u_p}}{\partial t^{u_p}}|_{t=\eta_p}
\left[
\frac{t^{k_1}((1-\alpha_{q_{l_p}}^2)+\eta_p\alpha_{q_{l_p}}^2/t)}
{\prod_{j=p+1}^{n-1}(t-\eta_j)^{m_j}}
\left((1-\alpha_{q_{l_p}}^2)z_2+\alpha_{q_{l_p}}^2z_1/t
\right)^{k_1+k_2-N_{u_p}}
\right]
\end{eqnarray*}
\begin{eqnarray}\label{etapp}
& &
=\;
{\bf 1}_{k_1+k_2\geq N_{u_p}}
\;\{\;
\sum_{v=p+1}^{n-1}
\frac{1}{(m_v-1)!}
\;\times
\;\;\;\;\;\;\;\;\;\;\;\;\;\;\;\;\;\;\;\;\;\;\;\;\;\;\;\;\;\;\;\;\;\;\;\;\;\;\;\;\;\;\;\;\;\;\;\;\;\;\;\;\;\;\;\;\;
\;\;\;
\end{eqnarray}
\begin{eqnarray*}
\times\;
\frac{\partial^{m_v-1}}{\partial t^{m_v-1}}|_{t=\eta_v}
\left[
\frac{t^{k_1}
((1-\alpha_v^2)+\eta_p\alpha_v^2/t)}
{(t-\eta_p)^{u_p+1}
\prod_{j=p+1,j\neq p}^{n-1}(t-\eta_j)^{m_j}}
\left((1-\alpha_v^2)z_2+\alpha_v^2z_1/t
\right)^{k_1+k_2-N_{u_p}}
\right]
\end{eqnarray*}
\begin{eqnarray*}
& + &
\frac{1}{u_p!}
\frac{\partial^{u_p}}{\partial t^{u_p}}|_{t=\eta_p}
\left[
\frac{t^{k_1}((1-\alpha_p^2)+\eta_p\alpha_p^2/t)}
{\prod_{j=p+1}^{n-1}(t-\eta_j)^{m_j}}
\left((1-\alpha_p^2)z_2+\alpha_p^2z_1/t
\right)^{k_1+k_2-N_{u_p}}
\right]
\;\}\;.
\\
\end{eqnarray*}

Notice that we used the continuity on $r$
of $P_r(X)$. One has else
$X^{k_1}\left((1-r^2)X+\eta_pr^2\right)
=(1-r^2)X^{k_1+1}+\eta_pr^2X^{k_1}$
and one can separately deal with each integral
as above.
\bigskip

On the other hand, 
if
$k_1<u_p+\cdots+m_{n-1}$ or
$k_1+k_2<N_{u_p}$,
we have
\begin{eqnarray*}
\frac{1}{2\pi i}
\int_{|\zeta_2|=\sqrt{1-r^2}}
\frac{\zeta_2^{k_1+k_2-N_{u_p}}}
{\zeta_2-(1-r^2)z_2}
P_r\left(\frac{r^2z_1}{\zeta_2-(1-r^2)z_2}\right)
\,d\zeta_2
& = &
0\,.
\\
\end{eqnarray*}
It follows 
by (\ref{infinip})
that,
for all
$k_1\geq u_p+\cdots+m_{n-1}$,
for all
$k_1+k_2\geq N_{u_p}$,
for all
$\varepsilon>0$ small enough
and all
$z\in U_{\eta}$,
\begin{eqnarray*}
\frac{1}{(2\pi i)^2}
\int_{r=1-\varepsilon}
\frac{
(\overline{\zeta}_2+\eta_p\overline{\zeta}_1)
\,\zeta_1^{k_1}\zeta_2^{k_2-m_n}\,\omega(\zeta)}
{(\zeta_1-\eta_p\zeta_2)^{u_p+1}
\prod_{j=p+1}^{n-1}
(\zeta_1-\eta_j\zeta_2)^{m_j}
(1-<\overline{\zeta},z>)}
& = &
\;\;\;\;\;\;\;\;\;\;\;\;\;\;\;\;\;\;\;\;\;\;\;\;\;\;\;\;\;\;
\end{eqnarray*}
\begin{eqnarray*}
& = &
\frac{1}{2\pi i}
\int_{|\zeta_2|=\sqrt{1-r^2}}
\frac{\zeta_2^{k_2-m_n}}{\zeta_2-(1-r^2)z_2}
\;d\zeta_2
\;\times
\\
& &
\;\;\;\;\;\;\;\;\;
\times\;
\frac{1}{2\pi i}
\int_{|\zeta_1|=+\infty}
\frac{
\zeta_1^{k_1}
\left(
(1-r^2)\zeta_1+\eta_pr^2\zeta_2
\right)
\,d\zeta_1}
{(\zeta_1-\eta_p\zeta_2)^{u_p+1}
\prod_{j=p+1}^{n-1}
(\zeta_1-\eta_j\zeta_2)^{m_j}
\left(
\zeta_1
-\frac{r^2z_1\zeta_2}{\zeta_2-(1-r^2)z_2}
\right)}
\\
\\
& = &
\frac{1}{2\pi i}
\int_{|\zeta_2|=\sqrt{1-r^2}}
\frac{\zeta_2^{k_1+k_2-N_{u_p}}}
{\zeta_2-(1-r^2)z_2}
P_r\left(\frac{r^2z_1}{\zeta_2-(1-r^2)z_2}\right)
\,d\zeta_2
\end{eqnarray*}
\begin{eqnarray*}
& = &
\lim_{\varepsilon'\rightarrow0}
\frac{1}{2\pi i}
\int_{|\zeta_2-(1-r^2)z_2|=\varepsilon'}
\frac{\zeta_2^{k_1+k_2-N_{u_p}}}
{\zeta_2-(1-r^2)z_2}
\;\times
\\
& &
\times\;
\frac{
\left(
\frac{r^2z_1}{\zeta_2-(1-r^2)z_2}
\right)^{k_1}
\left(
\frac{(1-r^2)r^2z_1}{\zeta_2-(1-r^2)z_2}+\eta_pr^2
\right)
\,d\zeta_2}
{\left(
\frac{r^2z_1}{\zeta_2-(1-r^2)z_2}-\eta_p
\right)^{u_p+1}
\prod_{j=p+1}^{n-1}
\left(
\frac{r^2z_1}{\zeta_2-(1-r^2)z_2}-\eta_j
\right)^{m_j}}
\;\;\;\;\;\;\;\;\;\;\;\;\;\;\;\;\;\;\;\;\;\;\;\;\;\;\;\;\;\;\;\;\;\;\;\;
\end{eqnarray*}
\begin{eqnarray*}
& + &
\frac{1}{u_p!}
\frac{\partial^{u_p}}{\partial t^{u_p}}|_{t=\eta_p}
\frac{t^{k_1}((1-r^2)t+\eta_pr^2)}
{\prod_{j=p+1}^{n-1}(t-\eta_j)^{m_j}}
\lim_{\varepsilon'\rightarrow0}
\frac{1}{2\pi i}
\int_{|\zeta_2-(1-r^2)z_2|=\varepsilon'}
\frac{\zeta_2^{k_1+k_2-N_{u_p}}}
{t(\zeta_2-(1-r^2)z_2)-r^2z_1}
d\zeta_2
\\
& + &
\sum_{v=p+1}^{n-1}
\frac{1}{(m_v-1)!}
\frac{\partial^{m_v-1}}{\partial t^{m_v-1}}|_{t=\eta_v}
\frac{t^{k_1}((1-r^2)t+\eta_pr^2)}
{(t-\eta_p)^{u_p+1}
\prod_{j=p+1,j\neq p}^{n-1}(t-\eta_j)^{m_j}}
\;\times
\\
& & 
\;\;\;\;\;\;\;\;\;\;\;\;\;\;\;\;\;\;\;\;\;\;\;\;\;\;\;\;\;\;\;\;\;\;\;
\times\;
\lim_{\varepsilon'\rightarrow0}
\frac{1}{2\pi i}
\int_{|\zeta_2-(1-r^2)z_2|=\varepsilon'}
\frac{\zeta_2^{k_1+k_2-N_{u_p}}}
{t(\zeta_2-(1-r^2)z_2)-r^2z_1}
\,d\zeta_2
\end{eqnarray*}
\begin{eqnarray*}
& = &
r^2(r^2z_1)^{k_1}
\lim_{\varepsilon'\rightarrow0}
\frac{1}{2\pi i}
\int_{|\zeta_2-(1-r^2)z_2|=\varepsilon'}
\frac{d\zeta_2}
{(\zeta_2-(1-r^2)z_2)^{k_1-(u_p+\cdots+m_{n-1})+1}}
\;\times
\\
& & 
\;\;\;\;
\times\;
\frac{
\zeta_2^{k_1+k_2-N_{u_p}}
\left((1-r^2)z_1+\eta_p(\zeta_2-(1-r^2)z_2)
\right)}
{\left(r^2z_1-\eta_p(\zeta_2-(1-r^2)z_2)\right)^{u_p+1}
\prod_{j=p+1}^{n-1}
\left(r^2z_1-\eta_j(\zeta_2-(1-r^2)z_2)\right)^{m_j}}
\\
& + &
0
\end{eqnarray*}
\begin{eqnarray*}
& = &
r^2(r^2z_1)^{k_1}
\sum_{v_p+\cdots+v_n=k_1-(u_p+\cdots+m_{n-1})}
\frac{(u_p+v_p)!}{u_p!\,v_p!}
\;\times
\;\;\;\;\;\;\;\;\;\;\;\;\;\;\;\;\;\;\;\;\;\;\;\;\;\;\;\;\;\;\;\;\;\;\;\;\;\;\;\;\;\;\;\;\;\;\;\;\;\;\;\;\;\;\;\;
\\
& &
\times\;
\frac{\eta_p^{v_p}}
{(r^2z_1)^{u_p+v_p+1}}
\prod_{j=p+1}^{n-1}
\frac{(v_j+m_j-1)!}{v_j!\,(m_j-1)!}
\frac{\eta_j^{v_j}}
{(r^2z_1)^{v_j+m_j}}
\\
& &
\times\;
\frac{1}{v_n!}
\frac{\partial^{v_n}}{\partial\zeta_2^{v_n}}|_{\zeta_2=(1-r^2)z_2}
\left(
\zeta_2^{k_1+k_2-N_{u_p}}
\left(\eta_p\zeta_2+(1-r^2)(z_1-\eta_pz_2)\right)
\right)
\end{eqnarray*}
\begin{eqnarray*}
&  \xrightarrow[r=1-\varepsilon\rightarrow1]{} &
\sum_{v_p+\cdots+v_n=k_1-(u_p+\cdots+m_{n-1})\,,v_n\geq1}
\frac{(u_p+v_p)!}{u_p!\,v_p!}
\eta_p^{v_p}
\prod_{j=p+1}^{n-1}
\frac{(v_j+m_j-1)!}{v_j!\,(m_j-1)!}
\eta_j^{v_j}
\\
& \times &
\frac{(k_1+k_2-N_{u_p}+1)!}
{v_n!\,(k_1+k_2-N_{u_p}+1-v_n)!}
\,\eta_p\,z_1^{v_n-1}
(0)^{k_1+k_2-N_{u_p}+1-v_n}
\\
\end{eqnarray*}
\begin{eqnarray*}
& = &
{\bf 1}_{k_2\leq m_n-1}
\;\eta_p\,z_1^{k_1+k_2-N_{u_p}}
\\
& &
\;\;\;\;\;\;\;\;\;
\times\;
\sum_{v_p+\cdots+v_{n-1}=m_n-1-k_2}
\frac{(u_p+v_p)!}{u_p!\,v_p!}
\,\eta_p^{v_p}
\prod_{j=p+1}^{n-1}
\frac{(v_j+m_j-1)!}{v_j!\,(m_j-1)!}
\,\eta_j^{v_j}
\,
\end{eqnarray*}
(since if
$k_2\geq m_n$ then
$k_1+k_2-(u_p+\cdots+m_n)+1-v_n\geq
k_1-(u_p+\cdots+m_{n-1})-v_n+1\geq1$).
By lemma \ref{lagrangeuclide},
this gives
\begin{eqnarray*}
{\bf 1}_{k_2\leq m_n-1}
\;\eta_p\,z_1^{k_1+k_2-N_{u_p}}
Q\left(
X^{N_{u_p}-k_2},
(X-\eta_p)^{u_p+1}
\prod_{j=p+1}^{n-1}(X-\eta_j)^{m_j}
\right)
|_{X=0}
\end{eqnarray*}
\begin{eqnarray*}
\;\;\;\;\;\;\;\;\;
& = &
{\bf 1}_{k_2\leq m_n-1}
\;\eta_p\,z_1^{k_1+k_2-N_{u_p}}
\;\times
\\
& &
\times
\left[
\frac{X^{N_{u_p}-k_2}
-R\left(
X^{N_{u_p}-k_2},
(X-\eta_p)^{u_p+1}
\prod_{j=p+1}^{n-1}(X-\eta_j)^{m_j}
\right)}
{(X-\eta_p)^{u_p+1}
\prod_{j=p+1}^{n-1}(X-\eta_j)^{m_j}}
\right]
|_{X=0}
\end{eqnarray*}
\begin{eqnarray}\label{etap1}
&  &
=\;
{\bf 1}_{k_2\leq m_n-1}
\;\eta_p\,z_1^{k_1+k_2-N_{u_p}}
\;\times\;\{\;
\frac{1}{u_p!}
\frac{\partial^{u_p}}{\partial t^{u_p}}|_{t=\eta_p}
\left[
\frac{t^{N_{u_p}-k_2-1}}
{\prod_{j=p+1}^{n-1}(t-\eta_j)^{m_j}}
\right]
\\\nonumber
& &
+\;
\sum_{q=p+1}^{n-1}
\frac{1}{(m_q-1)!}
\frac{\partial^{m_q-1}}{\partial t^{m_q-1}}|_{t=\eta_q}
\left[
\frac{t^{N_{u_p}-k_2-1}}
{(t-\eta_p)^{u_p+1}
\prod_{j=p+1,j\neq q}^{n-1}(t-\eta_j)^{m_j}}
\right]
\;\}
\end{eqnarray}

\bigskip

Similarly, we have 
\begin{eqnarray*}
\frac{1}{(2\pi i)^2}
\int_{r=\varepsilon}
\frac{
(\overline{\zeta}_2+\eta_p\overline{\zeta}_1)
\,\zeta_1^{k_1}\zeta_2^{k_2-m_n}\,\omega(\zeta)}
{(\zeta_1-\eta_p\zeta_2)^{u_p+1}
\prod_{j=p+1}^{n-1}
(\zeta_1-\eta_j\zeta_2)^{m_j}
(1-<\overline{\zeta},z>)}
& = &
\;\;\;\;\;\;\;\;\;\;\;\;\;\;\;\;\;\;\;\;
\end{eqnarray*}
\begin{eqnarray*}
& = &
{\bf 1}_{k_1\geq u_p+\cdots+m_{n-1}}
\;\frac{1}{2\pi i}
\int_{|\zeta_2|=\sqrt{1-r^2}}
\frac{\zeta_2^{k_1+k_2-N_{u_p}}}
{\zeta_2-(1-r^2)z_2}
P_r\left(\frac{r^2z_1}{\zeta_2-(1-r^2)z_2}\right)
\,d\zeta_2
\\
& - &
\frac{1}{u_p!}
\frac{\partial^{u_p}}{\partial t^{u_p}}|_{t=\eta_p}
\frac{t^{k_1}((1-r^2)+\eta_pr^2/t)}
{\prod_{j=p+1}^{n-1}(t-\eta_j)^{m_j}}
\frac{1}{2\pi i}
\int_{|\zeta_2|=\sqrt{1-r^2}}
\frac{\zeta_2^{k_1+k_2-N_{u_p}}}
{\zeta_2-(1-r^2)z_2-r^2z_1/t}
d\zeta_2
\\
& - &
\sum_{v=p+1}^{n-1}
\frac{1}{(m_v-1)!}
\frac{\partial^{m_v-1}}{\partial t^{m_v-1}}|_{t=\eta_v}
\frac{t^{k_1}((1-r^2)+\eta_pr^2/t)}
{(t-\eta_p)^{u_p+1}
\prod_{j=p+1,j\neq v}^{n-1}(t-\eta_j)^{m_j}}
\;\times
\\
& &
\;\;\;\;\;\;\;\;\;\;\;\;\;\;\;\;\;\;\;\;\;\;\;\;\;\;\;\;\;\;\;\;\;\;\;
\times\;
\frac{1}{2\pi i}
\int_{|\zeta_2|=\sqrt{1-r^2}}
\frac{\zeta_2^{k_1+k_2-N_{u_p}}}
{\zeta_2-(1-r^2)z_2-r^2z_1/t}
\,d\zeta_2
\end{eqnarray*}
\begin{eqnarray*}
& = &
{\bf 1}_{k_1\geq u_p+\cdots+m_{n-1},\,k_1+k_2\geq N_{u_p}}
\;\times
\\
& \times &
\sum_{v_p+\cdots+v_n=k_1-(u_p+\cdots+m_{n-1})}
\frac{(u_p+v_p)!}{u_p!\,v_p!}
\eta_p^{v_p}
\prod_{j=p+1}^{n-1}
\frac{(v_j+m_j-1)!}{v_j!\,(m_j-1)!}
\eta_j^{v_j}
\\
& \times &
r^2(r^2z_1)^{v_n-1}
\frac{1}{v_n!}
\frac{\partial^{v_n}}{\partial\zeta_2^{v_n}}|_{\zeta_2=(1-r^2)z_2}
\left(
\zeta_2^{k_1+k_2-N_{u_p}}
\left(\eta_p\zeta_2+(1-r^2)(z_1-\eta_pz_2)\right)
\right)
\end{eqnarray*}
\begin{eqnarray*}
& - &
{\bf 1}_{k_1+k_2\geq N_{u_p}}
\frac{1}{u_p!}
\frac{\partial^{u_p}}{\partial t^{u_p}}|_{t=\eta_p}
\frac{t^{k_1}((1-r^2)+\eta_pr^2/t)}
{\prod_{j=p+1}^{n-1}(t-\eta_j)^{m_j}}
\left(
(1-r^2)z_2+r^2z_1/t
\right)^{k_1+k_2-N_{u_p}}
\\
& - &
{\bf 1}_{k_1+k_2\geq N_{u_p}}
\sum_{v=p+1}^{n-1}
\frac{1}{(m_v-1)!}
\;\times
\\
& &
\times\;
\frac{\partial^{m_v-1}}{\partial t^{m_v-1}}|_{t=\eta_v}
\frac{t^{k_1}((1-r^2)+\eta_pr^2/t)}
{(t-\eta_p)^{u_p+1}
\prod_{j=p+1,j\neq v}^{n-1}(t-\eta_j)^{m_j}}
\left(
(1-r^2)z_2+r^2z_1/t
\right)^{k_1+k_2-N_{u_p}}
\end{eqnarray*}
\begin{eqnarray*}
& \xrightarrow[r=\varepsilon\rightarrow0]{} &
{\bf 1}_{k_1+k_2\geq N_{u_p}}
\sum_{v_p+\cdots+v_n=k_1-(u_p+\cdots+m_{n-1}),v_n=0}
\frac{(u_p+v_p)!}{u_p!\,v_p!}
\eta_p^{v_p}
\prod_{j=p+1}^{n-1}
\frac{(v_j+m_j-1)!}{v_j!\,(m_j-1)!}
\eta_j^{v_j}
\\
& & 
\;\;\times\;
1\times
z_2^{k_1+k_2-N_{u_p}}
\\
& &
-\;
{\bf 1}_{k_1+k_2\geq N_{u_p}}
z_2^{k_1+k_2-N_{u_p}}
\times\{\;
\frac{1}{u_p!}
\frac{\partial^{u_p}}{\partial t^{u_p}}|_{t=\eta_p}
\frac{t^{k_1}}
{\prod_{j=p+1}^{n-1}(t-\eta_j)^{m_j}}
\\
& & 
+\;
\sum_{v=p+1}^{n-1}
\frac{1}{(m_v-1)!}
\frac{\partial^{m_v-1}}{\partial t^{m_v-1}}|_{t=\eta_v}
\frac{t^{k_1}}
{(t-\eta_p)^{u_p+1}
\prod_{j=p+1,j\neq v}^{n-1}(t-\eta_j)^{m_j}}
\;\}
\end{eqnarray*}
\begin{eqnarray}\label{etap0}
& = &
0\,,
\;\;\;\;\;\;\;\;\;\;\;\;\;\;\;\;\;\;\;\;\;\;\;\;\;\;\;\;\;\;\;\;\;\;\;\;\;\;\;\;\;\;\;\;\;\;\;\;\;\;\;\;\;\;\;\;\;\;\;
\;\;\;\;\;\;\;\;\;\;\;\;\;\;\;\;\;\;\;\;\;\;\;\;\;\;\;\;\;\;\;\;\;\;\;\;\;\;\;\;
\\\nonumber
\end{eqnarray}
the last equality coming from 
lemma \ref{lagrangeuclide} 
since
\begin{eqnarray*}
\sum_{v_p+\cdots+v_{n-1}=k_1+1-(u_p+1+m_{p+1}+\cdots+m_{n-1})}
\frac{(u_p+v_p)!}{u_p!\,v_p!}
\eta_p^{v_p}
\prod_{j=p+1}^{n-1}
\frac{(v_j+m_j-1)!}{v_j!\,(m_j-1)!}
\eta_j^{v_j}
& &
\\
+\;
\frac{1}{u_p!}
\frac{\partial^{u_p}}{\partial t^{u_p}}|_{t=\eta_p}
\frac{t^{k_1+1}}
{\prod_{j=p+1}^{n-1}(t-\eta_j)^{m_j}(-t)}
\;\;\;\;\;\;\;\;\;
& &
\\
+\;
\sum_{v=p+1}^{n-1}
\frac{1}{(m_v-1)!}
\frac{\partial^{m_v-1}}{\partial t^{m_v-1}}|_{t=\eta_v}
\frac{t^{k_1+1}}
{(t-\eta_p)^{u_p+1}
\prod_{j=p+1,j\neq v}^{n-1}(t-\eta_j)^{m_j}(-t)}
& = &
\end{eqnarray*}
\begin{eqnarray*}
=
\left[
Q\left(
X^{k_1+1},
(X-\eta_p)^{u_p+1}
\prod_{j=p+1}^{n-1}(X-\eta_j)^{m_j}
\right)
+\frac{
R\left(
X^{k_1+1},
(X-\eta_p)^{u_p+1}
\prod_{j=p+1}^{n-1}(X-\eta_j)^{m_j}
\right)}
{(X-\eta_p)^{u_p+1}
\prod_{j=p+1}^{n-1}(X-\eta_j)^{m_j}}
\right]
|_{X=0}
\end{eqnarray*}
\begin{eqnarray*}
& = &
\left[
\frac{X^{k_1+1}}
{(X-\eta_p)^{u_p+1}
\prod_{j=p+1}^{n-1}(X-\eta_j)^{m_j}}
\right]
|_{X=0}
\;\;\;\;\;\;\;\;\;\;\;\;\;\;\;\;\;\;\;\;\;\;\;\;\;\;\;\;\;\;\;\;\;\;\;\;\;\;\;\;\;\;\;\;\;\;\;\;\;\;\;\;\;\;\;\;\;\;\;\;
\\
\\
& = &
0\,.
\\
\end{eqnarray*}

Finally, 
by (\ref{etap0p1}),
(\ref{etapp}),
(\ref{etap1}), and 
(\ref{etap0}),
we get for all
$k_1+k_2\geq N_{u_p}$
(else we get zero)
and all
$z\in U_{\eta}$
\begin{eqnarray*}
\lim_{\varepsilon\rightarrow0}
\frac{1}{(2\pi i)^2}
\int_{\partial\widetilde{\Sigma}_{\varepsilon}}
\frac{
(\overline{\zeta}_2+\eta_p\overline{\zeta}_1)
\,\zeta_1^{k_1}\zeta_2^{k_2-m_n}\,\omega(\zeta)}
{(\zeta_1-\eta_p\zeta_2)^{u_p+1}
\prod_{j=p+1}^{n-1}
(\zeta_1-\eta_j\zeta_2)^{m_j}
(1-<\overline{\zeta},z>)}
& = &
\;\;\;\;\;\;\;\;\;\;\;\;\;\;\;\;\;\;\;\;\;
\end{eqnarray*}
\begin{eqnarray*}
& = &
{\bf 1}_{k_2\leq m_n-1}
\;\eta_p\,z_1^{k_1+k_2-N_{u_p}}
\;\times\;\{\;
\frac{1}{u_p!}
\frac{\partial^{u_p}}{\partial t^{u_p}}|_{t=\eta_p}
\left[
\frac{t^{N_{u_p}-k_2-1}}
{\prod_{j=p+1}^{n-1}(t-\eta_j)^{m_j}}
\right]
\\\nonumber
& &
\;\;\;\;
+\;
\sum_{q=p+1}^{n-1}
\frac{1}{(m_q-1)!}
\frac{\partial^{m_q-1}}{\partial t^{m_q-1}}|_{t=\eta_q}
\left[
\frac{t^{N_{u_p}-k_2-1}}
{(t-\eta_p)^{u_p+1}
\prod_{j=p+1,j\neq q}^{n-1}(t-\eta_j)^{m_j}}
\right]
\;\}
\end{eqnarray*}
\begin{eqnarray*}
& - &
\frac{1}{u_p!}
\frac{\partial^{u_p}}{\partial t^{u_p}}|_{t=\eta_p}
\left[
\frac{t^{k_1}\,
\frac{1+|\eta_p|^2\eta_p/t}{1+|\eta_p|^2}}
{\prod_{j=p+1}^{n-1}(t-\eta_j)^{m_j}}
\left(
\frac{z_2+|\eta_p|^2z_1/t}{1+|\eta_p|^2}
\right)^{k_1+k_2-N_{u_p}}
\right]
\\
& - &
\sum_{v=p+1}^{n-1}
\frac{1}{(m_v-1)!}
\frac{\partial^{m_v-1}}{\partial t^{m_v-1}}|_{t=\eta_v}
\left[
\frac{t^{k_1}\,
\frac{1+|\eta_v|^2\eta_p/t}{1+|\eta_v|^2}
\left(
\frac{z_2+|\eta_v|^2z_1/t}{1+|\eta_v|^2}
\right)^{k_1+k_2-N_{u_p}}}
{(t-\eta_p)^{u_p+1}
\prod_{j=p+1,j\neq p}^{n-1}(t-\eta_j)^{m_j}}
\right]
\\
\\
& - &
0
\\
\end{eqnarray*}
and the lemma is proved.

\end{proof}

\bigskip

Lastly, we prove the following lemma.

\begin{lemma}\label{eta1}

For all
$z\in U_{\eta}$
and all
$u_1=0,\ldots,m_1-1$, we have

\begin{eqnarray*}
\lim_{\varepsilon\rightarrow0}
\frac{1}{(2\pi i)^2}
\int_{\partial\widetilde{\Sigma}_{\varepsilon}}
\frac{
\overline{\zeta}_2\,\zeta_2^{k_2-m_n}
\zeta_1^{k_1-u_1-1}
\;\omega(\zeta)}
{\prod_{j=2}^{n-1}(\zeta_1-\eta_j\zeta_2)^{m_j}
\,(1-<\overline{\zeta},z>)}
& = &
\;\;\;\;\;\;\;\;\;\;\;\;\;\;\;\;\;\;\;\;\;\;\;\;\;\;\;\;\;\;\;\;\;\;\;\;\;\;\;\;\;\;\;\;\;
\end{eqnarray*}
\begin{eqnarray*}
& = &
{\bf 1}_{k_1+k_2\geq N_{u_1}}\;
\sum_{p=2}^{n-1}
\frac{1}{(m_p-1)!}
\frac{\partial^{m_p-1}}{\partial t^{m_p-1}}|_{t=\eta_p}
\;\frac{1}{t^{u_1+1}\prod_{j=2,j\neq p}^{n-1}(t-\eta_j)^{m_j}}
\times
\\
&  &
\times\;
\{\;
{\bf 1}_{k_1\leq u_1}\;
t^{k_1}z_2^{k_1+k_2-N_{u_1}}
-\;
\frac{t^{k_1}}{1+|\eta_p|^2}
\left(
\frac{z_2+|\eta_p|^2z_1/t}{1+|\eta_p|^2}
\right)^{k_1+k_2-N_{u_1}}
\;\;\;\}
\;.
\\
\\
\end{eqnarray*}

\end{lemma}

\bigskip

\begin{proof}

The proof is analogous to the one of the previous lemma. 
We want to calculate
\begin{eqnarray*}
\lim_{\varepsilon\rightarrow0}
\frac{1}{(2\pi i)^2}
\left[
\int_{r=1-\varepsilon}
-\sum_{l=2}^{\widetilde{n}-1}
\left(
\int_{r=\alpha_{q_l}+\varepsilon}
-\int_{r=\alpha_{q_l}-\varepsilon}
\right)
-\int_{r=\varepsilon}
\right]
\;\;\;\;\;\;\;\;\;\;\;\;\;\;\;\;\;\;\;\;\;\;\;
\\
\;\;\;\;\;\;\;\;\;\;\;\;\;\;\;\;\;\;\;\;\;\;\;\;\;
\times\;
\frac{(1-r^2)\,\zeta_2^{k_2-m_n}}
{\zeta_2-(1-r^2)z_2}\;
\frac{\zeta_1^{k_1-u_1}}
{\prod_{j=2}^{n-1}(\zeta_1-\eta_j\zeta_2)^{m_j}
\left(
\zeta_1-
\frac{r^2z_1\zeta_2}{\zeta_2-(1-r^2)z_2}
\right)}
\;\omega(\zeta)
\;.
\\
\end{eqnarray*}

First, we have similarly
\begin{eqnarray}\label{eta1infini}
\frac{1}{2\pi i}
\int_{|\zeta_1|=+\infty}
\frac{\zeta_1^{k_1-u_1}}
{\prod_{j=2}^{n-1}(\zeta_1-\eta_j\zeta_2)^{m_j}
\left(
\zeta_1-
\frac{r^2z_1\zeta_2}{\zeta_2-(1-r^2)z_2}
\right)}
\,d\zeta_1
& = &
\;\;\;\;
\end{eqnarray}
\begin{eqnarray*}
\;\;\;\;\;\;\;\;\;\;\;\;\;\;\;\;\;
& = &
{\bf 1}_{k_1\geq u_1+m_2+\cdots+m_{n-1}}
\,\zeta_2^{k_1-(u_1+m_2+\cdots+m_{n-1})}
P\left(
\frac{r^2z_1}{\zeta_2-(1-r^2)z_2}
\right)\,,
\end{eqnarray*}
with the following quotient
\begin{eqnarray*}
P(X) 
& = &
Q\left(
X^{k_1-u_1},
\prod_{j=2}^{n-1}(X-\eta_j)^{m_j}
\right)
\,
\\
\end{eqnarray*}
(notice that if
$k_1\geq u_1+m_2+\cdots+m_{n-1}$
then in particular
$k_1\geq u_1$).

It follows that, for all
$l=1,\ldots,\widetilde{n}-1$
and for all small enough
$\varepsilon>0$,
we have
\begin{eqnarray*}
\frac{1}{2\pi i}
\int_{r=\alpha_{q_l}+\varepsilon}
\frac{\zeta_1^{k_1-u_1}}
{\prod_{j=2}^{n-1}(\zeta_1-\eta_j\zeta_2)^{m_j}
\left(
\zeta_1-
\frac{r^2z_1\zeta_2}{\zeta_2-(1-r^2)z_2}
\right)}
\,d\zeta_1
& = &
\;\;\;\;\;\;\;\;\;\;\;\;\;\;\;\;\;\;\;\;\;\;\;\;\;
\end{eqnarray*}
\begin{eqnarray*}
& = &
{\bf 1}_{k_1\geq u_1+m_2+\cdots+m_{n-1}}
\,\zeta_2^{k_1-(u_1+m_2+\cdots+m_{n-1})}
P\left(
\frac{r^2z_1}{\zeta_2-(1-r^2)z_2}
\right)
\\
& - &
\zeta_2^{k_1-(u_1+m_2+\cdots+m_{n-1})}
\sum_{v\leq n-1,\alpha_v>\alpha_{q_l}}
\frac{1}{(m_v-1)!}
\;\times
\\
& &
\;\times\;
\frac{\partial^{m_v-1}}
{\partial t^{m_v-1}}|_{t=\eta_v}
\left[
\frac{t^{k_1-u_1}}
{\prod_{j=2,j\neq v}^{n-1}
(t-\eta_j)^{m_j}
\left(
t-\frac{r^2z_1}{\zeta_2-(1-r^2)z_2}
\right)}
\right]
\\
\end{eqnarray*}
then by lemma \ref{residup},
for
$r=\alpha_{q_l}+\varepsilon$
and all
$z\in U_{\eta}$,
\begin{eqnarray*}
\frac{1}{(2\pi i)^2}
\int_{r=\alpha_{q_l}+\varepsilon}
\frac{(1-r^2)\,\zeta_2^{k_2-m_n}}
{\zeta_2-(1-r^2)z_2}\;
\frac{\zeta_1^{k_1-u_1}}
{\prod_{j=2}^{n-1}(\zeta_1-\eta_j\zeta_2)^{m_j}
\left(
\zeta_1-
\frac{r^2z_1\zeta_2}{\zeta_2-(1-r^2)z_2}
\right)}
\,\omega(\zeta)
& = &
\end{eqnarray*}
\begin{eqnarray*}
=\;
{\bf 1}_{k_1\geq u_1+m_2+\cdots+m_{n-1}}
\frac{1-r^2}{2\pi i}
\int_{|\zeta_2|=\sqrt{1-r^2}}
\frac{\zeta_2^{k_1+k_2-N_{u_1}}}
{\zeta_2-(1-r^2)z_2}
P\left(
\frac{r^2z_1}{\zeta_2-(1-r^2)z_2}
\right)\,
d\zeta_2
\;
\end{eqnarray*}
\begin{eqnarray*}
& - &
\sum_{v\leq n-1,\alpha_v>\alpha_{q_l}}
\frac{1-r^2}{(m_v-1)!}
\;\times
\\
& &
\frac{\partial^{m_v-1}}
{\partial t^{m_v-1}}|_{t=\eta_v}
\frac{t^{k_1-u_1-1}}
{\prod_{j=2,j\neq v}^{n-1}
(t-\eta_j)^{m_j}}
\frac{1}{2\pi i}
\int_{|\zeta_2|=\sqrt{1-r^2}}
\frac{\zeta_2^{k_1+k_2-N_{u_1}}\,d\zeta_2}
{\zeta_2-(1-r^2)z_2-r^2z_1/t}
\\
\end{eqnarray*}
\begin{eqnarray*}
& = &
{\bf 1}_{k_1\geq u_1+m_2+\cdots+m_{n-1},\,k_1+k_2\geq N_{u_1}}
\;(1-r^2)
\;\times
\;\;\;\;\;\;\;\;\;\;\;\;\;\;\;\;\;\;\;\;\;\;\;\;\;\;\;\;\;\;\;\;\;\;\;\;\;\;\;\;\;\;\;\;\;\;\;\;\;\;\;\;\;\;\;\;\;\;\;\;
\end{eqnarray*}
\begin{eqnarray*}
\{\;
\lim_{\varepsilon'\rightarrow0}
\frac{1}{2\pi i}
\int_{|\zeta_2-(1-r^2)z_2|=\varepsilon'}
\frac{\zeta_2^{k_1+k_2-N_{u_1}}
(r^2z_1)^{k_1-u_1}\;d\zeta_2}
{(\zeta_2-(1-r^2)z_2)^{k_1-(u_1+\cdots+m_{n-1})+1}
\prod_{j=2}^{n-1}
\left(
r^2z_1-\eta_j(\zeta_2-(1-r^2)z_2)
\right)^{m_j}}
\end{eqnarray*}
\begin{eqnarray*}
& &
+\;0\;\}
\\
& - &
{\bf 1}_{k_1+k_2\geq N_{u_1}}
\sum_{v\leq n-1,\alpha_v>\alpha_{q_l}}
\frac{1-r^2}{(m_v-1)!}
\;\times
\\
& &
\times\;
\frac{\partial^{m_v-1}}
{\partial t^{m_v-1}}|_{t=\eta_v}
\left[
\frac{t^{k_1-u_1-1}}
{\prod_{j=2,j\neq v}^{n-1}
(t-\eta_j)^{m_j}}
\left(
(1-r^2)z_2+r^2z_1/t
\right)^{k_1+k_2-N_{u_1}}
\right]
\end{eqnarray*}
\begin{eqnarray*}
& = &
{\bf 1}_{k_1+k_2\geq N_{u_1}}
(1-r^2)
\sum_{v_1+\cdots+v_{n-1}=k_1-(u_1+\cdots+m_{n-1})}
\prod_{j=2}^{n-1}
\frac{(v_j+m_j-1)!}{v_j!\,(m_j-1)!}
\,\eta_j^{v_j}
\;\times
\\
& &
\times\;
\frac{(k_1+k_2-N_{u_1})!}
{v_1!\,(k_1+k_2-N_{u_1}-v_1)!}
\left(r^2z_1\right)^{v_1}
\left((1-r^2)z_2\right)^{k_1+k_2-N_{u_1}-v_1}
\\
& - &
{\bf 1}_{k_1+k_2\geq N_{u_1}}
\sum_{v\leq n-1,\alpha_v>\alpha_{q_l}}
\frac{1-r^2}{(m_v-1)!}
\;\times
\\
& &
\times\;
\frac{\partial^{m_v-1}}
{\partial t^{m_v-1}}|_{t=\eta_v}
\left[
\frac{t^{k_1-u_1-1}}
{\prod_{j=2,j\neq v}^{n-1}
(t-\eta_j)^{m_j}}
\left(
(1-r^2)z_2+r^2z_1/t
\right)^{k_1+k_2-N_{u_1}}
\right]
\,.
\\
\end{eqnarray*}

\medskip

Similarly, for all
$l=2,\ldots,\widetilde{n}$
and
$r=\alpha_{q_l}-\varepsilon$,
we have
\begin{eqnarray*}
\frac{1}{(2\pi i)^2}
\int_{r=\alpha_{q_l}-\varepsilon}
\frac{(1-r^2)\,\zeta_2^{k_2-m_n}}
{\zeta_2-(1-r^2)z_2}\;
\frac{\zeta_1^{k_1-u_1}}
{\prod_{j=2}^{n-1}(\zeta_1-\eta_j\zeta_2)^{m_j}
\left(
\zeta_1-
\frac{r^2z_1\zeta_2}{\zeta_2-(1-r^2)z_2}
\right)}
\,\omega(\zeta)
& = &
\end{eqnarray*}
\begin{eqnarray*}
& = &
{\bf 1}_{k_1+k_2\geq N_{u_1}}
(1-r^2)
\sum_{v_1+\cdots+v_{n-1}=k_1-(u_1+\cdots+m_{n-1})}
\prod_{j=2}^{n-1}
\frac{(v_j+m_j-1)!}{v_j!\,(m_j-1)!}
\,\eta_j^{v_j}
\;\times
\\
& &
\times\;
\frac{(k_1+k_2-N_{u_1})!}
{v_1!\,(k_1+k_2-N_{u_1}-v_1)!}
\left(r^2z_1\right)^{v_1}
\left((1-r^2)z_2\right)^{k_1+k_2-N_{u_1}-v_1}
\\
& - &
{\bf 1}_{k_1+k_2\geq N_{u_1}}
\sum_{v\leq n-1,\alpha_v\geq\alpha_{q_l}}
\frac{1-r^2}{(m_v-1)!}
\;\times
\\
& &
\times\;
\frac{\partial^{m_v-1}}
{\partial t^{m_v-1}}|_{t=\eta_v}
\left[
\frac{t^{k_1-u_1-1}}
{\prod_{j=2,j\neq v}^{n-1}
(t-\eta_j)^{m_j}}
\left(
(1-r^2)z_2+r^2z_1/t
\right)^{k_1+k_2-N_{u_1}}
\right]
\,.
\\
\end{eqnarray*}

\medskip

It follows that
\begin{eqnarray*}
\lim_{\varepsilon\rightarrow0}
\frac{1}{(2\pi i)^2}
\sum_{l=2}^{\widetilde{n}-1}
\left(
\int_{r=\alpha_{q_l}+\varepsilon}
-\int_{r=\alpha_{q_l}-\varepsilon}
\right)
\frac{\overline{\zeta}_2\,\zeta_2^{k_2-m_n}
\zeta_1^{k_1-u_1-1}\;\omega(\zeta)}
{\prod_{j=2}^{n-1}(\zeta_1-\eta_j\zeta_2)^{m_j}
\left(1-<\overline{\zeta},z>\right)}
& = &
\;\;\;\;\;
\end{eqnarray*}
\begin{eqnarray*}
& = &
{\bf 1}_{k_1+k_2\geq N_{u_1}}
\;\sum_{l=2}^{\widetilde{n}-1}
\;\sum_{2\leq v\leq n-1,\alpha_v=\alpha_{q_l}}
\frac{1-\alpha_{q_l}^2}{(m_v-1)!}
\;\times
\\
& \times &
\frac{\partial^{m_v-1}}
{\partial t^{m_v-1}}|_{t=\eta_v}
\left[
\frac{t^{k_1-u_1-1}}
{\prod_{j=2,j\neq v}^{n-1}
(t-\eta_j)^{m_j}}
\left(
(1-\alpha_{q_l}^2)z_2+\alpha_{q_l}^2z_1/t
\right)^{k_1+k_2-N_{u_1}}
\right]
\end{eqnarray*}
\begin{eqnarray}\label{eta1p}
& &
=\;
{\bf 1}_{k_1+k_2\geq N_{u_1}}
\;\sum_{v=2}^{n-1}
\frac{1-\alpha_v^2}{(m_v-1)!}
\;\times
%
\\\nonumber
& &
\;\;\times\; 
\frac{\partial^{m_v-1}}
{\partial t^{m_v-1}}|_{t=\eta_v}
\left[
\frac{t^{k_1-u_1-1}}
{\prod_{j=2,j\neq v}^{n-1}
(t-\eta_j)^{m_j}}
\left(
(1-\alpha_v^2)z_2+\alpha_v^2z_1/t
\right)^{k_1+k_2-N_{u_1}}
\right]
\,.
\end{eqnarray}

\bigskip

On the other hand,
\begin{eqnarray}\label{eta11}
\lim_{\varepsilon\rightarrow0}
\frac{1}{(2\pi i)^2}
\int_{r=1-\varepsilon}
\frac{\overline{\zeta}_2\,\zeta_2^{k_2-m_n}
\zeta_1^{k_1-u_1-1}\;\omega(\zeta)}
{\prod_{j=2}^{n-1}(\zeta_1-\eta_j\zeta_2)^{m_j}
\left(1-<\overline{\zeta},z>\right)}
& = &
\;\;\;\;\;\;\;\;\;\;\;
\end{eqnarray}
\begin{eqnarray*}
& = &
{\bf 1}_{k_1+k_2\geq N_{u_1}}
(1-r^2)
\sum_{v_1+\cdots+v_{n-1}=k_1-(u_1+\cdots+m_{n-1})}
\prod_{j=2}^{n-1}
\frac{(v_j+m_j-1)!}{v_j!\,(m_j-1)!}
\,\eta_j^{v_j}
\;\times
\\
& &
\times\;
\frac{(k_1+k_2-N_{u_1})!}
{v_1!\,(k_1+k_2-N_{u_1}-v_1)!}
\left(r^2z_1\right)^{v_1}
\left((1-r^2)z_2\right)^{k_1+k_2-N_{u_1}-v_1}
\\
\\
& \xrightarrow[r=1-\varepsilon\rightarrow1]{} &
0\,.
\\
\end{eqnarray*}

\bigskip

Lastly,
we have for
$r=\varepsilon$
\begin{eqnarray*}
\frac{1}{(2\pi i)^2}
\int_{r=\varepsilon}
\frac{\overline{\zeta}_2\,\zeta_2^{k_2-m_n}
\zeta_1^{k_1-u_1-1}\;\omega(\zeta)}
{\prod_{j=2}^{n-1}(\zeta_1-\eta_j\zeta_2)^{m_j}
\left(1-<\overline{\zeta},z>\right)}
& = &
\;\;\;\;\;\;\;\;\;\;\;\;\;\;\;\;\;\;\;\;\;\;\;\;\;\;\;\;\;\;\;\;\;
\end{eqnarray*}
\begin{eqnarray*}
& = &
{\bf 1}_{k_1+k_2\geq N_{u_1}}
(1-r^2)
\sum_{v_1+\cdots+v_{n-1}=k_1-(u_1+\cdots+m_{n-1})}
\prod_{j=2}^{n-1}
\frac{(v_j+m_j-1)!}{v_j!\,(m_j-1)!}
\,\eta_j^{v_j}
\;\times
\\
& &
\times\;
\frac{(k_1+k_2-N_{u_1})!}
{v_1!\,(k_1+k_2-N_{u_1}-v_1)!}
\left(r^2z_1\right)^{v_1}
\left((1-r^2)z_2\right)^{k_1+k_2-N_{u_1}-v_1}
\\
& - &
{\bf 1}_{k_1+k_2\geq N_{u_1}}
\sum_{v=2}^{n-1}
\frac{1-r^2}{(m_v-1)!}
\;\times
\\
& &
\times\;
\frac{\partial^{m_v-1}}
{\partial t^{m_v-1}}|_{t=\eta_v}
\left[
\frac{t^{k_1-u_1-1}}
{\prod_{j=2,j\neq v}^{n-1}
(t-\eta_j)^{m_j}}
\left(
(1-r^2)z_2+r^2z_1/t
\right)^{k_1+k_2-N_{u_1}}
\right]
\end{eqnarray*}
\begin{eqnarray*}
& \xrightarrow[r=\varepsilon\rightarrow0]{} &
{\bf 1}_{k_1+k_2\geq N_{u_1}}
\,z_2^{k_1+k_2-N_{u_1}}
\sum_{v_2+\cdots+v_{n-1}=k_1-(u_1+\cdots+m_{n-1})}
\prod_{j=2}^{n-1}
\frac{(v_j+m_j-1)!}{v_j!\,(m_j-1)!}
\,\eta_j^{v_j}
\\
\\
& &
-\;
{\bf 1}_{k_1+k_2\geq N_{u_1}}
\,z_2^{k_1+k_2-N_{u_1}}
\sum_{v=2}^{n-1}
\frac{1}{(m_v-1)!}
\frac{\partial^{m_v-1}}
{\partial t^{m_v-1}}|_{t=\eta_v}
\left[
\frac{t^{k_1-u_1-1}}
{\prod_{j=2,j\neq v}^{n-1}
(t-\eta_j)^{m_j}}
\right]
\,.
\;\;\;\;\;\;\;\;\;\;\;\;\;\;\;\;\;\;\;\;\;\;\;\;\;\;\;\;\;\;\;\;\;\;\;\;\;\;\;\;\;\;\;\;\;\;\;\;\;\;\;\;
\\
\end{eqnarray*}

If 
$k_1\geq u_1+1$, 
we get by
lemma \ref{lagrangeuclide}
\begin{eqnarray*}
\sum_{v_2+\cdots+v_{n-1}=k_1-u_1-(m_2+\cdots+m_{n-1})}
\prod_{j=2}^{n-1}
\frac{(v_j+m_j-1)!}{v_j!\,(m_j-1)!}
\,\eta_j^{v_j}
\;\;\;\;\;\;\;\;\;\;\;\;\;\;\;\;
& &
\\
+\;\;
\sum_{v=2}^{n-1}
\frac{1}{(m_v-1)!}
\frac{\partial^{m_v-1}}
{\partial t^{m_v-1}}|_{t=\eta_v}
\left[
\frac{t^{k_1-u_1}}
{\prod_{j=2,j\neq v}^{n-1}
(t-\eta_j)^{m_j}
(-t)}
\right]
& = &
\;\;\;\;\;\;\;\;\;\;\;\;\;\;\;\;\;\;\;\;\;\;\;\;\;
\end{eqnarray*}
\begin{eqnarray*}
\;\;\;\;\;\;\;\;\;
& = &
\left[
Q\left(
X^{k_1-u_1},
\prod_{j=2}^{n-1}(X-\eta_j)^{m_j}
\right)
+\;
\frac{
R\left(
X^{k_1-u_1},
\prod_{j=2}^{n-1}(X-\eta_j)^{m_j}
\right)}
{\prod_{j=2}^{n-1}(X-\eta_j)^{m_j}}
\right]
|_{X=0}
\\
\\
& = &
\left[
\frac{X^{k_1-u_1}}
{\prod_{j=2}^{n-1}(X-\eta_j)^{m_j}}
\right]
|_{X=0}
\\
\\
& = &
0\;.
\\
\end{eqnarray*}
Else $k_1\leq u_1$. Since there exists
$m_p\geq1,\;2\leq p\leq n-1$, it follows that
$k_1\leq u_1<u_1+m_2\cdots+m_{n-1}$ then
\begin{eqnarray*}
\sum_{v_2+\cdots+v_{n-1}=k_1-u_1-(m_2+\cdots+m_{n-1})}
\prod_{j=2}^{n-1}
\frac{(v_j+m_j-1)!}{v_j!\,(m_j-1)!}
\,\eta_j^{v_j}
& = &
0
\end{eqnarray*}
and we get, for all $k_1,\,k_2\geq0$,
\begin{eqnarray}\label{eta101}
&  &
-\;
{\bf 1}_{k_1+k_2\geq N_{u_1},\;k_1\leq u_1}
\;\times
\\\nonumber
& & 
\;\;
\times\;
z_2^{k_1+k_2-N_{u_1}}
\sum_{v=2}^{n-1}
\frac{1}{(m_v-1)!}
\frac{\partial^{m_v-1}}
{\partial t^{m_v-1}}|_{t=\eta_v}
\left[
\frac{t^{k_1-u_1-1}}
{\prod_{j=2,j\neq v}^{n-1}
(t-\eta_j)^{m_j}}
\right]
\,.
\\\nonumber
\end{eqnarray}

We finally get by
(\ref{eta1p}), (\ref{eta11}),
and (\ref{eta101})
\begin{eqnarray*}
\lim_{\varepsilon\rightarrow0}
\frac{1}{(2\pi i)^2}
\left[
\int_{r=1-\varepsilon}
-\sum_{l=2}^{\widetilde{n}-1}
\left(
\int_{r=\alpha_{q_l}+\varepsilon}
-\int_{r=\alpha_{q_l}-\varepsilon}
\right)
-\int_{r=\varepsilon}
\right]
\frac{\overline{\zeta}_2
\,\zeta_2^{k_2-m_n}
\zeta_1^{k_1-u_1-1}
\,\omega(\zeta)}
{\prod_{j=2}^{n-1}(\zeta_1-\eta_j\zeta_2)^{m_j}
\left(
1-<\overline{\zeta},z>
\right)}
\;=
\end{eqnarray*}
\begin{eqnarray*}
& = &
0
\;-\;
{\bf 1}_{k_1+k_2\geq N_{u_1}}
\;\sum_{v=2}^{n-1}
\frac{1}{1+|\eta_v|^2}
\,\frac{1}{(m_v-1)!}
\;\times
\\
&  &
\times\;
\frac{\partial^{m_v-1}}
{\partial t^{m_v-1}}|_{t=\eta_v}
\left[
\frac{t^{k_1-u_1-1}}
{\prod_{j=2,j\neq v}^{n-1}
(t-\eta_j)^{m_j}}
\left(
\frac{z_2+|\eta_v|^2z_1/t}
{1+|\eta_v|^2}
\right)^{k_1+k_2-N_{u_1}}
\right]
\\
\\
& + &
{\bf 1}_{k_1+k_2\geq N_{u_1},\,k_1\leq u_1}
z_2^{k_1+k_2-N_{u_1}}
\sum_{v=2}^{n-1}
\frac{1}{(m_v-1)!}
\frac{\partial^{m_v-1}}
{\partial t^{m_v-1}}|_{t=\eta_v}
\left[
\frac{t^{k_1-u_1-1}}
{\prod_{j=2,j\neq v}^{n-1}
(t-\eta_j)^{m_j}}
\right]
\;
\\
\end{eqnarray*}
and the proof is achieved.

\end{proof}

\bigskip

\section{Proof of theorem \ref{theorem}}\label{proof}

\bigskip

Consider
$f\in\mathcal{O}\left(\mathbb{B}_2\right)$
and 
$f(z)=\sum_{k_1,k_2\geq0}
a_{k_1,k_2}
z_1^{k_1}z_2^{k_2}$
its Taylor expansion for all
$z\in\mathbb{B}_2$.

First, we have the following preliminar result.

\begin{lemma}\label{cvtaylor}

The Taylor expansion of $f$ is absolutely convergent
on any compact subset
$K\subset\mathbb{B}_2$.

\end{lemma}


\begin{proof}

$z\in\mathbb{B}_2$ being fixed, consider
the bidisc
$D_z:=D_2(0,r_z)$
where
$r_z=(r_{z,1},r_{z,2})$
is such that
$|z_1|<r_{z,1}$ 
(resp.
$|z_2|<r_{z,2}$)
and
$r_{z,z}^2+r_{z,2}^2<1$
(this is possible since
$|z_1|^2+|z_2|^2<1$).
Then
$z\in D_z\subset\overline{D_z}
\subset\mathbb{B}_2$ 
and, for all
$k_1,\,k_2\geq0$,
the Cauchy formula
on $\overline{D_z}$ gives
\begin{eqnarray*}
a_{k_1,k_2}
=\frac{1}{k_1!\,k_2!}
\frac{\partial^{k_1+k_2}f}
{\partial\zeta_1^{k_1}\partial\zeta_2^{k_2}}
(0,0)
& = &
\frac{1}{(2\pi i)^2}
\int_{|\zeta_1|=r_{z,1},\,|\zeta_2|=r_{z,2}}
\frac{f(\zeta_1,\zeta_2)\,\omega(\zeta)}
{\zeta_1^{k_1+1}\zeta_2^{k_2+1}}\,
\\
\end{eqnarray*}
then 
\begin{eqnarray}\label{coeff}
|a_{k_1,k_2}| 
& \leq &
\dfrac{\sup_{\zeta\in\overline{D_z}}
|f(\zeta)|}
{r_{z,1}^{k_1}\,r_{z,2}^{k_2}}
\,.
\\\nonumber
\end{eqnarray}

If we set
$D'_{z}:=D_2(0,r'_z)$
where
$|z_1|<r'_{z,1}<r_{z,1}$
(resp. 
$|z_2|<r'_{z,2}<r_{z,2}$),
it follows that, for all
$z'\in\overline{D'_z}$,
\begin{eqnarray*}
\sum_{k_1,k_2\geq0}
\left|a_{k_1,k_2}{z'_1}^{k_1}{z'_2}^{k_2}\right|
& \leq &
\sup_{\zeta\in\overline{D_z}}
|f(\zeta)|
\;\sum_{k_1\geq0}
\left(
\frac{|z'_1|}{r_{z,1}}
\right)^{k_1}
\times
\sum_{k_2\geq0}
\left(
\frac{|z'_2|}{r_{z,2}}
\right)^{k_2}
\\
& \leq &
\frac{\sup_{\zeta\in\overline{D_z}}
|f(\zeta)|}
{(1-r'_{z,1}/r_{z,1})(1-r'_{z,2}/r_{z,2})}
\;<\;+\infty
\end{eqnarray*}
and the taylor series is absolutely convergent on
the neighborhood
$\overline{D'_z}$ of $z$.

Now if
$K\subset\mathbb{B}_2$ a compact subset,
it can be covered by a finite number
of such
$D'_{z_j}$ in which the Taylor series
is absolutely convergent.

\end{proof}

\bigskip

Now we can give the proof of theorem \ref{theorem}.

\begin{proof}

First, notice that if
$m_2=\cdots=m_{n-1}=0$, we have (since the
Taylor expansion of $f$ is absolutely convergent)
\begin{eqnarray*}
f(z) 
& = &
\sum_{k_1,k_2\geq0}
a_{k_1,k_2}
z_1^{k_1}z_2^{k_2}
\\\nonumber
& = &
\left[
\sum_{k_1\leq m_1-1,k_2\geq0}
+\sum_{k_1\geq0,k_2\leq m_n-1}
-\sum_{k_1\leq m_1-1,k_2\leq m_n-1}
\right]
a_{k_1,k_2}
z_1^{k_1}z_2^{k_2}
\\
& &
+\;
\sum_{k_1\geq m_1,k_2\geq m_n}
a_{k_1,k_2}
z_1^{k_1}z_2^{k_2}
\\
\\
& = &
\sum_{k_1\leq m_1-1}
\frac{z_1^{k_1}}{k_1!}
\left(
\frac{\partial^{k_1}f}{\partial z_1^{k_1}}
\right)(0,z_2)
+\sum_{k_2\leq m_n-1}
\frac{z_2^{k_2}}{k_2!}
\left(
\frac{\partial^{k_2}f}{\partial z_2^{k_2}}
\right)(z_1,0)
\\
& &
-\sum_{k_1\leq m_1-1,k_2\leq m_n-1}
\frac{z_1^{k_1}}{k_1!}
\frac{z_2^{k_2}}{k_2!}
\left(
\frac{\partial^{k_1+k_2}f}{\partial z_1^{k_1}\partial z_2^{k_2}}
\right)(0)
+\sum_{k_1\geq m_1,k_2\geq m_n}
a_{k_1,k_2}
z_1^{k_1}z_2^{k_2}
\end{eqnarray*}
and theorem \ref{theorem} is proved in this case with
\begin{eqnarray}
\mathcal{G}
\left(
\eta_1^{m_1},\eta_2^0,\ldots,\eta_{n-1}^0,\eta_n^{m_n};f
\right)(z)
\;=\;
\sum_{k_1\leq m_1-1}
\frac{z_1^{k_1}}{k_1!}
\left(
\frac{\partial^{k_1}f}{\partial z_1^{k_1}}
\right)(0,z_2)
\end{eqnarray}
\begin{eqnarray*}
\;\;\;\;
& &
+\sum_{k_2\leq m_n-1}
\frac{z_2^{k_2}}{k_2!}
\left(
\frac{\partial^{k_2}f}{\partial z_2^{k_2}}
\right)(z_1,0)
-\sum_{k_1\leq m_1-1,k_2\leq m_n-1}
\frac{z_1^{k_1}}{k_1!}
\frac{z_2^{k_2}}{k_2!}
\left(
\frac{\partial^{k_1+k_2}f}{\partial z_1^{k_1}\partial z_2^{k_2}}
\right)(0)
\,.
\\
\end{eqnarray*}

\bigskip

Now we can assume that there exists
$m_p\geq1,\;2\leq p\leq n-1$.
An application
of proposition \ref{proptranche} to each
monomial
$z_1^{k_1}z_2^{k_2},\,k_1,\,k_2\geq0$ 
(that is holomorphic on
$\overline{\mathbb{B}_2}$ even if 
$f$ is not) gives
\begin{eqnarray}\label{prelimtheo}
z_1^{k_1}z_2^{k_2} 
& = &
\lim_{\varepsilon\rightarrow0}
\frac{1}{(2\pi i)^2}
\int_{\partial\widetilde{\Sigma}_{\varepsilon}}
\frac{\zeta_1^{k_1}\zeta_2^{k_2}
\det\left(\overline{\zeta},P_n(\zeta,z)\right)}{g_n(\zeta)\left(1-<\overline{\zeta},z>\right)}
\,\omega(\zeta)
\\\nonumber
& &
-\;\lim_{\varepsilon\rightarrow0}\,
\frac{g_n(z)}{(2\pi i)^2}\,
\int_{\widetilde{\Sigma}_{\varepsilon}}
\frac{\zeta_1^{k_1}\zeta_2^{k_2}
\;\omega'\left(\overline{\zeta}\right)\wedge\omega(\zeta)}
{g_n(\zeta)\,\left(1-<\overline{\zeta},z>\right)^2}
\,.
\;\;\;\;\;\;\;\;\;\;\;\;\;\;\;\;\;
\\\nonumber
\end{eqnarray}

By proposition \ref{remainder}
we have, for all
$z\in U_{\eta}$
and all
$k_1,\,k_2\geq0$,
\begin{eqnarray}\label{remaintheo}
-\;\lim_{\varepsilon\rightarrow0}\frac{g_n(z)}{(2\pi i)^2}
\int_{\Sigma_{\varepsilon}}
\frac{\zeta_1^{k_1}\zeta_2^{k_2}\;\omega'\left(\overline{\zeta}\right)\wedge\omega(\zeta)}
{g_n(\zeta)\;\left(1-<\overline{\zeta},z>\right)^2}
& = & 
\;\;\;\;\;\;\;\;\;\;\;\;\;\;\;\;\;\;\;\;\;\;\;\;\;\;\;\;\;\;\;\;\;\;\;\;\;\;\;
\end{eqnarray}
\begin{eqnarray*}
& = &
{\bf 1}_{k_1+k_2\geq N,\,k_1\geq m_1,k_2\geq m_n}
\,z_1^{k_1}\,z_2^{k_2}\\
\\
& - &
{\bf 1}_{k_1+k_2\geq N}\;
\sum_{p=2}^{n-1}
z_1^{m_1}\prod_{j=2,j\neq p}^{n-1}(z_1-\eta_jz_2)^{m_j}z_2^{m_n}
\sum_{s=0}^{m_p-1}z_2^{m_p-1-s}(z_1-\eta_pz_2)^s
\\
& & 
\;\times\,\frac{1}{s!}\frac{\partial^s}{\partial t^s}|_{t=\eta_p}
\left[\frac{t^{k_1}}{t^{m_1}\prod_{j=2,j\neq p}^{n-1}(t-\eta_j)^{m_j}}
\left(\frac{z_2+|\eta_p|^2z_1/t}{1+|\eta_p|^2}\right)^{k_1+k_2-N+1}\right]\\
\\
& + & 
{\bf 1}_{k_1\leq m_1-1,k_2\geq N-k_1}
\sum_{p=2}^{n-1}z_1^{m_1}\prod_{j=2,j\neq p}^{n-1}(z_1-\eta_jz_2)^{m_j}z_2^{m_n}
\sum_{s=0}^{m_p-1}z_2^{m_p-1-s}(z_1-\eta_pz_2)^s\\
& & 
\;\times\,
\frac{1}{s!}\frac{\partial^s}{\partial t^s}|_{t=\eta_p}
\left[
\frac{t^{k_1}z_2^{k_1+k_2-N+1}}
{t^{m_1}\prod_{j=2,j\neq p}^{n-1}(t-\eta_j)^{m_j}}
\right]\\
\\
& + & 
{\bf 1}_{k_2\leq m_n-1,k_1\geq N-k_2}
\sum_{p=2}^{n-1}z_1^{m_1}\prod_{j=2,j\neq p}^{n-1}(z_1-\eta_jz_2)^{m_j}z_2^{m_n}
\sum_{s=0}^{m_p-1}z_2^{m_p-1-s}(z_1-\eta_pz_2)^s
\\
& & 
\;\times\,
\frac{1}{s!}\frac{\partial^s}{\partial t^s}|_{t=\eta_p}
\left[
\frac{t^{N-1-k_2}z_1^{k_1+k_2-N+1}}
{t^{m_1}\prod_{j=2,j\neq p}^{n-1}(t-\eta_j)^{m_j}}
\right]\,.
\\
\end{eqnarray*}

Similarly, we have
by proposition \ref{interpoler},
for all
$z\in U_{\eta}$ and all
$k_1,\,k_2\geq0$,
\begin{eqnarray}\label{intertheo1}
\lim_{\varepsilon\rightarrow0}
\frac{1}{(2\pi i)^2}
\int_{\partial\widetilde{\Sigma}_{\varepsilon}}
\frac{\zeta_1^{k_1}\zeta_2^{k_2}\det\left(\overline{\zeta},P_n(\zeta,z)\right)}
{g_n(\zeta)\left(1-<\overline{\zeta},z>\right)}
\,\omega(\zeta)
& = &
\;\;\;\;\;\;\;\;\;\;\;\;\;\;\;\;\;\;\;\;\;\;\;\;\;\;\;\;
\end{eqnarray}
\begin{eqnarray*}
& = &
\sum_{u_1=0}^{m_1-1}z_1^{u_1}
\prod_{j=2}^{n-1}(z_1-\eta_jz_2)^{m_j}z_2^{m_n}
\sum_{p=2}^{n-1}
\frac{1}{(m_p-1)!}
\frac{\partial^{m_p-1}}{\partial t^{m_p-1}}|_{t=\eta_p}
\;\frac{1}{t^{u_1+1}\prod_{j=2,j\neq p}^{n-1}(t-\eta_j)^{m_j}}
\times
\\
& &
\times\;\{\;
{\bf 1}_{k_1+k_2\geq N_{u_1}}
\frac{1}{1+|\eta_p|^2}
t^{k_1}
\left(
\frac{z_2+|\eta_p|^2z_1/t}{1+|\eta_p|^2}
\right)^{k_1+k_2-N_{u_1}}
\\
& &
\;\;-\;
{\bf 1}_{k_1\leq u_1,\,k_2\geq N_{u_1}-k_1}
t^{k_1}
z_2^{k_1+k_2-N_{u_1}}
\;\}
\\
\end{eqnarray*}
\begin{eqnarray*}
& + & 
\sum_{p=2}^{n-1}\sum_{u_p=0}^{m_p-1}
(z_1-\eta_pz_2)^{u_p}
\prod_{j=p+1}^{n-1}(z_1-\eta_jz_2)^{m_j}z_2^{m_n}
\;\times
\\
& \times &
\{
{\bf 1}_{k_1+k_2\geq N_{u_p}}
\{\,
\frac{1}{u_p!}
\frac{\partial^{u_p}}{\partial t^{u_p}}|_{t=\eta_p}
\,\frac{1+|\eta_p|^2\eta_p/t}{1+|\eta_p|^2}
\frac{\,t^{k_1}
\left(
\frac{z_2+|\eta_p|^2z_1/t}{1+|\eta_p|^2}
\right)^{k_1+k_2-N_{u_p}}}
{\prod_{j=p+1}^{n-1}(t-\eta_j)^{m_j}}
\\
& &
+
\sum_{q=p+1}^{n-1}
\frac{1}{(m_q-1)!}
\frac{\partial^{m_q-1}}{\partial t^{m_q-1}}|_{t=\eta_q}
\,\frac{1+|\eta_q|^2\eta_p/t}{1+|\eta_q|^2}
\frac{\,t^{k_1}
\left(
\frac{z_2+|\eta_q|^2z_1/t}{1+|\eta_q|^2}
\right)^{k_1+k_2-N_{u_p}}}
{(t-\eta_p)^{u_p+1}\prod_{j=p+1,j\neq q}^{n-1}(t-\eta_j)^{m_j}}
\;\}
\\
\\
& & 
-\,
{\bf 1}_{k_2\leq m_n-1,k_1\geq N_{u_p}-k_2}
\,\eta_p
z_1^{k_1+k_2-N_{u_p}}
\;\times
\{\;
\frac{1}{u_p!}
\frac{\partial^{u_p}}{\partial t^{u_p}}|_{t=\eta_p}
\,
\frac{t^{N_{u_p}-1-k_2}}
{\prod_{j=p+1}^{n-1}(t-\eta_j)^{m_j}}
\\
& &
\;\;+
\sum_{q=p+1}^{n-1}
\frac{1}{(m_q-1)!}
\frac{\partial^{m_q-1}}{\partial t^{m_q-1}}|_{t=\eta_q}
\,\frac{t^{N_{u_p}-1-k_2}}
{(t-\eta_p)^{u_p+1}\prod_{j=p+1,j\neq q}^{n-1}(t-\eta_j)^{m_j}}
\;\}
\\
\end{eqnarray*}
\begin{eqnarray*}
& + &
{\bf 1}_{k_1\geq0,k_2\leq m_n-1}
\,z_1^{k_1}z_2^{k_2}\;.
\;\;\;\;\;\;\;\;\;\;\;\;\;\;\;\;\;\;\;\;\;\;\;\;\;\;\;\;\;\;\;\;\;\;\;\;\;\;\;\;\;\;\;\;\;\;\;\;\;\;\;\;\;\;
\;\;\;\;\;\;\;\;\;\;\;\;\;\;\;\;\;\;\;\;\;\;\;\;
\\
\end{eqnarray*}

It follows from (\ref{prelimtheo}), 
(\ref{remaintheo}), 
(\ref{intertheo1})
and by continuity with repect to
$z$ that we have, for all
$z\in\mathbb{B}_2$,
\begin{eqnarray*}
f(z)
\;=\;
\sum_{k_1,k_2\geq0}
a_{k_1,k_2}
z_1^{k_1}z_2^{k_2}
& = &
\;\;\;\;\;\;\;\;\;\;\;\;\;\;\;\;\;\;\;\;\;\;\;\;\;\;\;\;\;\;\;\;\;\;\;\;\;\;\;\;\;\;\;\;\;\;\;\;\;\;\;\;\;\;\;\;
\;\;\;\;\;\;\;\;\;\;\;\;\;\;\;\;\;\;
\end{eqnarray*}
\begin{eqnarray*}
& = &
\sum_{u_1=0}^{m_1-1}z_1^{u_1}
\prod_{j=2}^{n-1}(z_1-\eta_jz_2)^{m_j}z_2^{m_n}
\sum_{p=2}^{n-1}
\frac{1}{(m_p-1)!}
\frac{\partial^{m_p-1}}{\partial t^{m_p-1}}|_{t=\eta_p}
\;\frac{1}{t^{u_1+1}\prod_{j=2,j\neq p}^{n-1}(t-\eta_j)^{m_j}}
\times
\\
& &
\times\;
\{\;
\frac{1}{1+|\eta_p|^2}
\sum_{k_1+k_2\geq N_{u_1}}
a_{k_1,k_2}t^{k_1}
\left(
\frac{z_2+|\eta_p|^2z_1/t}{1+|\eta_p|^2}
\right)^{k_1+k_2-N_{u_1}}
\\
& &
\;\;-\;
\sum_{k_1\leq u_1,\,k_2\geq N_{u_1}-k_1}
a_{k_1,k_2}t^{k_1}
z_2^{k_1+k_2-N_{u_1}}
\;\}
\\
\end{eqnarray*}
\begin{eqnarray*}
& + & 
\sum_{p=2}^{n-1}\sum_{u_p=0}^{m_p-1}
(z_1-\eta_pz_2)^{u_p}
\prod_{j=p+1}^{n-1}(z_1-\eta_jz_2)^{m_j}z_2^{m_n}
\;\times
\\
& \times\{ &
\frac{1}{u_p!}
\frac{\partial^{u_p}}{\partial t^{u_p}}|_{t=\eta_p}
\;\{
\;
\frac{\frac{1+|\eta_p|^2\eta_p/t}{1+|\eta_p|^2}}
{\prod_{j=p+1}^{n-1}(t-\eta_j)^{m_j}}
\sum_{k_1+k_2\geq N_{u_p}}
a_{k_1,k_2}t^{k_1}
\left(
\frac{z_2+|\eta_p|^2z_1/t}{1+|\eta_p|^2}
\right)^{k_1+k_2-N_{u_p}}
\;\}
\\
& &
+
\sum_{q=p+1}^{n-1}
\frac{1}{(m_q-1)!}
\frac{\partial^{m_q-1}}{\partial t^{m_q-1}}|_{t=\eta_q}
\;\{
\;
\frac{\frac{1+|\eta_q|^2\eta_p/t}{1+|\eta_q|^2}}
{(t-\eta_p)^{u_p+1}\prod_{j=p+1,j\neq q}^{n-1}(t-\eta_j)^{m_j}}
\;\times
\\
& & 
\;\;\times
\sum_{k_1+k_2\geq N_{u_p}}
a_{k_1,k_2}t^{k_1}
\left(
\frac{z_2+|\eta_q|^2z_1/t}{1+|\eta_q|^2}
\right)^{k_1+k_2-N_{u_p}}
\;\}
\end{eqnarray*}
\begin{eqnarray*}
& - & 
\{\;
\frac{1}{u_p!}
\frac{\partial^{u_p}}{\partial t^{u_p}}|_{t=\eta_p}
\,\frac{\eta_p
\sum_{k_2\leq m_n-1,k_1\geq N_{u_p}-k_2}
a_{k_1,k_2}t^{N_{u_p}-1-k_2}
z_1^{k_1+k_2-N_{u_p}}}
{\prod_{j=p+1}^{n-1}(t-\eta_j)^{m_j}}
\\
& &
+
\sum_{q=p+1}^{n-1}
\frac{1}{(m_q-1)!}
\frac{\partial^{m_q-1}}{\partial t^{m_q-1}}|_{t=\eta_q}
\,\frac{1}{(t-\eta_p)^{u_p+1}\prod_{j=p+1,j\neq q}^{n-1}(t-\eta_j)^{m_j}}
\;\times
\\
& &
\;\;\;\;\times\;
\eta_p
\sum_{k_2\leq m_n-1,k_1\geq N_{u_p}-k_2}
a_{k_1,k_2}t^{N_{u_p}-1-k_2}
z_1^{k_1+k_2-N_{u_p}}
\;\;\}\;\;\}
\\
\end{eqnarray*}
\begin{eqnarray*}
& + &
\sum_{k_1\geq0,u_n\leq m_n-1}
a_{k_1,u_n}z_1^{k_1}z_2^{u_n}
\;+\;
\sum_{k_1+k_2\geq N,\,k_1\geq m_1,k_2\geq m_n}
a_{k_1,k_2}\,z_1^{k_1}\,z_2^{k_2}
\;\;\;\;\;\;\;\;\;\;\;
\\
\end{eqnarray*}
\begin{eqnarray*}
& - & 
\sum_{p=2}^{n-1}
z_1^{m_1}\prod_{j=2,j\neq p}^{n-1}(z_1-\eta_jz_2)^{m_j}z_2^{m_n}
\sum_{s=0}^{m_p-1}z_2^{m_p-1-s}(z_1-\eta_pz_2)^s
\\
& & 
\times\,\frac{1}{s!}\frac{\partial^s}{\partial t^s}|_{t=\eta_p}
\left[\frac{1}{t^{m_1}\prod_{j=2,j\neq p}^{n-1}(t-\eta_j)^{m_j}}
\sum_{k_1+k_2\geq N}a_{k_1,k_2}\,t^{k_1}
\left(\frac{z_2+|\eta_p|^2z_1/t}{1+|\eta_p|^2}\right)^{k_1+k_2-N+1}\right]\\
\end{eqnarray*}
\begin{eqnarray*}
& + & 
\sum_{p=2}^{n-1}z_1^{m_1}\prod_{j=2,j\neq p}^{n-1}(z_1-\eta_jz_2)^{m_j}z_2^{m_n}
\sum_{s=0}^{m_p-1}z_2^{m_p-1-s}(z_1-\eta_pz_2)^s\\
& & 
\times\,
\frac{1}{s!}\frac{\partial^s}{\partial t^s}|_{t=\eta_p}
\left[\frac{1}{t^{m_1}\prod_{j=2,j\neq p}^{n-1}(t-\eta_j)^{m_j}}
\sum_{k_1\leq m_1-1,k_2\geq N-k_1}a_{k_1,k_2}\,t^{k_1}z_2^{k_1+k_2-N+1}\right]\\
\end{eqnarray*}
\begin{eqnarray*}
& + & 
\sum_{p=2}^{n-1}z_1^{m_1}\prod_{j=2,j\neq p}^{n-1}(z_1-\eta_jz_2)^{m_j}z_2^{m_n}
\sum_{s=0}^{m_p-1}z_2^{m_p-1-s}(z_1-\eta_pz_2)^s\\
& & 
\times\,
\frac{1}{s!}\frac{\partial^s}{\partial t^s}|_{t=\eta_p}
\left[\frac{1}{t^{m_1}\prod_{j=2,j\neq p}^{n-1}(t-\eta_j)^{m_j}}
\sum_{k_2\leq m_n-1,k_1\geq N-k_2}a_{k_1,k_2}\,t^{N-1-k_2}z_1^{k_1+k_2-N+1}\right]\,.
\\
\end{eqnarray*}

We have used the following lemma that allows us to
switch the above series and derivative with respect to $t$.

\begin{lemma}\label{serie}

Let $K\subset\mathbb{B}_2$
be a compact subset,
$q\geq l\geq0$
and 
$p=2,\ldots,n-1$.
Then for all 
$z\in K$
and all
$(t,w)$ in a neighborhood of
$(\eta_p,\eta_p)$,
the following series
\begin{eqnarray*}
\sum_{k_1+k_2\geq q}
a_{k_1,k_2}t^{k_1}
\left(\frac{z_2+|w|^2z_1/t}{1+|w|^2}\right)^{k_1+k_2-q}
\;,\;
\;\;\;\;\;\;\;\;\;\;\;\;\;\;\;\;\;\;\;\;\;\;
\\
\\
\sum_{k_1\leq l,k_2\geq q-k_1}
a_{k_1,k_2}
t^{k_1}\;z_2^{k_1+k_2-q}
\;,\;
\sum_{k_2\leq l,k_1\geq q-k_1}
a_{k_1,k_2}
t^{l-k_2}\;z_1^{k_1+k_2-q}
\\
\end{eqnarray*}
are absolutely convergent. 
In particular, all their derivatives with respect to $t$ 
are absolutely convergent as series of
holomorphic functions.

\end{lemma}

\medskip

\begin{proof}

Consider the first series. One has, for all
$z\in K$,
\begin{eqnarray*}
\left|
\frac{z_2+|w|^2z_1/t}{1+|w|^2}
\right|
\;\leq\;
\|z\|\,
\frac{\sqrt{1+|w|^4/|t|^2}}{1+|w|^2}
\;\leq\;
(1-\varepsilon_K)\,
\frac{\sqrt{1+|w|^4/|t|^2}}{1+|w|^2}
\,.
\\
\end{eqnarray*}
Since
$\|(t,1)\|=\sqrt{1+|t|^2}$,
one can choose a neighborhood
$W(\eta_p,\eta_p)$ such that, for all
$(t,w)\in W(\eta_p,\eta_p)$,
\begin{eqnarray*}
\frac{1-\varepsilon_K}{\sqrt{1+|\eta_p|^2}}
\;\leq\;
\frac{\sqrt{1+|w|^4/|t|^2}}{1+|w|^2}
\;\leq\;
\frac{\sqrt{1+|w|^4/|t|^2}}{1+|w|^2}
\sqrt{1+|t|^2}
\;\leq\;1+\varepsilon_K
\,.
\\
\end{eqnarray*}
In particular,
\begin{eqnarray*}
\left(
t\,
\frac{(1-\varepsilon_K)\sqrt{1+|w|^4/|t|^2}}{1+|w|^2},
\,\frac{(1-\varepsilon_K)\sqrt{1+|w|^4/|t|^2}}{1+|w|^2}
\right)
\;\in\;
\overline{\mathbb{B}_2}
\left(0,(1-\varepsilon_K^2)\right)
\,.
\\
\end{eqnarray*}
It follows that 
\begin{eqnarray*}
\sum_{k_1+k_2\geq q}
\;\sup_{z\in K,(t,w)\in W(\eta_p,\eta_p)}
\left|
a_{k_1,k_2}\,t^{k_1}
\left(
\frac{z_2+|w|^2z_1/t}{1+|w|^2}
\right)^{k_1+k_2-q}
\right|
& \leq &
\;\;\;\;\;\;\;\;\;\;\;\;\;\;\;\;\;\;\;\;\;
\end{eqnarray*}
\begin{eqnarray*}
& \leq &
\sum_{k_1+k_2\geq q}
\;\sup_{(t,w)\in W(\eta_p,\eta_p)}
\left|
a_{k_1,k_2}\,t^{k_1}
\left(
(1-\varepsilon_K)\,
\frac{\sqrt{1+|w|^4/|t|^2}}{1+|w|^2}
\right)^{k_1+k_2-q}
\right|
\\
& \leq &
\left(
\frac{\sqrt{1+|\eta_p|^2}}{(1-\varepsilon_K)^2}
\right)^q
\sum_{k_1+k_2\geq q}
\;\sup_{(t,w)\in W(\eta_p,\eta_p)}
\left|
a_{k_1,k_2}\,t^{k_1}
\left(
(1-\varepsilon_K)\,
\frac{\sqrt{1+|w|^4/|t|^2}}{1+|w|^2}
\right)^{k_1+k_2}
\right|
\\
& \leq &
\left(
\frac{\sqrt{1+|\eta_p|^2}}{(1-\varepsilon_K)^2}
\right)^q
\sum_{k_1+k_2\geq q}
\;\sup_{\zeta\in\overline{\mathbb{B}_2}(0,1-\varepsilon_K^2)}
\left|
a_{k_1,k_2}\zeta_1^{k_1}\zeta_2^{k_2}
\right|
\;\;<\;\;+\infty
\end{eqnarray*}
( $\overline{\mathbb{B}_2}\left(0,1-\varepsilon_K^2\right)$ 
can be covered by a finite number of polydiscs
in which the series is absolutely convergent,
see lemma \ref{cvtaylor}).
\bigskip

The second series
is a polynomial function with respect to
$t$. On the other hand, 
notice that
$|z_2|\leq\|z\|\leq1-\varepsilon_K$
and
$\overline{D_2}(0,(\varepsilon_K,1-\varepsilon_K))
\subset\mathbb{B}_2$. It follows that, for all
$k_1=0,\ldots,l$,
\begin{eqnarray*}
\sum_{k_2\geq q-k_1}
\sup_{z\in K}
\left|a_{k_1,k_2}z_2^{k_1+k_2-q}\right|
& \leq &
\;\;\;\;\;\;\;\;\;\;\;\;\;\;\;\;\;\;\;\;\;\;\;\;\;\;\;\;\;\;\;\;\;\;\;\;\;\;\;\;\;\;\;\;\;\;\;\;\;\;\;\;\;\;\;
\;\;\;\;\;\;\;
\end{eqnarray*}
\begin{eqnarray*}
\;\;\;\;\;\;\;
& \leq &
\frac{1}{\varepsilon_K^{k_1}(1-\varepsilon_K)^{q-k_1}}
\sum_{k_2\geq q-k_1}
\left|
a_{k_1,k_2}
\varepsilon_K^{k_1}(1-\varepsilon_K)^{k_2}
\right|
\\
& \leq &
\frac{1}{\varepsilon_K^{k_1}(1-\varepsilon_K)^{q-k_1}}
\sum_{u_1+u_2\geq q}
\sup_{\zeta\in\overline{D_2}(0,(\varepsilon_K,1-\varepsilon_K))}
\left|
a_{u_1,u_2}
\zeta_1^{u_1}\zeta_2^{u_2}
\right|
\\
& < &
+\infty\,,
\\
\end{eqnarray*}
as well as
$\sum_{k_1=0}^l
\sum_{k_2\geq q-k_1}
\sup_{z\in K,\,t\in W(\eta_p)}
\left|
t^{k_1}
a_{k_1,k_2}z_2^{k_1+k_2-q}
\right|$.
\bigskip

The proof for the last series is similar.

\end{proof}

\bigskip

It follows that
\begin{eqnarray*}
f(z)
& = &
\sum_{u_1=0}^{m_1-1}z_1^{u_1}
z_2^{m_2+\cdots+m_n}
\prod_{j=2}^{n-1}(z_1/z_2-\eta_j)^{m_j}
\sum_{p=2}^{n-1}
\frac{1}{(m_p-1)!}
\frac{\partial^{m_p-1}}{\partial t^{m_p-1}}|_{t=\eta_p}
\;\frac{1}{\prod_{j=2,j\neq p}^{n-1}(t-\eta_j)^{m_j}}
\times
\\
& \times &
\{\;
\frac{1}{1+|\eta_p|^2}
\sum_{k_1+k_2\geq N_{u_1}}
a_{k_1,k_2}t^{k_1-u_1-1}
\left(
\frac{z_2+|\eta_p|^2z_1/t}{1+|\eta_p|^2}
\right)^{k_1+k_2-N_{u_1}}
\\
& &
\;\;-\;
\sum_{k_1\leq u_1,\,k_2\geq N_{u_1}-k_1}
a_{k_1,k_2}t^{k_1-u_1-1}
z_2^{k_1+k_2-N_{u_1}}
\;\}
\\
\end{eqnarray*}
\begin{eqnarray*}
& + & 
\sum_{p=2}^{n-1}\sum_{u_p=0}^{m_p-1}
z_2^{N_{u_p}}
(z_1/z_2-\eta_p)^{u_p+1}
\prod_{j=p+1}^{n-1}(z_1/z_2-\eta_j)^{m_j}
\;\frac{1}{z_1/z_2-\eta_p}
\;\times
\\
& \times\{ &
\frac{1}{u_p!}
\frac{\partial^{u_p}}{\partial t^{u_p}}|_{t=\eta_p}
\,
\frac{\frac{1+|\eta_p|^2\eta_p/t}{1+|\eta_p|^2}}
{\prod_{j=p+1}^{n-1}(t-\eta_j)^{m_j}}
\sum_{k_1+k_2\geq N_{u_p}}
a_{k_1,k_2}t^{k_1}
\left(
\frac{z_2+|\eta_p|^2z_1/t}{1+|\eta_p|^2}
\right)^{k_1+k_2-N_{u_p}}
\\
& &
+
\sum_{q=p+1}^{n-1}
\frac{1}{(m_q-1)!}
\frac{\partial^{m_q-1}}{\partial t^{m_q-1}}|_{t=\eta_q}
\;
\frac{\frac{1+|\eta_q|^2\eta_p/t}{1+|\eta_p|^2}}
{(t-\eta_p)^{u_p+1}\prod_{j=p+1,j\neq q}^{n-1}(t-\eta_j)^{m_j}}
\;\times
\\
& & 
\;\;\times
\sum_{k_1+k_2\geq N_{u_p}}
a_{k_1,k_2}t^{k_1}
\left(
\frac{z_2+|\eta_q|^2z_1/t}{1+|\eta_q|^2}
\right)^{k_1+k_2-N_{u_p}}
\end{eqnarray*}
\begin{eqnarray*}
& & 
-\;\{\;
\frac{1}{u_p!}
\frac{\partial^{u_p}}{\partial t^{u_p}}|_{t=\eta_p}
\{
\frac{\eta_p
\sum_{k_2\leq m_n-1,k_1\geq N_{u_p}-k_2}
a_{k_1,k_2}t^{N_{u_p}-1-k_2}
z_1^{k_1+k_2-N_{u_p}}}
{\prod_{j=p+1}^{n-1}(t-\eta_j)^{m_j}}
\\
& &
+
\sum_{q=p+1}^{n-1}
\frac{1}{(m_q-1)!}
\frac{\partial^{m_q-1}}{\partial t^{m_q-1}}|_{t=\eta_q}
\,
\frac{\eta_p
\sum_{k_2\leq m_n-1,k_1\geq N_{u_p}-k_2}
a_{k_1,k_2}t^{N_{u_p}-1-k_2}
z_1^{k_1+k_2-N_{u_p}}}
{(t-\eta_p)^{u_p+1}\prod_{j=p+1,j\neq q}^{n-1}(t-\eta_j)^{m_j}}
\;\}\;\;\}
\\
\end{eqnarray*}
\begin{eqnarray*}
& + &
\sum_{u_n=0}^{m_n-1}
z_2^{u_n}
\sum_{k_1\geq0}
a_{k_1,u_n}z_1^{k_1}
\;+\;
\sum_{k_1+k_2\geq N,\,k_1\geq m_1,k_2\geq m_n}
a_{k_1,k_2}\,z_1^{k_1}\,z_2^{k_2}
\;\;\;\;\;\;\;\;\;\;
\\
\end{eqnarray*}
\begin{eqnarray*}
& - & 
\sum_{p=2}^{n-1}
\prod_{j=2,j\neq p}^{n-1}(z_1/z_2-\eta_j)^{m_j}
\sum_{s=0}^{m_p-1}(z_1/z_2-\eta_p)^s
\\
& \times &
\frac{1}{s!}\frac{\partial^s}{\partial t^s}|_{t=\eta_p}
\left[\frac{1}{\prod_{j=2,j\neq p}^{n-1}(t-\eta_j)^{m_j}}
(z_1/z_2)^{m_1}z_2^{N-1}
\sum_{k_1+k_2\geq N}
a_{k_1,k_2}\,t^{k_1-m_1}
\left(\frac{z_2+|\eta_p|^2z_1/t}{1+|\eta_p|^2}\right)^{k_1+k_2-N+1}\right]
\\
\end{eqnarray*}
\begin{eqnarray*}
& + & 
\sum_{p=2}^{n-1}
\prod_{j=2,j\neq p}^{n-1}(z_1/z_2-\eta_j)^{m_j}
\sum_{s=0}^{m_p-1}(z_1/z_2-\eta_p)^s
\\
& \times &
\frac{1}{s!}\frac{\partial^s}{\partial t^s}|_{t=\eta_p}
\left[\frac{1}{\prod_{j=2,j\neq p}^{n-1}(t-\eta_j)^{m_j}}
(z_1/z_2)^{m_1}
\sum_{k_1\leq m_1-1,k_2\geq N-k_1}
a_{k_1,k_2}\,t^{k_1-m_1}z_2^{k_1+k_2}\right]
\\
\end{eqnarray*}
\begin{eqnarray*}
& + & 
(z_1/z_2)^{m_1}z_2^{N-1}
\sum_{p=2}^{n-1}
\prod_{j=2,j\neq p}^{n-1}(z_1/z_2-\eta_j)^{m_j}
\sum_{s=0}^{m_p-1}(z_1/z_2-\eta_p)^s
\\
& \times &
\frac{1}{s!}\frac{\partial^s}{\partial t^s}|_{t=\eta_p}
\left[\frac{1}{\prod_{j=2,j\neq p}^{n-1}(t-\eta_j)^{m_j}}
\sum_{k_2\leq m_n-1,k_1\geq N-k_2}
a_{k_1,k_2}\,t^{N-m_1-1-k_2}z_1^{k_1+k_2-N+1}\right]
\\
\end{eqnarray*}
\begin{eqnarray*}
& = &
\sum_{u_1=0}^{m_1-1}
(z_1/z_2)^{u_1}z_2^{N_{u_1}}
\;\times
\\
& \times\;\{ &
\mathcal{L}
\left(
\eta_2^{m_2},\ldots,\eta_{n-1}^{m_{n-1}};
\sum_{k_1+k_2\geq N_{u_1}}
\frac{a_{k_1,k_2}t^{k_1-u_1-1}}{1+|w|^2}
\left(
\frac{z_2+|w|^2z_1/t}{1+|w|^2}
\right)^{k_1+k_2-N_{u_1}}
\right)(z_1/z_2)
\end{eqnarray*}
\begin{eqnarray*}
& &
-\;
\mathcal{L}
\left(
\eta_2^{m_2},\ldots,\eta_{n-1}^{m_{n-1}};
\sum_{k_1\leq u_1,\,k_2\geq N_{u_1}-k_1}
a_{k_1,k_2}t^{k_1-u_1-1}
z_2^{k_1+k_2-N_{u_1}}
\right)(z_1/z_2)
\;\;\}
\end{eqnarray*}
\begin{eqnarray*}
+\;
\sum_{p=2}^{n-1}\sum_{u_p=0}^{m_p-1}
z_2^{N_{u_p}}
\;\times
\;\;\;\;\;\;\;\;\;\;\;\;\;\;\;\;\;\;\;\;\;\;\;\;\;\;\;\;\;\;\;\;\;\;\;\;\;\;\;\;\;\;\;\;\;\;\;\;\;\;\;\;\;\;\;\;\;\;\;\;
\;\;\;\;\;\;\;\;\;\;\;\;\;\;\;\;\;\;\;\;\;\;\;\;\;\;
\end{eqnarray*}
\begin{eqnarray*}
\times\;\{\;
\mathcal{L}
\left(
\eta_p^{u_p+1},\ldots,\eta_{n-1}^{m_{n-1}};
\frac{1+|w|^2\eta_p/t}{1+|w|^2}
\sum_{k_1+k_2\geq N_{u_p}}
\frac{a_{k_1,k_2}t^{k_1}}{z_1/z_2-\eta_p}
\left(
\frac{z_2+|w|^2z_1/t}{1+|w|^2}
\right)^{k_1+k_2-N_{u_p}}
\right)(z_1/z_2)
\end{eqnarray*}
\begin{eqnarray*}
& - &
\mathcal{L}
\left(
\eta_p^{u_p+1},\ldots,\eta_{n-1}^{m_{n-1}};
\sum_{k_2\leq m_n-1,k_1\geq N_{u_p}-k_2}
\frac{\eta_pa_{k_1,k_2}t^{N_{u_p}-1-k_2}
z_1^{k_1+k_2-N_{u_p}}}
{z_1/z_2-\eta_p}
\right)(z_1/z_2)
\;\}
\\
\end{eqnarray*}
\begin{eqnarray*}
& + &
\mathcal{L}
\left(
0^{m_n};
\frac{f(z_1,t)}{z_2-t}
\right)(z_2)
\;+\;
\sum_{k_1+k_2\geq N,\,k_1\geq m_1,k_2\geq m_n}
a_{k_1,k_2}\,z_1^{k_1}\,z_2^{k_2}
\\
& - &
\mathcal{L}
\left(
\eta_2^{m_2},\ldots,\eta_{n-1}^{m_{n-1}};
\frac{z_2^{N-1}(z_1/z_2)^{m_1}}{z_1/z_2-t}
\sum_{k_1+k_2\geq N}
a_{k_1,k_2}\,t^{k_1-m_1}
\left(\frac{z_2+|w|^2z_1/t}{1+|w|^2}\right)^{k_1+k_2-N+1}
\right)
(z_1/z_2)
\\
& + &
\mathcal{L}
\left(
\eta_2^{m_2},\ldots,\eta_{n-1}^{m_{n-1}};
\frac{(z_1/z_2)^{m_1}}{z_1/z_2-t}
\sum_{k_1\leq m_1-1,k_2\geq N-k_1}
a_{k_1,k_2}\,t^{k_1-m_1}z_2^{k_1+k_2}
\right)
(z_1/z_2)
\\
& + &
\mathcal{L}
\left(
\eta_2^{m_2},\ldots,\eta_{n-1}^{m_{n-1}};
\frac{(z_1/z_2)^{m_1}z_2^{N-1}}{z_1/z_2-t}
\sum_{k_2\leq m_n-1,k_1\geq N-k_2}
a_{k_1,k_2}\,t^{N-m_1-1-k_2}z_1^{k_1+k_2-N+1}
\right)
(z_1/z_2)
\,.
\\
\end{eqnarray*}
The proof of the theorem will be complete with the following result.

\end{proof}

\bigskip

\begin{lemma}\label{explicit}

Assume that, for all
$p=1,\ldots,n-1$ (resp. $p=n$) and
$u_p=0,\ldots,m_p-1$ (resp.
$u_n=0,\ldots,m_n-1$),
we know
$\left(
\dfrac{\partial^{u_p}f}{\partial z_1^{u_p}}
\right)|_{\{z_1=\eta_pz_2\}}$
(resp.
$\left(
\dfrac{\partial^{u_n}f}{\partial z_2^{u_n}}
\right)|_{\{z_2=0\}}$ ).
\medskip

Then, for all $z\in\mathbb{B}_2$,
\begin{eqnarray}
f(z)
& - &
\sum_{k_1+k_2\geq N,\,k_1\geq m_1,\,k_2\geq m_n}
a_{k_1,\,k_2}
\,z_1^{k_1}\,z_2^{k_2}
\\\nonumber
\end{eqnarray}
can be known as an explicit formula constructed
from this data.

\end{lemma}

\bigskip

\begin{proof}

First, we have seen that
\begin{eqnarray*}
\mathcal{L}
\left(
0^{m_n};
\frac{f(z_1,t)}{z_2-t}
\right)(z_2)
& = &
\sum_{u_n=0}^{m_n-1}
\frac{z_2^{u_n}}{u_n!}
\left(
\frac{\partial^{u_n}f}{\partial z_2^{u_n}}
\right)(z_1,0)
\,.
\\
\end{eqnarray*}

\bigskip

Next, we have
\begin{eqnarray*}
f(tv,v)
\;=\;
\sum_{k_1,k_2\geq0}
a_{k_1,k_2}
t^{k_1}v^{k_1+k_2}
\;=\;
\sum_{l\geq0}v^l
\sum_{k_1+k_2=l}a_{k_1,k_2}t^{k_1}
\\
\end{eqnarray*}
that is also as an analytic function with respect to $v$
\begin{eqnarray*}
f(tv,v)
& = &
\sum_{l\geq0}
\frac{v^l}{l!}
\frac{\partial^l}{\partial v^l}|_{v=0}
\left[
f(tv,v)
\right]
\,.
\end{eqnarray*}
By the uniqueness of the coefficients we get,
for all $v\geq0$,
\begin{eqnarray*}
\sum_{k_1+k_2=l}a_{k_1,k_2}t^{k_1}
& = &
\frac{1}{l!}
\frac{\partial^l}{\partial v^l}|_{v=0}
\left[
f(tv,v)
\right]\,,
\\\nonumber
\end{eqnarray*}
then
\begin{eqnarray*}
\sum_{k_1+k_2\geq N_{u_p}}
a_{k_1,k_2}\,t^{k_1}
\left(\frac{z_2+|w|^2z_1/t}{1+|w|^2}\right)^{k_1+k_2-N_{u_p}}
& = &
\;\;\;\;\;\;\;\;\;\;\;\;\;\;\;\;\;\;\;\;\;\;\;\;\;\;\;\;\;\;\;\;\;\;\;\;\;\;
\end{eqnarray*}
\begin{eqnarray*}
& = &
\sum_{l\geq N_{u_p}}
\left(\frac{z_2+|w|^2z_1/t}{1+|w|^2}\right)^{l-N_{u_p}}
\sum_{k_1+k_2=l}
a_{k_1,k_2}\,t^{k_1}
\\
& = &
\sum_{l\geq N_{u_p}}
\left(\frac{z_2+|w|^2z_1/t}{1+|w|^2}\right)^{l-N_{u_p}}
\frac{1}{l!}
\frac{\partial^l}{\partial v^l}|_{v=0}
\left[
f(tv,v)
\right]\,.
\end{eqnarray*}
Now notice that, for all
$q=2,\ldots,n-1$, for all
$s=0,\ldots,m_q-1$ and all $l\geq0$,
the following derivative is known:
\begin{eqnarray*}
\frac{1}{s!}
\frac{\partial^s}{\partial t^s}|_{t=\eta_q}
\left[
\frac{1}{l!}
\frac{\partial^l}{\partial v^l}|_{v=0}
\left(
f(tv,v)
\right)
\right]
& = &
\frac{1}{l!}
\frac{\partial^l}{\partial v^l}|_{v=0}
\left[
\frac{v^s}{s!}
\left(
\frac{\partial^sf}{\partial z_1^s}
\right)
(\eta_qv,v)
\right]
\,.
\\
\end{eqnarray*}
Indeed, for all
$v$ close to $0$,
$(\eta_qv,v)\in\{z\in\mathbb{B}_2,\;z_1-\eta_qz_2=0\}$,
then we know all the
$\left(
\dfrac{\partial^sf}{\partial z_1^s}
\right)
(\eta_qv,v),\,
q=2,\ldots,n-1,\,s=0,\ldots,m_q-1$.
It follows by lemma \ref{serie} that,
for all $z\in \mathbb{B}_2$,
we know 
\begin{eqnarray*}
\frac{1}{s!}
\frac{\partial^s}{\partial t^s}|_{t=\eta_q}
\left[
\sum_{k_1+k_2\geq N_{u_p}}
a_{k_1,k_2}\,t^{k_1}
\left(\frac{z_2+|w|^2z_1/t}{1+|w|^2}\right)^{k_1+k_2-N_{u_p}}
\right]\,,
\end{eqnarray*}
as well as,
for all
$u_1=0,\ldots,m_1-1$,
\begin{eqnarray*}
\mathcal{L}
\left(
\eta_2^{m_2},\ldots,\eta_{n-1}^{m_{n-1}};
\sum_{k_1+k_2\geq N_{u_1}}
\frac{a_{k_1,k_2}t^{k_1-u_1-1}}{1+|w|^2}
\left(
\frac{z_2+|w|^2z_1/t}{1+|w|^2}
\right)^{k_1+k_2-N_{u_1}}
\right)(z_1/z_2)
\,,
\end{eqnarray*}
for all
$p=2,\ldots,n-1$ and $u_p=0,\ldots,m_p-1$,
\begin{eqnarray*}
\mathcal{L}
\left(
\eta_p^{u_p+1},\ldots,\eta_{n-1}^{m_{n-1}};
\frac{1+|w|^2\eta_p/t}{1+|w|^2}
\sum_{k_1+k_2\geq N_{u_p}}
a_{k_1,k_2}t^{k_1}
\left(
\frac{z_2+|w|^2z_1/t}{1+|w|^2}
\right)^{k_1+k_2-N_{u_p}}
\right)(z_1/z_2)
\end{eqnarray*}
and
\begin{eqnarray*}
\mathcal{L}
\left(
\eta_2^{m_2},\ldots,\eta_{n-1}^{m_{n-1}};
\sum_{k_1+k_2\geq N}
\frac{a_{k_1,k_2}\,t^{k_1-m_1}}{z_1/z_2-t}
\left(\frac{z_2+|w|^2z_1/t}{1+|w|^2}\right)^{k_1+k_2-N+1}
\right)
(z_1/z_2)
\,.
\\
\end{eqnarray*}

\bigskip

Similarly, for all
$u_1=0,\ldots,m_1-1$, we know
$\left(
\dfrac{\partial^{u_1}f}{\partial z_1^{u_1}}
\right)(0,z_2)$
then for all
$l\geq0$ we know
\begin{eqnarray*}
\frac{1}{l!}
\frac{\partial^l}{\partial z_2^l}|_{z_2=0}
\left[
\frac{1}{u_1!}
\left(
\frac{\partial^{u_1}f}{\partial z_1^{u_1}}
\right)(0,z_2)
\right]
\;=\;
\frac{1}{u_1!\,l!}
\left(
\frac{\partial^{u_1+l}f}{\partial z_1^{u_1}\partial z_2^l}
\right)(0)
\;=\;
a_{u_1,l}\,.
\\
\end{eqnarray*}
It follows that we know,
for all $u_1=0,\ldots,m_1-1$,
\begin{eqnarray*}
\mathcal{L}
\left(
\eta_2^{m_2},\ldots,\eta_{n-1}^{m_{n-1}};
\sum_{k_1\leq u_1,\,k_2\geq N_{u_1}-k_1}
a_{k_1,k_2}t^{k_1-u_1-1}
z_2^{k_1+k_2-N_{u_1}}
\right)(z_1/z_2)
\,,
\end{eqnarray*}
as well as
\begin{eqnarray*}
\mathcal{L}
\left(
\eta_2^{m_2},\ldots,\eta_{n-1}^{m_{n-1}};
\frac{1}{z_1/z_2-t}
\sum_{k_1\leq m_1-1,k_2\geq N-k_1}
a_{k_1,k_2}\,t^{k_1-m_1}z_2^{k_1+k_2}
\right)
(z_1/z_2)
\,.
\end{eqnarray*}

\bigskip

Lastly, for all
$u_n=0,\ldots,m_n-1$,
we know
$\left(
\dfrac{\partial^{u_n}f}{\partial z_2^{u_n}}
\right)(z_1,0)$,
as well as for all $l\geq0$
\begin{eqnarray*}
\frac{1}{l!}
\frac{\partial^l}{\partial z_1^l}|_{z_1=0}
\left[
\frac{1}{u_n!}
\left(
\frac{\partial^{u_n}f}{\partial z_2^{u_n}}
\right)(z_1,0)
\right]
& = &
a_{l,u_n}
\,.
\\
\end{eqnarray*}
It follows that we know, for all
$p=2,\ldots,n-1$ and $u_p=0,\ldots,m_p-1$,
\begin{eqnarray*}
\mathcal{L}
\left(
\eta_p^{u_p+1},\ldots,\eta_{n-1}^{m_{n-1}};
\frac{1+|w|^2\eta_p/t}{1+|w|^2}
\sum_{k_1+k_2\geq N_{u_p}}
a_{k_1,k_2}t^{k_1}
\left(
\frac{z_2+|w|^2z_1/t}{1+|w|^2}
\right)^{k_1+k_2-N_{u_p}}
\right)(z_1/z_2)
\,,
\end{eqnarray*}
as well as
\begin{eqnarray*}
\mathcal{L}
\left(
\eta_2^{m_2},\ldots,\eta_{n-1}^{m_{n-1}};
\frac{1}{z_1/z_2-t}
\sum_{k_2\leq m_n-1,k_1\geq N-k_2}
a_{k_1,k_2}\,t^{N-m_1-1-k_2}z_1^{k_1+k_2-N+1}
\right)
(z_1/z_2)
\\
\end{eqnarray*}
and the lemma is proved.

\end{proof}

\bigskip

One can specify
$\mathcal{G}\left(\eta_1^{m_1},\ldots,\eta_n^{m_n};f\right)$
in the special case with
$m_2=\cdots=m_{n-1}=1$ and 
$m_1=m_n=0$. Then
$N=n-2$, for all
$p=2,\ldots,n-1$,
$N_{u_p}=n-p-1=N-p+1$
and
\begin{eqnarray}\label{single}
\mathcal{G}(\eta_2,\ldots,\eta_{n-1};f)(z)
& = &
\;\;\;\;\;\;\;\;\;\;\;\;\;\;\;\;\;\;\;\;\;\;\;\;\;\;\;\;\;\;\;\;\;\;\;\;\;\;\;\;\;\;\;\;\;\;\;\;\;\;\;\;\;\;\;\;\;\;\;\;
\end{eqnarray}
\begin{eqnarray*}
& = &
\sum_{p=2}^{n-1}
\prod_{j=p+1}^{n-1}
(z_1-\eta_jz_2)
\sum_{q=p}^{n-1}
\frac{1+\eta_p\overline{\eta_q}}{1+|\eta_q|^2}
\frac{1}{\prod_{j=p,j\neq q}^{n-1}(\eta_q-\eta_j)}
\;\times
\\
& &
\times\;
\sum_{k_1+k_2\geq N-p+1}
a_{k_1,k_2}\eta_q^{k_1}
\left(
\frac{z_2+\overline{\eta_q}z_1}{1+|\eta_q|^2}
\right)^{k_1+k_2-(N-p+1)}
\\
& - &
\sum_{p=2}^{n-1}
\prod_{j=2,j\neq p}^{n-1}
\frac{z_1-\eta_jz_2}{\eta_p-\eta_j}
\sum_{k_1+k_2\geq N}
a_{k_1,k_2}\eta_p^{k_1}
\left(
\frac{z_2+\overline{\eta_p}z_1}{1+|\eta_p|^2}
\right)^{k_1+k_2-N+1}
\end{eqnarray*}
\begin{eqnarray*}
& = &
\sum_{p=2}^{n-1}
\prod_{j=p+1}^{n-1}
(z_1-\eta_jz_2)
\sum_{q=p}^{n-1}
\frac{1+\eta_p\overline{\eta_q}}{1+|\eta_q|^2}
\frac{1}{\prod_{j=p,j\neq q}^{n-1}(\eta_q-\eta_j)}
\;\times
\\
& &
\times\;
\sum_{l\geq N-p+1}
\left(
\frac{z_2+\overline{\eta_q}z_1}{1+|\eta_q|^2}
\right)^{l-(N-p+1)}
\frac{1}{l!}
\frac{\partial^l}{\partial v^l}|_{v=0}
[f(\eta_qv,v)]
\\
& - &
\sum_{p=2}^{n-1}
\prod_{j=2,j\neq p}^{n-1}
\frac{z_1-\eta_jz_2}{\eta_p-\eta_j}
\sum_{l\geq N}
\left(
\frac{z_2+\overline{\eta_p}z_1}{1+|\eta_p|^2}
\right)^{l-N+1}
\frac{1}{l!}
\frac{\partial^l}{\partial v^l}|_{v=0}
[f(\eta_pv,v)]
\,.
\\
\end{eqnarray*}

\bigskip

We finish with the last result where
we specify the precision for
the approximation of $f$ by
$\mathcal{G}\left(\eta_1^{m_1},\ldots,\eta_n^{m_n};f\right)$
when
$N\rightarrow+\infty$.

\begin{corollary}\label{vitessecv}

For all compact subset
$K\subset\mathbb{B}_2$, we have
\begin{eqnarray}
\sup_{z\in K}
\left|
f(z)-\mathcal{G}
\left(
\eta_1^{m_1},\ldots,\eta_n^{m_n};f
\right)(z)
\right|
& \leq &
C(K,f)\left(1-\varepsilon_K\right)^N
\;,
\\\nonumber
\end{eqnarray}
where
$C(K,f)=C_K\,\sup_{\zeta\in K'}|f(\zeta)|$, 
$K'\supset K$
and $\varepsilon_K,\,C_K$
depend on $K$.
\medskip

In particular, if
$\mathcal{F}\subset\mathcal{O}(\mathbb{B}_2)$
is a compact subset (i.e. a subset of 
holomorphic functions that is uniformly bounded
on all compact subset of
$\mathbb{B}_2$), then
\begin{eqnarray}
\sup_{f\in\mathcal{F}}
\,\sup_{z\in K}
\left|
f(z)-\mathcal{G}
\left(
\eta_1^{m_1},\ldots,\eta_n^{m_n};f
\right)(z)
\right|
& \leq &
C(K,\mathcal{F})(1-\varepsilon_K)^N
\,.
\\\nonumber
\end{eqnarray}

\end{corollary}

\bigskip

\begin{proof}

It follows from theorem \ref{theorem} that
\begin{eqnarray*}
f(z)-\mathcal{G}
\left(
\eta_1^{m_1},\ldots,\eta_n^{m_n};f
\right)(z)
& = &
\sum_{k_1+k_2\geq N,k_1\geq m_1,k_2\geq m_n}
a_{k_1,k_2}z_1^{k_1}z_2^{k_2}
\,.
\\
\end{eqnarray*}
On the other hand, we know by lemma \ref{cvtaylor}
that the Taylor expansion of $f$ is absolutely convergent in $K$.
More precisely, $K$ being covered by a finite number of
bidiscs 
$D(0,r_j),\;j=1,\ldots,J$, one can choose 
$D(0,r'_j)\supset \overline{D}(0,r_j),
\;j=1,\ldots,J$, such that
(see (\ref{coeff}) in the proof of lemma \ref{cvtaylor})
\begin{eqnarray*}
\sup_{z\in K}
\left|
a_{k_1,k_2}z_1^{k_1}z_2^{k_2}
\right|
& \leq &
\max_{1\leq j\leq J}
\sup_{z\in D(0,r_j)}
\left|
a_{k_1,k_2}z_1^{k_1}z_2^{k_2}
\right|
\\
& \leq &
\max_{1\leq j\leq J}
\left[
\sup_{\zeta\in D(0,r'_j)}|f(\zeta)|
\left(\frac{r_{j,1}}{r'_{j,1}}\right)^{k_1}
\left(\frac{r_{j,2}}{r'_{j,2}}\right)^{k_2}
\right]
\\
& \leq &
\sup_{\zeta\in K'}|f(\zeta)|
\;(1-\varepsilon_K)^{k_1+k_2}
\\
\end{eqnarray*}
(where $K'=\bigcup_{j=1}^J\overline{D}(0,r'_j)$ ).
Then
\begin{eqnarray*}
\sup_{z\in K}
\left|
\sum_{k_1+k_2\geq N,k_1\geq m_1,k_2\geq m_n}
a_{k_1,k_2}z_1^{k_1}z_2^{k_2}
\right|
& \leq &
\sup_{\zeta\in K'}|f(\zeta)|
\sum_{k_1+k_2\geq N}
(1-\varepsilon_K)^{k_1+k_2}
\\
& = &
\sup_{\zeta\in K'}|f(\zeta)|
\sum_{q\geq 0}
(q+N+1)(1-\varepsilon_K)^{q+N}
\\
& \leq &
C_K\,
\sup_{\zeta\in K'}|f(\zeta)|
\left(\sqrt{1-\varepsilon_K}\right)^N
\,,
\\
\end{eqnarray*}
and the corollary is proved by choosing 
$\varepsilon'_K<\varepsilon_K$.

\end{proof}

\end{document}